\documentclass[11pt,oneside,reqno]{amsart}
\usepackage{amssymb}
\usepackage{empheq}
\usepackage{fullpage}
\usepackage[colorlinks,allcolors=blue]{hyperref}
\usepackage{mathtools}
\usepackage{aligned-overset}
\usepackage{hyperref}
\usepackage{amsmath}
\usepackage{tikz}
\numberwithin{equation}{section}
\newtheorem{theorem}{Theorem}[section]
\newtheorem{corollary}[theorem]{Corollary}
\newtheorem{lemma}[theorem]{Lemma}
\newtheorem{proposition}[theorem]{Proposition}
\theoremstyle{definition}
\newtheorem{definition}[theorem]{Definition}

\theoremstyle{remark}
\newtheorem{remark}[theorem]{Remark}
\newcommand{\nrm}[1]{\left\Vert#1 \right\Vert}
\newcommand{\abs}[1]{\left\vert #1\right\vert}
\newcommand{\brk}[1]{\left\langle#1\right\rangle}
\newcommand{\set}[1]{\left\{#1\right\}}
\newcommand{\tld}[1]{\widetilde{#1}}
\newcommand{\br}[1]{\overline{#1}}

\newcommand{\wht}[1]{\widehat{#1}}	
\newcommand{\dist}{\mathrm{dist}}
\newcommand{\sgn}{{\mathrm{sgn}}}

\newcommand{\supp}{\mathrm{supp}}

\newcommand{\rd}{\partial}

\newcommand{\at}[2][]{#1\big|_{#2}}
\newcommand{\0}{\emptyset}

\newcommand{\bbR}{\mathbb R}
\newcommand{\bbC}{\mathbb C}
\newcommand{\bbT}{\mathbb T}
\newcommand{\bbN}{\mathbb N}

\newcommand{\bbZ}{\mathbb Z}
\newcommand{\calA}{\mathcal A}
\newcommand{\calS}{\mathcal S}
\newcommand{\calF}{\mathcal F}

\newcommand{\dbeta}{\overline{\partial}_\beta \overline{\psi}}
\newcommand{\dvarphi}{\overline{\partial}_\varphi \overline{\psi}}
\newcommand{\dphi}{\overline{\partial}_\phi \overline{\psi}}
\newcommand{\dbetaphi}{\partial_\phi \overline{\partial}_\beta \overline{\psi}}
\newcommand{\dbetavarphi}{(\overline{\partial}_\varphi +1) \overline{\partial}_\beta \overline{\psi}}
\newcommand{\dbbeta}{\overline{\partial}_\beta}
\newcommand{\dbvarphi}{\overline{\partial}_\varphi}
\newcommand{\dbvarphibeta}{(\overline{\partial}_\varphi+1)\overline{\partial}_\beta}
\newcommand{\dbphibeta}{\partial_\phi \overline{\partial}_\beta}
\newcommand{\sigmaz}{\sum_{n \in \mathbb{Z}}}
\newcommand{\n}[1]{#1^{(n)}}
\newcommand{\ui}[1]{#1^{(i)}}
\newcommand{\uj}[1]{#1^{(j)}}
\newcommand{\uij}[1]{#1^{(i+j)}}
\newcommand{\maxi}[1]{\textnormal{max} \left\{ #1 \right\}}
\newcommand{\mini}[1]{\text{min} \left\{ #1 \right\}}
\newcommand{\unibed}{\overset{\sim}{\hookrightarrow}}
\newcommand{\dns}[2]{D^{\left( #1, #2 \right)}}
\newcommand{\dnsi}[2]{\left(D^{\left( #1, #2 \right)}\right)^{-1}}
\newcommand{\Pnm}{P^{(n)}_-}
\newcommand{\Pnp}{P^{(n)}_+}
\newcommand{\Qinv}{\left( Q+1 \right)^{-1}}
\newcommand{\Pnminv}{\left( P^{(n)}_- \right)^{-1}}
\newcommand{\Pnpinv}{\left( P^{(n)}_+ \right)^{-1}}

\newcommand{\Wnm}{W^{(n)}_{-,N}}
\newcommand{\Wnz}{W^{(n)}_{0,N}}
\newcommand{\Wnp}{W^{(n)}_{+,N}}
\newcommand{\Xnz}{X^{(n)}_{0,N}}
\newcommand{\Znz}{Z^{(n)}_{0,N}}
\newcommand{\Wm}{W_{-,N}}
\newcommand{\Wz}{W_{0,N}}
\newcommand{\Wp}{W_{+,N}}
\newcommand{\Xz}{X_{0,N}}
\newcommand{\Zz}{Z_{0,N}}
\newcommand{\Yz}{Y_{0,N}}

\newcommand{\distPnm}{\text{dist}([-\delta, \delta], (2-n)\mu -1)}
\newcommand{\distPnp}{\text{dist}([-\delta, \delta], (2+n)\mu -1)}
\newcommand{\distPnpm}{\text{dist}([-\delta, \delta], (2 \pm n)\mu -1)}
\newcommand{\ncbd}[1]{\left\Vert#1 \right\Vert_{C_b^\delta}}
\newcommand{\ncb}[1]{\left\Vert#1 \right\Vert_{C_b}}
\newcommand{\cbd}{C_b^\delta}

\newcommand{\cbdn}{C_{b, N}^\delta}
\newcommand{\cxizn}{\mathbb{C} \xi_{0, N}}
\newcommand{\cxiin}{\mathbb{C} \xi_{\infty, N}}
\newcommand{\cxiz}{\mathbb{C} \xi_0}
\newcommand{\cxii}{\mathbb{C} \xi_\infty}
\newcommand{\cn}{\mathbb{C}^{(n)}_N}
\newcommand{\nmt}{\frac{1}{(N-2)\mu +\frac{5}{6}}}
\newcommand{\nmtn}{\frac{\left\langle N \right\rangle }{(N-2)\mu +\frac{5}{6}}}
\newcommand{\RT}{\overline{\mathbb{R}}_+ \times \mathbb{T}}
\newcommand{\ball}[3]{B_{#1}\left( #2, #3 \right)}
\newcommand{\trisol}{\overline{\psi}_0}

\newcommand{\res}{\Bigg\vert_{(\beta, \phi)=(0, \theta)}}
\newcommand{\epdot}{\dot{\varepsilon}}
\newcommand{\rmz}{\mathbb{R}^2 \backslash \{ 0 \}}
\newcommand*\circled[1]{\tikz[baseline=(char.base)]{
            \node[shape=circle,draw,inner sep=0.7pt] (char) {#1};}}

\makeatletter
\@namedef{subjclassname@2020}{\textup{2020} Mathematics Subject Classification}
\makeatother

\title{A solution of 2D incompressible Euler equation with algebraic spiral roll-up in the presence of Wiener type perturbation}
\author{Woohyu Jeon}
\address{Department of Mathematical Sciences, Seoul National University, Seoul 08826, Republic of Korea}
\email{woohyu1030@snu.ac.kr}


\begin{document}
\begin{abstract}
	Extending the results of Elling \cite{Elling-2013, Elling-2016}, we construct a weak solution of 2D incompressible Euler equation with initial vorticity of the form $w_0(x)={\left\vert x \right\vert}^{-1/\mu}g(\theta)$, where $g \in L^p(\mathbb{T})$ satisfies $\sum_{\mathbb{Z}}{\left\vert n \right\vert}^{-0.5} {\left\vert\widehat{g}(n)\right\vert} < \infty$. In particular, the solution is self-similar and shows algebraic spiral roll-up.
\end{abstract}
\maketitle

\tableofcontents

\newpage
\section{Introduction}
\label{sec: Introduction}

\subsection{History}
\label{subsec: History}

One of the most interesting phenomenon which has captured our interest is the flow of non-viscous fluid, e.g., ideal gas. In $\bbR^2$, its dynamics can be modeled through the well-known incompressible Euler equation

\begin{equation} \label{Euler eq} \tag{EE}
	\left\{\begin{aligned}
		& u_t+ \nabla_x \cdot (u \otimes u) + \nabla_x p =0, \\
		& \nabla_x \cdot u =0, 
	\end{aligned}\right.
\end{equation}
or through its vorticity form
\begin{equation} \label{Euler eq-vor} \tag{EE-v}
	\left\{\begin{aligned}
		& w_t+\nabla_x \cdot (wu)=0, \\
		& u=\nabla_x^\perp \psi ,\quad w= \Delta_x \psi.
	\end{aligned}\right.
\end{equation}
Here, $\nabla_x=(\rd_{x_1}, \rd_{x_2}), \nabla_x^\perp=(-\rd_{x_2}, \rd_{x_1})$\footnote{We use subscript below nabla to clarify which coordinate system we are considering.} and $u: \bbR^2 \times [0, \infty) \rightarrow \bbR^2$ denotes a velocity field, $p: \bbR^2 \times [0, \infty) \rightarrow \bbR$ denotes a pressure, $w:\bbR^2 \times [0, \infty) \rightarrow \bbR$ is a scalar vorticity, and $\psi: \bbR^2 \times [0, \infty) \rightarrow \bbR$ is a stream function.

%A various aspect of the incompressible Euler equation has been extensibly dealt with due to its importance, such as local well-posedness ($C^{k,\alpha}$ well-posdedness in Lichtenstein \cite{Lichtenstein}, $H^s$ well-posedness in Kato \cite{Kato-HsWP}, $W^{k,p}$ well-posedness in Kato and Ponce \cite{Kato-Commutator}), global wellposedness issue ($C^{k, \alpha}$ global well-posedness in 2D in \cite{Wolibner}, $L^p$ global well-posedness 2D in Yudovich \cite{Yudovich}, \cite{Yudovich2}, $C^{1,\alpha}$ blow-up in 3D in Elgindi \cite{Elgindi-C1aIP}), ill-posedness ($C^k$ ill-posdeness in Bourgain and Li \cite{Bourgain-CkIP}, $H^s$ ill-posedness in Bourgain and Li \cite{Bourgain-HsIP}), blow-up criteria (Beale-Kato-Majda criterion in Beale et al. \cite{BKM}, Constantin-Fefferman-Majda criterion in Constantin et al. \cite{CFM}), regularity issue of vortex patch ($C^\infty$ global well-posdedness in Chemin \cite{Chemin-C1aWP}, $C^2$ ill-posedness in \cite{Kiselev-C2IP})

Due to its importance, a various aspect of the incompressible Euler equation has been extensibly dealt with, such as local well-posedness \cite{Lichtenstein, Kato-HsWP, Kato-Commutator}, ill-posedness \cite{Bourgain-CkIP, Bourgain-HsIP}, global well-posedness issue \cite{Wolibner, Yudovich, Yudovich2, Elgindi-C1aIP}, blow-up criteria \cite{BKM, CFM}, regularity issue of vortex patch \cite{Chemin-C1aWP, MB, Kiselev-C2IP}, and bifurcation of rotating vortex patch \cite{Burbea, Hmidi, Hassainia}.

In this paper, we focus on the solution with spiral roll-up structure. The flow with spiral roll-up draw a lot of attention due to its ubiquity in nature; flow near the wings of an airplane (Fig 81, 91 in \cite{Vandyke}), flow after it pass through the solid corner (Fig 80, 95, 98 in \cite{Vandyke}), boundary layer separation (Fig 32 in \cite{Vandyke}), etc. Therefore, there has been a lot of attempts to understand spiral solution. For example, the 2D logarithmic spiral ($r \sim e^{k\theta}$) vortex sheet was introduced by Prandtl \cite{Prandtl} in 1924 and Alexander \cite{Alexander} generalized it. The well-posedness of self-similar logarithmic spiral vortex sheet remains open for a century since Prandtl, but very recently, Cie\'slak et al. \cite{Cieslac} gives reasonable sufficient condition for well-posedness of logarithmic spiral. The case of algebraic spiral ($r \sim \theta^{-\mu}$) has been studied also. Kaden \cite{Kaden} asserted in 1933 that in the initially straight shear layer, a vortex sheet with algebraic spiral ($r \sim \left( \frac{t}{\theta} \right)^{\frac{2}{3}}$) structure emerges. Since then, several works had tried to describe this algebraic spiral more precisely, trying to find higher expansion term near the spiral center \cite{Stern, Mangler}. However, Moore \cite{Moore} asserted that improvement of Kaden spiral could not be done by considering local behavior near the center, but could only be done by analyzing whole vortex sheet with Birkhoff's integro-differential equation \cite{Birkhoff}. Although there has been a lot of effort to prove the existence of algebraic spiral vortex sheet rigorously, it is still an open problem until now.

In the mean time, a pioneering progress was made by a celebrated result in Elling \cite{Elling-2013}. Although the solution he constructed was differentiable solution of 2D Euler equation, not a vortex sheet, he proved rigorously the existence of Euler equation which shows some algebraic spiral structure (More precise definition of algebraic spiral structure is provided later). To state the main theorem of Elling \cite{Elling-2013} precisely, we first clarify some definition and concept.

\subsection{Some basic definition and related works}
\label{subsec: Some basic definition and related works}

We use standard notations; for vector fields
\begin{equation*}
	v= \begin{pmatrix}
		v^1 \\
		v^2\\
	\end{pmatrix}, \qquad w= \begin{pmatrix}
		w^1 \\
		w^2 \\
	\end{pmatrix}
\end{equation*}
and $2 \times 2$ matrices
\begin{equation*}
	A=(a_{ij})_{1 \leq i,j \leq 2}, \qquad B=(b_{ij})_{1 \leq i,j \leq 2},
\end{equation*}
\begin{equation*}
	v \otimes w \coloneqq vw^T= \begin{pmatrix}
		v^1w^1 & v^1w^2 \\
		v^2w^1 & v^2 w^2 \\
	\end{pmatrix}, \qquad \nabla_x v \coloneqq  \begin{pmatrix}
		\rd_{x_1}v^1 & \rd_{x_2}v^1 \\
		\rd_{x_1}v^2 & \rd_{x_2}v^2
	\end{pmatrix}, \qquad A:B \coloneqq \sum_{1 \leq i,j \leq 2} a_{ij}b_{ij}.
\end{equation*}

\begin{definition} \label{weak solution for u}
	We say $u \in L^\infty \left( [0,\infty); L^2_{\textnormal{loc}}(\bbR^2 ; \bbR^2) \right)$ is a weak solution of \eqref{Euler eq} if
	\begin{enumerate}
		\item The velocity field $u$ is weakly divergence free, i.e., for any (time-independent) test function $f \in C^\infty_c(\bbR^2)$ and $t \geq 0$, 
		\begin{equation*}
			\int_{\bbR^2} u(\cdot, t) \cdot \nabla_x f \, dx=0.
		\end{equation*}
		\item For all divergence free test function $h \in C^\infty_c (\bbR^2 \times [0,\infty) : \bbR^2),$ \begin{equation*}
			\int_{\bbR^2} u(\cdot,0) \cdot h(\cdot, 0) \, dx+ \int_0^\infty \int_{\bbR^2} u \cdot g_t + u \otimes u : \nabla_xh \, dxdt =0.
		\end{equation*}
	\end{enumerate}
\end{definition}

\begin{definition} \label{weak solution for w}
	We say $w \in L^\infty \left( [0, \infty); L^2_{\textnormal{loc}}(\bbR^2) \right)$\footnote{We require $w(\cdot, t)$ belongs to $L^2_{\textnormal{loc}}(\bbR^2; \bbR^2)$ instead of $L^1_{\textnormal{loc}}(\bbR^2; \bbR^2)$ to guarantee $wu \in L^\infty([0,\infty);L^1_{\textnormal{loc}}(\bbR^2 ; \bbR^2)$.} is a weak solution of \eqref{Euler eq-vor} if
	\begin{enumerate}
		\item There exist $\psi \in L^\infty([0,\infty); W^{1,1}_{\textnormal{loc}}(\bbR^2))$ and $u \in L^\infty([0,\infty); L^2_{\textnormal{loc}}(\bbR^2; \bbR^2))$ such that $\nabla_x^\perp \psi =u,\, \Delta_x \psi=w$ in distribution sense, i.e., for any test function $f \in C^\infty_c (\bbR^2), g \in C^\infty_c(\bbR^2 ; \bbR^2)$,
		\begin{equation*}
			-\int_{\bbR^2} \psi (\nabla_x^\perp \cdot g) \, dx = \int_{\bbR^2} u \cdot g \, dx, \qquad-\int_{\bbR^2} \nabla_x \psi \cdot \nabla_x f = \int_{\bbR^2} wf.
		\end{equation*}
		\item For any test function $h \in C^\infty_c (\bbR^2 \times [0, \infty))$, \begin{equation*}
			\int_0^\infty w(\cdot, 0) h(\cdot, 0) \, dx+ \int_0^\infty \int_{\bbR^2} wh_t+w(u \cdot \nabla h) \, dxdt =0.
		\end{equation*}
	\end{enumerate}
\end{definition}

\begin{remark}
	We will actually find (classically) differentiable $\psi \in L^\infty([0, \infty); C^1(\bbR))$ such that $\Delta_x \psi=w$ in distribution sense, i.e., for any $f \in C^\infty_c(\bbR^2)$ and $t\geq 0$,
	\begin{equation*}
		-\int_{\bbR^2} \nabla_x \psi (\cdot, t) \cdot \nabla_x f \,dx = \int_{\bbR^2} w(\cdot, t) f \, dx,
	\end{equation*}
	and then define $u$ by $\nabla_x^\perp \psi$. Then, it is trivial that $w, u$ satisfies the first condition in Definition \ref{weak solution for w}.
\end{remark}

Now, we define self-similarity. By scaling property of Euler equation, if $w(x,t), u(x,t), \psi(x,t)$ is a solution for \eqref{Euler eq} or \eqref{Euler eq-vor},
\begin{equation} \label{scaled func}
	w_\lambda(x,t) \coloneqq \frac{1}{\lambda^a}w \left(\frac{x}{\lambda},\frac{t}{\lambda^a}\right), \quad u_\lambda(x,t) \coloneqq \frac{1}{\lambda^{a-1}}u\left(\frac{x}{\lambda},\frac{t}{\lambda^a}\right), \quad \psi_\lambda(x,t) \coloneqq \frac{1}{\lambda^{a-2}}\psi \left(\frac{x}{\lambda},\frac{t}{\lambda^a} \right)
\end{equation}
is also a solution of it for all $\lambda>0$.

\begin{definition} \label{self-similarity}
	A solution $(w, u, \psi)$ of \eqref{Euler eq} or \eqref{Euler eq-vor} is called \textit{self-similar} if
	\begin{equation*}
		w_\lambda(x,t)=w(x,t), \qquad u_\lambda(x,t)=u(x,t), \qquad \psi_\lambda(x,t)=\psi(x,t)
	\end{equation*}
	for all $\lambda>0$.
\end{definition}

Lastly, for $f \in L^1(\bbT)$ where $\bbT \coloneqq \bbR / 2 \pi \bbZ$, we denote $n$-th Fourier coefficient of $f$ by $\wht{f}(n) = \frac{1}{2 \pi} \int_{\bbT} f(x)e^{-inx} \,dx $.

\begin{definition}
	For $s \in \bbR$,
	\begin{equation*}
		\calA^s (\bbT) \coloneqq \set{f \in L^1(\bbT) \, \Bigg\vert \, \nrm{f}_{\calA^s (\bbT)} \coloneqq \sum_{n \in \bbZ} \brk{n}^s \abs{\wht{f}(n)}< \infty},
	\end{equation*}
	where $\brk{n} \coloneqq (1+n^2)^{1/2}$.
\end{definition}

Now, we can state main theorem in Elling \cite{Elling-2013}.

\begin{theorem} [Elling \cite{Elling-2013}] \label{thm elling 2013} 
	Let $\mu \in (\frac{2}{3},\infty)$ be given. Then, there exists sufficiently large $N_{\mu} \in \bbN$ such that if $g \in \calA^0(\bbT)$ satisfies \textbf{`certain conditions'}, a weak solution of \eqref{Euler eq} exists with initial vorticity $w_0(x)= \abs{x}^{-\frac{1}{\mu}}g(\theta)$. In particular, the solution is self-similar and any integral curve of $u-\mu z$ is an algebraic spiral.
	
	The \textbf{certain conditions} are
	\begin{enumerate}
		\item High periodicity: $g(\theta)=g\left( \theta+\frac{2 \pi}{N} \right)$ for $N \geq N_\mu$.
		\item Small $\calA^0(\bbT)$ norm perturbation around constant: There exists a constant $c$ such that
		\begin{equation*}
			\nrm{g(\theta)-c}_{\calA^0 (\bbT)} < \varepsilon_N \abs{c}
		\end{equation*}
		for some $\varepsilon_N$ determined by $N$.
	\end{enumerate}
\end{theorem}

Due to the embedding $\calA^0 (\bbT) \hookrightarrow C^0(\bbT)$, Theorem \ref{thm elling 2013} requires that $g(\theta)$ be close to a constant function in $L^\infty$-norm. However, as been mentioned in Section \ref{subsec: History}, natural phenomenon have lead us to construct vortex sheet and this occurs when  $g(\theta)$ is a dirac delta distribution. Unfortunately, the second condition in Theorem \ref{thm elling 2013} makes $g(\theta)$ far away from dirac delta. Meanwhile, in 2016, Elling \cite{Elling-2016} took one step forward and develop his result in \cite{Elling-2013}. To state main result in \cite{Elling-2016}, we first clarify what it means by \textit{spiral roll-up} in this paper.

\begin{definition} \label{spiral roll-up}
	Let $w(x,t)$ be a weak solution of \eqref{Euler eq-vor} with initial vorticity \begin{equation*}
		w(x,0)=\abs{x}^{-\frac{1}{\mu}}g(\theta).
		\end{equation*}
		Define $A \coloneqq \set{\theta \in \bbT \, \vert \, g(\theta) =0}$. Then $w(x,t)$ is said to show \textit{spiral roll-up} if there exists a continuous function $h: (0,\infty) \times \bbT \rightarrow [a,b]$, where $0<a<b<\infty$, such that
		\begin{equation} \label{eq90}
			m\Bigg( \set{x \in \rmz \, \vert \, w(x,t)=0} \Delta \bigcup_{\theta \in A} \bigg\{ h(\beta, \theta) \underbrace{\left( \frac{t}{\beta} \right)^\mu (\cos (\beta + \theta), \sin (\beta+ \theta) )}_{\text{An algebraic spiral for all }\theta \in A,\, t>0 } \, \bigg\vert \beta \in (0, \infty) \bigg\}  \Bigg)=0
		\end{equation}
		for all $t>0$. Here, $m(A \Delta B)=m(A-B) + m(B-A)$ where $m$ denotes Lebesgue measure on $\bbR^2$. If $g\in C^0 (\bbT)$, we require that 
		\begin{equation*}
			\set{x \in \rmz \, \vert \, w(x,t)=0}= \bigcup_{\theta \in A} \bigg\{ h(\beta, \theta) \left( \frac{t}{\beta} \right)^\mu (\cos (\beta + \theta), \sin (\beta+ \theta) ) \, \bigg\vert \beta \in (0, \infty) \bigg\}
		\end{equation*}
		holds set theoretically.
\end{definition}

Then, the result in Elling \cite{Elling-2016} could be stated as follows.

\begin{theorem} [Elling \cite{Elling-2016}] \label{thm Elling 2016}
	Let $\mu \in \left( \frac{2}{3}, \infty \right)$ and $\varepsilon>0$ be given. Then, there exists sufficiently large $N_{\mu, \varepsilon} \in \bbN$ such that if $g \in \calA^0(\bbT) $ satisfies \textbf{`certain conditions'}, a weak solution of \eqref{Euler eq} exists with initial vorticity $w_0(x)=\abs{x}^{-\frac{1}{\mu}} g(\theta)$. In particular, the solution is self-similar and shows spiral roll-up in the sense of Definition \ref{spiral roll-up}.\footnote{Elling \cite{Elling-2016} shows that $w(x,t)$ is a weak solution of \eqref{Euler eq-vor} on $\rmz$, not on whole $\bbR^2$. }
	
	The \textbf{certain conditions} are
	\begin{enumerate}
		\item High periodicity: $g(\theta)=g\left( \theta+\frac{2 \pi}{N} \right)$ for $N \geq N_{\mu,\varepsilon}$.
		\item Small $\calA^0(\bbT)$ semi-norm compared to $\wht{g}(0)$: 
		\begin{equation*}
			\varepsilon \sum_{n \in \bbZ \backslash \set{0}} \abs{\wht{g}(n)} < \abs{\wht{g}(0)} 
		\end{equation*}
	\end{enumerate}
\end{theorem}

Compared to \cite{Elling-2013}, \cite{Elling-2016} embraces a larger class of initial data due to freedom in $\varepsilon$. Especially, as emphasized in the title, \cite{Elling-2016} admits initial data with mixed sign, e.g., $g(\theta)=1+2 \sin \left( \frac{2 \pi}{N}\right)$ for some large $N$. However, due to the embedding $\calA^0(\bbT) \hookrightarrow C^0(\bbT)$, $g(\theta)$ should be a continuous function even though assumption of $L^\infty$-small perturbation from constant in \cite{Elling-2013} is eliminated.

Very recently, Shao et al. \cite{Shao} made an remarkable progress after \cite{Elling-2016}. The main result of \cite{Shao} is:

\begin{theorem} [Shao et al. \cite{Shao}] \label{thm Shao}
	Let $\mu \in \left( \frac{1}{2}, \infty \right), N \in \bbN \backslash \set{1}$ is given. Then, there exists $\varepsilon>0$, which is independent of $N$, such that if $g \in L^1(\bbT)$ satisfies \textbf{`certain conditions'}, a weak solution of \eqref{Euler eq} exists with initial vorticity $w_0(x) = \abs{x}^{-\frac{1}{\mu}} g(\theta)$. In particular, the solution is self-similar.
	
	The \textbf{certain conditions} are
	\begin{enumerate}
		\item Periodicity: $g(\theta) = g \left( \theta + \frac{2 \pi}{N}\right)$
		\item Small $L^1$-seminorm compared to $\wht{g}(0)$:
		\begin{equation*}
			\nrm{g(\theta) - \wht{g}(0)}_{L^1(\bbT)} < \varepsilon \sqrt{N}\abs{\wht{g}(0)}.
		\end{equation*}
	\end{enumerate}
\end{theorem}

Theorem \ref{thm Shao} is awesome, because they break an assumption that $g(\theta)$ is continuous. Also, they eliminated the assumption of high periodicity and relieve the condition on $\mu$ from $\left( \frac{2}{3}, \infty \right)$ to $\left( \frac{1}{2}, \infty \right)$, making our distance closer to the goal of constructing spiral vortex sheet.

Now, we state our main theorem, which is located in the middle of Elling \cite{Elling-2016} and Shao et al. \cite{Shao}.

\begin{theorem} \label{main theorem}
	Let $\mu > \frac{2}{3}, \,1\leq p < 2\mu$ be given.\footnote{As we will see later, $\mu>\frac{2}{3}$ condition is required to show $u$ is a weak solution of \eqref{Euler eq}. Also, $1 \leq p < 2\mu$ condition is need to show $w(\cdot,t) \in L^p_{\textnormal{loc}}(\bbR^2)$.  } For all
	\begin{equation} \label{ran of N}
		N > (2\mu -1)\left(397+\frac{1090}{\mu} + \frac{1264}{\mu^2}+\frac{999}{\mu^3}+\frac{42}{\mu^4} \right),
	\end{equation}	
	 there exists a constant $C=C(N)$ such that for all $g \in \calA^{-0.5}(\bbT) \cap L^p (\bbT) $ satisfying
	\begin{enumerate}
		\item Periodicity: $g(\theta)=g\left(\theta+ \frac{2 \pi}{N} \right)$,
		\item Small $\calA^{-0.5}$-seminorm relative to $\wht{g}(0)$: $\sum_{n \in \bbZ \backslash \set{0}} \brk{n}^{-0.5} \abs{\wht{g}(n)} < C_N \abs{\wht{g}(0)}$,
	\end{enumerate}
	
	there exist
	\begin{equation} \label{eq103}
		w \in L^\infty([0,\infty); L^p_{\textnormal{loc}}(\bbR^2)), \quad u \in \begin{cases}
			L^\infty([0,\infty); L^{\frac{2-\mu}{1-\mu}}_{\textnormal{loc}}(\bbR^2; \bbR^2)) & \left( \frac{2}{3} < \mu <1 \right)\\
			L^\infty([0,\infty); L^\infty_{\textnormal{loc}}(\bbR^2; \bbR^2)) & (\mu \geq 1)\\
		\end{cases},\quad  \psi \in L^\infty([0,\infty); C^1(\bbR^2 \backslash \set{0}))
	\end{equation}
	which satisfy
	\begin{enumerate}
		\item [\circled{1}] $u$ is a weak solution of \eqref{Euler eq} with with initial velocity given by \eqref{initial data for velocity prop}.
		\item [\circled{2}] If $\mu>1$, $w$ is a weak solution of \eqref{Euler eq-vor} with initial vorticity $w_0(x)=\abs{x}^{-\frac{1}{\mu}}g(\theta)$.
		\item [\circled{3}] This solution is self-similar and shows spiral roll-up in the sense of \eqref{spiral roll-up}.
	\end{enumerate}
\end{theorem}

We briefly explain some remarks on Theorem \ref{main theorem} compared to previous results by Elling \cite{Elling-2016} and Shao \cite{Shao}.

\begin{enumerate}
	\item As $g \in \calA^{-0.5}(\bbT) \cap L^p(\bbT)$, $g(\theta)$ needs not be continuous. Also, as $1<2 \times \frac{2}{3}$, the most general $p=1$ case  is included for all $\mu>\frac{2}{3}$.
	\bigskip
	\item In the Theorem \ref{thm Elling 2016} in \cite{Elling-2016}, although it is computable, how large $N$ should be was not clarified. This is because the author used $\lesssim$\footnote{$A \lesssim B$ means that there exists some universal constant $c>0$ such that $A \leq cB$.} instead of $\leq$ when he estimated operator norm. Unlike in \cite{Elling-2016}, we will bound operator norm via $\leq$ rather than $\lesssim$ to estimate exact lower bound of $N$ which needs to guarantee the existence of weak solution, hoping that $N$ could be chosen reasonably small. Unfortunately, lower bound of $N$ we get is larger than expected, with the order of $10^3$.
	\bigskip
	\item Our result implies Elling \cite{Elling-2016} since we use $\calA^{-0.5}$-seminorm. As we have \begin{equation*}
		\sum_{n \in N\bbZ \backslash \set{0}} \brk{n}^{-0.5}\abs{\wht{g}(n)} \leq \brk{N}^{-0.5} \sum_{n \in N \bbZ \backslash \set{0}} \abs{\wht{g}(n)},
	\end{equation*}
	taking $N$ large enough substitute the role of $\varepsilon$.
	\bigskip
	\item Shao et al. \cite{Shao} can cover larger class of data than ours. In our result, $g(\theta) \in \calA^{-0.5}(\bbT) \cap L^1(\bbT)$, while in \cite{Shao}, $g(\theta)$ just needs to be in $L^1(\bbT)$. Also, \cite{Shao} hardly require that $g(\theta)$ be periodic (just $N \neq 1$ is enough) and allow $\mu$ to be $\frac{1}{2}< \mu \leq \frac{2}{3}$ also. However, a big difference between \cite{Shao} and ours is in how to measure the norm of $g(\theta)-\wht{g}(0)$. If $g(\theta)$ is given as a trigonometric polynomial, it is more convenient to estimate $\calA^{-0.5}$-seminorm than $L^1$-norm. For example, consider $g(\theta)$ of the following form \begin{equation} \label{eq99}
		g(\theta) = \sum_{n \in (kN)\bbZ \, \cap \, [-M,M]} c_n e^{in \theta},
		\end{equation}
		where $N$ is chosen so that \eqref{ran of N} is satisfied. Since \eqref{eq99} is a finite sum, $g \in L^1(\bbT)$ condition is always satisfied regardless of $M$. To verify whether we can construct Euler solution with initial vorticity $w_0(x)=\abs{x}^{-\frac{1}{\mu}}g(\theta)$, we have to check
		\begin{equation} \label{eq100}
			\sum_{n \in ((kN)\bbZ \backslash \set{0}) \, \cap \, [-M,M]} \brk{n}^{-0.5} \abs{c_n} \leq C_N
		\end{equation}
		to apply Theorem \ref{main theorem}. Meanwhile, when we try to apply Theorem \ref{thm Shao}, we have to check
		\begin{equation} \label{eq101}
			\bigg\Vert\sum_{n \in ((kN)\bbZ \backslash \set{0}) \, \cap \, [-M,M]} c_n e^{in\theta}\bigg\Vert_{L^1(\bbT)} \leq \sqrt{kN} \varepsilon.
		\end{equation}
		Since
		\begin{equation*}
			\sum_{n \in ((kN)\bbZ \backslash \set{0}) \, \cap \, [-M,M]} \brk{n}^{-0.5} \abs{c_n} \quad \leq \quad  \brk{kN}^{-\varepsilon} \sum_{n \in ((kN)\bbZ \backslash \set{0}) \, \cap \, [-M,M]} \brk{n}^{-0.5+\varepsilon} \abs{c_n}
		\end{equation*}
		for arbitrary fixed $0<\varepsilon<0.5$, to verify \eqref{eq100}, it suffices to check
		\begin{equation} \label{eq102}
			\sum_{n \in ((kN)\bbZ \backslash \set{0}) \, \cap \, [-M,M]} \brk{n}^{-0.5+\varepsilon} \abs{c_n} \leq \brk{kN}^\varepsilon C_N.
		\end{equation}
		Therefore, no mater how small implicit $C_N$ and $\varepsilon$ be, we can make right hand side of \eqref{eq101} and \eqref{eq102} bigger than (say) 1 by taking $k$ sufficiently large. For such $k$, to meet \eqref{eq102}, we can choose any sequence $\set{c_n}_{n \in \bbZ}$ such that $\sum_{n \in ((kN)\bbZ \backslash \set{0}) \, \cap \, [-M,M]}  \brk{n}^{-0.5+\varepsilon} \abs{c_n}$ is convergence whose limit is less than 1 (For example, just $\abs{c_n}<\frac{1}{n}$ is a sufficient condition for $\varepsilon=0.1$). Then, even if we make $M$ arbitrarily large, we can guarantee \eqref{eq102} holds.
		
		However, for general $c_n$, there is no clever way to control
		\begin{equation*}
			\bigg\Vert\sum_{n \in ((kN)\bbZ \backslash \set{0}) \, \cap \, [-M,M]} c_n e^{in\theta}\bigg\Vert_{L^1(\bbT)}
		\end{equation*}
		except a coarse triangle inequality
		\begin{equation*}
			\bigg\Vert\sum_{n \in ((kN)\bbZ \backslash \set{0}) \, \cap \, [-M,M]} c_n e^{in\theta}\bigg\Vert_{L^1(\bbT)} \leq 2 \pi \sum_{n \in ((kN)\bbZ \backslash \set{0}) \, \cap \, [-M,M]} \abs{c_n},
		\end{equation*}
		which leaves the possibility of divergence as $M \rightarrow \infty$ even for $\abs{c_n} < \frac{1}{n}$.
		\bigskip
		\item We show that constructed solution \eqref{eq103} is a weak solution for \eqref{Euler eq-vor} with an additional assumption $\mu>1$, while \cite{Elling-2016} and \cite{Shao} only showed $u$ is a weak solution of \eqref{Euler eq}.
	\end{enumerate}

\subsection{Sketch of the proof and organization of the paper}
\label{subsec: Sketch of the proof and organization of the paper}

In Section \ref{sec: Self-similarity and coordinate change}, with a self-similarity assumption, we rewrite \eqref{Euler eq-vor} in special coordinate, which is called adapted coordinate by Elling \cite{Elling-2013, Elling-2016}, as below: 

\begin{align*}
	&\qquad \,\, \underbrace{\begin{cases}
		w_t+ \nabla_x^\perp \psi \cdot \nabla_x w=0, \\
		w=\Delta_x \psi.
	\end{cases}}_{\text{ In }(x,t) - \text{coordinate}} \\
	\underset{\text{Sec } \ref{subsec: Euler equation under self-similarity}}&{\Longrightarrow} \underbrace{\begin{cases}
		-\tld{w}+(\nabla_z^\perp \tld{\psi}-\mu z) \cdot \nabla_z \tld{w}=0, \\
		w=\Delta_z \psi.
	\end{cases}}_{\substack{\text{In } (z_1, z_2)-\text{coordinate, where }z=xt^{-\mu} \\ \text{(self-similarity assumption)}}} \\
	\underset{\text{Sec } \ref{subsec: Equations in new coordinate}}&{\Longrightarrow} \underbrace{\begin{cases}
		w(\beta, \phi)=(\psi_\phi(\beta, \phi)-\psi_\beta(\beta, \phi))^{-\frac{1}{2\mu}} \Omega(\phi) \quad \text{for some }\Omega(\phi)  ,\\
		L(\psi(\beta, \phi), \Omega(\phi))=0, \quad \text{where $L$ is defiend as in } \eqref{def of L}.
	\end{cases}}_{\text{In }(\beta, \phi)-\text{coordinate defined by }\eqref{new coordinate} \text{ defined on Sec } \ref{subsec: Special change of coordinate: From spiral to line}}  \\
	\underset{\text{Sec } \ref{subsec: Rescaling to eliminate decay}}&{\Longrightarrow} \begin{cases}
		w(\beta, \phi)=(\psi_\phi(\beta, \phi)-\psi_\beta(\beta, \phi))^{-\frac{1}{2\mu}} \Omega(\phi) \quad \text{for some }\Omega(\phi),  \\
		\br{L}(\br{\psi}(\beta, \phi), \Omega(\phi))=0, \quad \text{where }	\br{L}(\br{\psi}, \Omega) = \beta^{2\mu} L(\psi, \Omega) \text{ with } \br{\psi}(\beta, \phi)=\frac{1}{\beta^{1-2\mu}}\psi(\beta, \phi). \end{cases}   \\
\end{align*}
	
	In Section \ref{sec: Main theorem in new coordinate}, We solve $\br{L}(\br{\psi}(\beta, \phi), \Omega(\phi))=0$. More precisely, we show that for given $\Omega(\phi)$, there exists $\br{\psi}=\br{\psi}^{(\Omega)}$ such that $\br{L}(\br{\psi}(\beta, \phi), \Omega(\phi))=0$. As you already have noticed, this is a tempting situation where we can use implicit function theorem. To apply it, we have to answer for these three questions:

\bigskip
\textbf{$\bullet$ Question 1:} \label{Q1} What is the Banach space $X, Y$ and  $Z$ that will be a domain and codomain for $\br{L}: X \times Y \rightarrow Z?$ In other words, where does $\br{\psi}, \Omega, \br{L}(\br{\psi},\Omega)$ lives?

\textbf{$\bullet$ Question 2:} \label{Q2} (When we denote a trivial solution of $\br{L}$ by $(\br{\psi}_0,\Omega_0)$) Is the map $\bar{L}$ $C^1$ in some neighborhood of $(\br{\psi}_0, \Omega_0)$?

\textbf{$\bullet$ Question 3:} \label{Q3} Is the partial Fr\'{e}chet derivative $\frac{\rd \br{L}}{\rd \br{\psi}}(\br{\psi}_0, \Omega_0)$ isomorphism?
\bigskip

After some review of basic notion in functional analysis in Section \ref{subsec: Wiener type function space on RT and its continuity results}, we give the answer to \textbf{Question 1} in Section \ref{subsec: Suited function space}. Using some properties of differential operator established in Section \ref{subsec: Properties of operator dns}, we give the positive answer to \textbf{Question 2} in Section \ref{subsec: Continuity results for differential operator in function space}. We settle \textbf{Question 3} in Section \ref{subsec: Isomorphism in high periodicity and implicit function theorem}. Then, applying implicit function theorem, we can find $\br{\psi}=\br{\psi}(\Omega)$ so that $\br{L}(\br{\psi}, \Omega)=0$ for suitable $\Omega$.

Section \ref{sec: Self-similarity and coordinate change} and \ref{sec: Main theorem in new coordinate} are almost same as that of Elling \cite{Elling-2016} except that we described in more detail for careful readers. Therefore, if you already understood Elling \cite{Elling-2016}, you can omit Section \ref{sec: Self-similarity and coordinate change} and \ref{sec: Main theorem in new coordinate}.

In Section \ref{sec: Main theorem in original coordinate}, we first reverse the coordinate change we did in Section \ref{sec: Self-similarity and coordinate change} and construct physical solution $(w(x,t), u(x,t), \psi(x,t))$ via $\br{\psi}(\beta, \phi), \Omega(\phi)$ in Section \ref{subsec: Regularity of stream function and change of coordinate}. As we will guarantee suitable regularity for $\br{\psi}$ via constructing space $X$ in \textbf{Question 1} carefully, change of coordinate we've done can be justified (i.e., it is $C^1$-diffeomorphism). A big difference between Elling \cite{Elling-2016} and ours emerges at this timing. In \cite{Elling-2016}, even though $\Omega$ originally lives in $Y=\calA^{-0.5}(\bbT)$, he restrict the range of $\Omega$ to $\calA^{-0.5}(\bbT) \cap \calA^{0}(\bbT)$. He would do so, because $\calA^0(\bbT)$ consists of continuous function and he use this continuity crucially in his proof. However, we restrict the range of $\Omega$ from $\calA^{-0.5}(\bbT)$ to $\calA^{-0.5}(\bbT) \cap L^p(\bbT)$ for $p \in [1,2\mu)$, from which we can include a larger class of initial data. In Section \ref{subsec: Weak solution; weak version}, we show the constructed $(w(x,t), u(x,t), \psi(x,t))$ is a weak solution of \eqref{Euler eq} and \eqref{Euler eq-vor} on $\rmz \times [a,\infty)$ for $a>0$. Through the investigation of convergence of $(w(x,t), u(x,t), \psi (x,t))$ as $t \rightarrow 0$ and some uniform-in-time boundedness in Section \ref{subsec: Convergence to initial data and time independent boundedness}, we show that they are in fact weak solution on whole domain $\bbR^2 \times [0,\infty)$ on Section \ref{subsec: Weak solution; strong version }. After we establish the relation between  $g(\theta)$ in Theorem \ref{main theorem} and $\Omega(\phi)$, we prove our main theorem in Section \ref{subsec: Initial data in original coordinate and main theorem}.
	\vspace{3mm}
	
	\textbf{Acknowledgements.} The author thanks In-Jee Jeong for stimulating discussions and mentioning \cite{Elling-2013} to us. The author also acknowledges partial support by Samsung Science
and Technology Foundation (SSTF-BA2002-04). 
	\bigskip

\section{Self-similarity and coordinate change}
\label{sec: Self-similarity and coordinate change}

\subsection{Euler equation under self-similarity}
\label{subsec: Euler equation under self-similarity}

In this section, we assume all functions are smooth, so there are no obstacles in manipulating equations. We will see in Section \ref{subsec: Regularity of stream function and change of coordinate} that this assumption is valid at least in distribution sense as they have some reasonable regularity. Consider a 2D Euler equation in the vorticity form, which is derived from \eqref{Euler eq-vor} with smoothness assumption.

\begin{equation} \label{EE}
	\left\{\begin{aligned}
		& w_t+u \cdot \nabla_x w =0, \\
		& u=\nabla_x^\perp \psi ,\quad w= \Delta_x \psi. 
	\end{aligned}\right.
\end{equation}
Scaling property \eqref{scaled func} of Euler equation says that if $(w, u, \psi)$ satisfy \eqref{EE}, then
\begin{equation} 
	w_\lambda(x,t) \coloneqq \frac{1}{\lambda^a}w \left(\frac{x}{\lambda},\frac{t}{\lambda^a}\right), \quad u_\lambda(x,t) \coloneqq \frac{1}{\lambda^{a-1}}u\left(\frac{x}{\lambda},\frac{t}{\lambda^a}\right), \quad \psi_\lambda(x,t) \coloneqq \frac{1}{\lambda^{a-2}}\psi \left(\frac{x}{\lambda},\frac{t}{\lambda^a} \right)
\end{equation}
are also solution for \eqref{EE} for all $\lambda>0, a \in \mathbb{R}$. Under the self-similarity assumption \eqref{self-similarity}, take $\lambda=t^\mu$ and $a=\frac{1}{\mu}$, then we have
\begin{equation} \label{tilde func}
	\begin{aligned}
	w(x,t)&=t^{-1}w(\overbrace{xt^{-\mu}}^{\eqqcolon ~z},1) \eqqcolon t^{-1} \tilde{w}(z),\\
	u(x,t)&=t^{\mu-1}u(xt^{-\mu},1) \eqqcolon t^{\mu-1} \tilde{u}(z), \\
	\psi(x,t)&=t^{2\mu-1}\psi(xt^{-\mu},1) \eqqcolon t^{2\mu-1} \tilde{\psi}(z).
	\end{aligned}
\end{equation}
Substitute \eqref{tilde func} to \eqref{EE} and applying chain rule yields
\begin{equation} \label{EEz}
	\left\{\begin{aligned}
		& -\widetilde{w}+(\nabla^\perp_z \widetilde{\psi}-\mu z) \cdot \nabla_z \widetilde{w}=0, \\
		& \Delta_z \widetilde{\psi}=\widetilde{w}.
	\end{aligned} \right.
\end{equation}

\bigskip

\subsection{Special change of coordinate: From spiral to line}
\label{subsec: Special change of coordinate: From spiral to line}

As we are modeling spiral structure, it is natural to try to solve \eqref{EEz} in polar coordinate. Instead of using standard polar coordinate $(z_1,z_2)=(r \cos \theta, r \sin \theta)$, following Elling \cite{Elling-2016}, we use \textit{log-polar} coordinate.
\begin{equation} \label{log-polar coordinate}
	(a,\theta) \mapsto (z_1, z_2)=(e^a \cos \theta, e^a \sin \theta)
\end{equation}
for $(a, \theta) \in \bbR \times \bbT$. Note that in log-polar coordinate, algebraic spiral $\lvert z \rvert = \theta^{-\mu}$ corresponds to the graph of exponential function $\theta=e^{-\frac{a}{\mu}}$.
%TODO picture for exponential to spiral

\begin{figure}[!h]
	\includegraphics[trim = 31mm 211mm 16mm 33mm, clip, width=\textwidth]{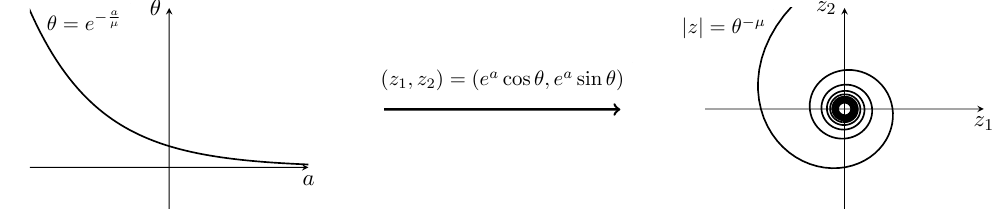}
	\caption{Exponential function in log-polar coordinate corresponds to algebraic spiral in Cartesian coordinate.} \label{Fig 1}
\end{figure}

However, we want more than that: We want to find another coordinate change $(\beta, \phi) \mapsto (a,\theta)=(a(\beta,\phi), \theta(a,\phi))$ such that
\begin{equation} \label{phil}
	\begin{aligned}
		&\textit{line parallel to $\beta$ - axis in $(\beta,\phi)$ - coordinate corresponds to} \\
		&~~~~ \quad \quad \textit{algebraic spiral in cartesian $(z_1, z_2)$ - coordinate.}
	\end{aligned}
\end{equation}
The following observation gives us a clue as to how to achieve this goal.

Consider stationary, radial, self-similar, trivial solution of \eqref{EE}
\begin{equation} \label{trivial sol}
	w_0(x,t)=C \lvert x \rvert^{-\frac{1}{\mu}}, \quad \psi_0(x,t)=C \left( 2-\frac{1}{\mu} \right)^{-2} \lvert x \rvert^{2-\frac{1}{\mu}}
\end{equation}
for $\mu \in \left( \frac{2}{3}, \infty \right)$, where corresponding solution in $(z_1, z_2)$ - coordinate is
\begin{equation} \label{trivial solz}
  \widetilde{w}_0(z)=C \lvert z \rvert^{-\frac{1}{\mu}}, \quad \widetilde{\psi}_0(z)=C \left( 2-\frac{1}{\mu} \right)^{-2} \lvert z \rvert^{2-\frac{1}{\mu}}.
\end{equation}
 Computing an integral curve of $\nabla^\perp_z \widetilde{\psi}-\mu z$ in \eqref{EEz} for explicit function \eqref{trivial solz} yields the following interesting result.
 
 \begin{proposition} \label{Int curve is alg spiral}
	For $\widetilde{\psi}_0(z)=C \left( 2-\frac{1}{\mu} \right)^{-2} \lvert z \rvert^{2-\frac{1}{\mu}}$, integral curve of $\nabla_z^\perp \widetilde{\psi}_0-\mu z$ is an algebraic spiral.
 \end{proposition}
 \begin{proof}
 	Via direct calculation,
 	\begin{equation*}
 		\nabla_z^\perp \widetilde{\psi}_0 - \mu z=C \left( 2-\frac{1}{\mu} \right)^{-1} \lvert z \rvert^{-\frac{1}{\mu}} (-z_2,z_1)-\mu (z_1, z_2).
 	\end{equation*}
 	Hence, an integral curve $(z_1(t), z_2(t))$ satisfies the following ODEs
\begin{equation*}
	\left\{\begin{aligned}
		z_1(t)'&=-C \left( 2 - \frac{1}{\mu} \right)^{-1} \lvert z(t) \rvert^{-\frac{1}{\mu}}z_2(t)-\mu z_1(t), \\
		z_2(t)' &= C \left( 2 - \frac{1}{\mu} \right)^{-1} \lvert z(t) \rvert^{-\frac{1}{\mu}}z_1(t)-\mu z_2(t).
	\end{aligned} \right.
\end{equation*}
Using this equation, we have
\begin{equation*}
	\left\{\begin{aligned}
		v_{z,rad}(t)&=\frac{z(t)}{\lvert z(t) \rvert} \cdot \left(z_1(t)', z_2(t)' \right)=-\mu\lvert z(t) \rvert, \\
		v_{z, ang}(t) &= \frac{z(t)^\perp}{\lvert z(t) \rvert} \cdot \left(z_1(t)', z_2(t)' \right)= C \left(2-\frac{1}{\mu} \right)^{-1} \lvert z(t) \rvert^{1-\frac{1}{\mu}}.
	\end{aligned} \right.
\end{equation*}
Therefore,
\begin{equation} \label{spiral ode}
	\frac{d\lvert z \rvert}{d \theta} = \frac{d \lvert z \rvert / dt}{d \theta/dt}=\frac{v_{z, rad}}{v_{z,ang}/\lvert z \rvert}=\frac{\mu\lvert z \rvert}{C \left(2-\frac{1}{\mu} \right)^{-1} \lvert z \rvert^{1-\frac{1}{\mu}}} = -\frac{\mu}{C} \left( 2-\frac{1}{\mu} \right) \lvert z \rvert^{1+\frac{1}{\mu}}.
\end{equation}
Solving ODE \eqref{spiral ode} gives
\begin{equation*}
	\lvert z(t) \rvert = \left( a \theta(t) +b \right)^{-\mu}
\end{equation*}
for some $a>0, b \in \mathbb{R}$, which is an algebraic spiral.
 \end{proof}

\begin{figure}[!h]	
	\includegraphics[trim = 31mm 165mm 15mm 33mm, clip, width=\textwidth]{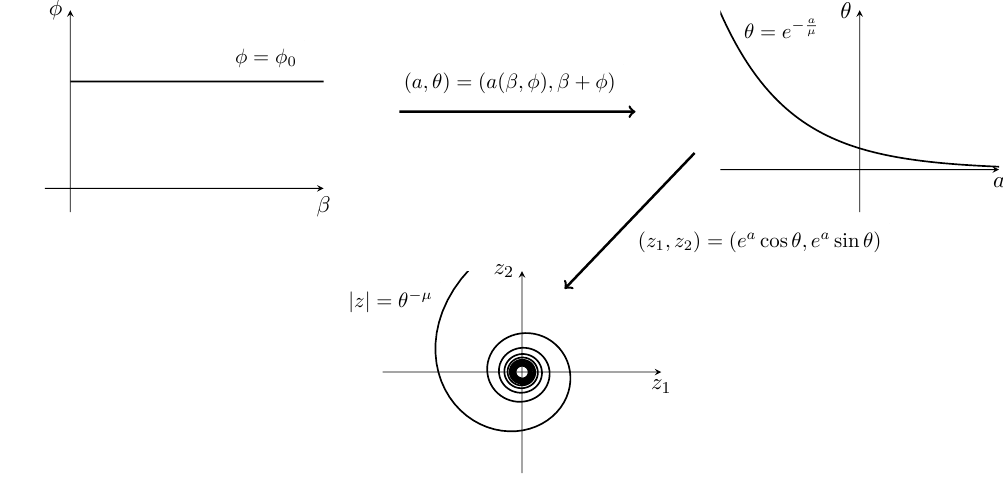}
	\caption{Strategy for setting new coordinate: Make algebraic spiral in Cartesian coordinate correspond to a line in new coordinate.} \label{Fig 2}
\end{figure}

Proposition \ref{Int curve is alg spiral} provides a strategy for choosing a new coordinate.
%TODO Insert the figure 2
Suppose we find a self-similar profile $\widetilde{\psi}(z)$ which satisfies \eqref{EEz}. Since we are about to apply implicit function theorem to find a solution, $\widetilde{\psi}(z)$ will be close to $\widetilde{\psi}_0(z)$ in some sense from which we can expect that integral curve of $\nabla_z^\perp\widetilde{\psi}-\mu z$ has spiral-like structure by the virtue of Proposition \ref{Int curve is alg spiral}. Therefore, given $\widetilde{\psi}(z)$, we choose $(a,\theta)=(a(\beta,\phi), \theta(\beta, \phi))$ so that when we pull back integral curve of $\nabla_z^\perp\widetilde{\psi}-\mu z$, it becomes a line parallel to $\beta$ - axis (Figure \ref{Fig 2}). Therefore, for fixed $\phi_0$, we hope that $\theta(\beta,\phi_0)$ is monotonically increasing as $\beta$ increases to make spiral winding. For this aim, it is natural to set $\theta(\beta,\phi)$ as
\begin{equation*}
	\theta(\beta,\phi)=\beta+\phi
\end{equation*}
for $(\beta,\phi) \in (\bbR_+\times \bbT)$. Before stating one of the key proposition, we introduce some notation. Let $T$ be a coordinate transformation map from $(\beta, \phi)$ to $(z_1, z_2)$ coordinate, which is the composition of two maps in Figure \ref{Fig 2}. Then, we denote pull back of $\widetilde{\psi}(z_1,z_2)$ to $(\beta,\phi)$ - coordinate simply by $\psi(\beta, \phi)$\footnote{Although we use same function symbol as in physical coordinate $(x,t)$, it may not cause confusion since the number of variable is different. Physical coordinate contains time variable, while $(z_1, z_2)$ coordinate does not. We will not use space-time coordinate until Section \ref{sec: Main theorem in original coordinate}.}, i.e.,
\begin{equation} \label{pull back notation}
	\widetilde{\psi}(T(\beta,\phi)) \eqqcolon \psi(\beta,\phi), \quad \widetilde{w}(T(\beta, \phi)) \eqqcolon w(\beta, \phi).
\end{equation}
Also, we introduce new differential operator as it appears so many time
\begin{equation} \label{diff phi}
	\rd_\varphi \coloneqq \rd_\phi -\rd_\beta.
\end{equation}
Then, we can achieve \eqref{phil} through the following proposition.

\begin{proposition} \label{new coordinate prop}
	Assume $\psi \in C^1\left(\bbR_+ \times \bbT \right)$ and $a_\varphi$ never vanish. If we choose
	\begin{equation} \label{new coordinate}
		a(\beta, \phi) = \frac{1}{2}\log \left( -\frac{\psi_\beta}{\mu} \right), \quad \theta(\beta, \phi)=\beta + \phi,
	\end{equation}
	the the pull back of integral curve of $\nabla_z^\perp \widetilde{\psi}- \mu z$ to $(\beta, \phi)$ - coordinate is parallel to $\beta - axis$\footnote{For $a(\beta, \phi)$ to be well-defined as in \eqref{new coordinate}, $\psi_\beta$ should be strictly negative for all $(\beta, \phi) \in \bbR_+ \times \bbT$. This will be verified in Section \ref{subsec: Continuity results for differential operator in function space}.}.
\end{proposition}

\begin{proof}
	Note that coordinate change map $T$, which is a composition of two coordinate change map; from $(a, \theta)$ to $(z_1,z_2)$ and from $(\beta, \phi)$ to $(a,\theta)$, can be written explicitly
	\begin{equation} \label{map T}
		T(\beta, \phi)= \left( e^{a(\beta, \phi)} \cos (\beta+\phi), e^{a(\beta,\phi)} \sin (\beta+\phi) \right) = \left( e^a \cos \theta, e^a \sin \theta \right).
	\end{equation}
	We can also explicitly calculate
	\begin{equation} \label{jacobian}
		J_T=
		\begin{pmatrix}
			a_\beta e^a \cos \theta-e^a \sin \theta & a_\phi e^a \cos \theta - e^a \sin \theta \\
			a_\beta e^a \sin \theta+e^a \cos \theta & a_\phi e^a \sin \theta + e^a \cos \theta
		\end{pmatrix},
	\end{equation}
	
	\begin{equation} \label{determinant}
		\abs{J_T}=-e^{2a}a_\varphi,
	\end{equation}
	
	\begin{equation} \label{diff change}
		\begin{pmatrix}
			\rd_{z_1} \\
			\rd_{z_2} \\
		\end{pmatrix}
		= \left( J_T^{-1} \right)^t
		\begin{pmatrix}
			\rd_\beta \\
			\rd_\phi \\
		\end{pmatrix}
		=-\frac{1}{e^a a_\varphi}
		\begin{pmatrix}
			a_\phi \sin \theta + \cos \theta & -a_\beta \sin \theta - \cos \theta \\
			-a_\phi \cos \theta + \sin \theta & a_\beta \cos \theta - \sin \theta \\
		\end{pmatrix}
		\begin{pmatrix}
			\rd_\beta \\
			\rd_\phi \\
		\end{pmatrix},
	\end{equation}
where $J_T$ is the Jacobian matrix of $T$ and $\abs{J_T}$ denotes its determinant. Then, the $\phi$ - component of pull back of $\nabla_z^\perp \widetilde{\psi}-\mu z$ to $(\beta, \theta)$ - coordinate is
\begin{align*}
	& \underbrace{(0,1)}_{\substack{\text{extract} \\ \phi - \text{component}}} \cdot \underbrace{ \left( J_T \right)^{-1}}_{\text{pull back}} \cdot \underbrace{ \Bigg( \underbrace{
	\begin{pmatrix}
		0 & -1 \\
		1 & 0 \\
	\end{pmatrix}}_{\substack{\text{gradient} \\ \text{perp}}} \underbrace{\left( J_T^{-1} \right)^t \begin{pmatrix}
		\rd_\beta \\
		\rd_\phi \\
	\end{pmatrix}}_{\nabla_{z}} \psi - \mu \underbrace{\begin{pmatrix}
		e^a \cos \theta \\
		e^a \sin \theta \\
	\end{pmatrix}}_{z} \Bigg) }_{\nabla^\perp_z \widetilde{\psi}-\mu z} \\&\underset{\eqref{jacobian}}{=} - \frac{\psi_\beta}{e^{2a}a_\varphi}-\frac{\mu}{a_\varphi}.
\end{align*}
Observe that the last term is zero if and only if
\begin{equation*}
	a(\beta, \phi)=\frac{1}{2} \log \left( -\frac{\psi_\beta}{\mu}  \right).
\end{equation*}
\end{proof}

\bigskip

\subsection{Equations in new coordinate}
\label{subsec: Equations in new coordinate}

In this subsection, we rewrite equation \eqref{EEz} in $(\beta, \phi)$ - coordinate where the relationship between each coordinate is given by \eqref{map T}.

We first deal with the first equation in \eqref{EEz} under the assumption that $a_\varphi$ and $\psi_\varphi$ never vanish, which will be justified later.
\begin{equation} \label{EE to ODE}
	\begin{aligned}
	0 & \underset{\eqref{EEz}}{=} -\widetilde{w}(z)+(\nabla^\perp_z \widetilde{\psi}(z)-\mu z) \cdot \nabla_z \widetilde{w}(z) \\
	& =-\widetilde{w}(z)+ \left( \begin{pmatrix}
		0 & -1 \\
		1 & 0 \\
	\end{pmatrix} \begin{pmatrix}
		\rd_{z_1} \\
		\rd_{z_2} \\
	\end{pmatrix} \widetilde{\psi}(z)-\mu
	\begin{pmatrix}
		z_1 \\
		z_2 \\
	\end{pmatrix} \right) \cdot \begin{pmatrix}
		\rd_{z_1} \\
		\rd_{z_2} \\
	\end{pmatrix} \widetilde{w}(z) \\
	& \text{(after pull back)} \\
		\underset{\eqref{diff change}}&{=} -w(\beta,\phi)+ \left( \begin{pmatrix}
		0 & -1 \\
		1 & 0 \\
	\end{pmatrix} \left( J_T^{-1} \right)^t \begin{pmatrix}
		\rd_\beta \\
		\rd_\phi \\
	\end{pmatrix} \psi(\beta, \phi) -\mu \begin{pmatrix}
		e^{a(\beta, \phi)} \cos (\beta+\phi) \\
		e^{a(\beta, \phi)} \sin (\beta+\phi) \\
	\end{pmatrix} \right) \cdot \left( \left( J_T^{-1} \right)^t \begin{pmatrix}
		\rd_\beta \\
		\rd_\phi \\
	\end{pmatrix} w(\beta, \phi) \right)\\
	\underset{\eqref{jacobian}}&{=}-w-\frac{1}{e^{2a}a_\varphi} \left( \psi_\beta w_\phi - \psi_\phi w_\beta \right) - \frac{\mu}{a_\varphi}w_\varphi \\
	&=\frac{1}{e^{2a}a_\varphi} \Bigg( \underbrace{e^{2a}a_\varphi}_{\substack{\text{By } \eqref{new coordinate} \\ e^{2a}=-\frac{\psi_\beta}{\mu}, ~ a_\varphi = \frac{\psi_{\beta \varphi}}{2\psi_\beta}}} w - \psi_\beta w_\phi +\psi_\phi w_\beta - \underbrace{\mu e^{2a}}_{-\psi_\beta}w_\varphi \Bigg) \\
	&=\frac{1}{e^{2a}a_\varphi} \left( \frac{\psi_{\beta \varphi}w}{2 \mu} - \psi_\beta w_\phi + \psi_\phi w_\beta +\psi_\beta w_\varphi \right) \\
	&=\frac{1}{e^{2a}a_\varphi} \Bigg( \frac{\psi_{\beta \varphi}w}{2 \mu} \underbrace{- \psi_\beta w_\phi + \psi_\beta w_\beta}_{-\psi_\beta w_\varphi} \underbrace{- \psi_\beta w_\beta+ \psi_\phi w_\beta}_{\psi_\varphi w_\beta} +\psi_\beta w_\varphi \Bigg) \\
	&=\frac{1}{e^{2a} a_\varphi} \left( \frac{\psi_{\beta \varphi}w}{2 \mu}+\psi_\varphi w_\beta \right) \\
	&= \frac{(\psi_\varphi)^{1-\frac{1}{2 \mu}}}{e^{2a} a_\varphi} \left( \frac{1}{2 \mu}(\psi_\varphi)^{\frac{1}{2 \mu}-1} \psi_{\beta \varphi}w+(\psi_\varphi)^{\frac{1}{2 \mu}} w_\beta \right) \\
	&= \frac{(\psi_\varphi)^{1-\frac{1}{2 \mu}}}{e^{2a} a_\varphi} \left( (\psi_\varphi)^{\frac{1}{2 \mu}}w \right)_\beta. \\ &&
\end{aligned}
\end{equation}

As we assume
\begin{equation*}
	\psi_\varphi \neq 0, \quad a_\varphi \neq 0,
\end{equation*}
we derive from above equalities that there exists $\Omega(\phi)$, depending only on $\phi$, such that
\begin{equation} \label{w formula}
	w(\beta, \phi)=\left( \psi_\varphi \right)^{-\frac{1}{2\mu}} \Omega(\phi).
\end{equation}
\bigskip

Next, we deal with the second equation in \eqref{EEz}; the vorticity-stream relation $\widetilde{w}=\Delta\widetilde{\psi}$. We utilize adjoint formula to facilitate the calculation. The adjoint formula says that for $C^1$ - coordinate change map $T: (\beta, \phi) \mapsto (z_1, z_2)$ and $\tld{v}(z) \in C^0 (\bbR^2; \bbR^2), \tld{f}(z) \in C^0(\bbR^2)$,
\begin{equation} \label{adjoint formula}	
	\nabla_z \cdot \widetilde{v} = \widetilde{f} \iff \nabla_{(\beta,\phi)} \cdot \left( \abs{J_T} J_T^{-1} \widetilde{v}(T(\beta,\phi)) \right)=\abs{J_T} \widetilde{f}(T(\beta, \phi)),
\end{equation}
where the divergence would be understood as distribution sense.

By adjoint formula, $\widetilde{w}=\Delta_z \widetilde{\psi}=\nabla_z \cdot \left( \nabla_z \widetilde{\psi} \right)$ in $(z_1, z_2)$ - coordinate holds if and only if
\begin{equation} \label{applying adj formula}
	\begin{pmatrix}
		\rd_\beta \\
		\rd_\phi \\
	\end{pmatrix} \cdot \left( \abs{J_T} J_T^{-1} \left(J_T^{-1} \right)^t \begin{pmatrix}
		\rd_\beta \\
		\rd_\phi \\
	\end{pmatrix} \psi \right) = \abs{J_T}w.
\end{equation}
The right-hand side of $\eqref{applying adj formula}$ is simply
\begin{equation*}
	\abs{J_T} w \underset{\eqref{determinant} }{=} -e^{2a} a_\varphi w \underset{\eqref{new coordinate}}{=}\frac{\psi_{\beta \varphi}}{2 \mu}w.
\end{equation*}
The left-hand side of \eqref{applying adj formula} can also be calculated by \eqref{jacobian} and \eqref{determinant};
\begin{align*}
	&\begin{pmatrix}
		\rd_\beta \\
		\rd_\phi \\
	\end{pmatrix} \cdot \left( \abs{J_T} J_T^{-1} \left(J_T^{-1} \right)^t \begin{pmatrix}
		\rd_\beta \\
		\rd_\phi \\
	\end{pmatrix} \psi \right) \\
	&=\begin{pmatrix}
		\rd_\beta \\
		\rd_\phi \\
	\end{pmatrix} \cdot \Bigg( \underbrace{e^a \begin{pmatrix}
		a_\phi \sin \theta+\cos \theta & -a_\phi \cos \theta+ \sin \theta \\
		-a_\beta \sin \theta-\cos \theta & a_\beta \cos \theta-\sin \theta \\
	\end{pmatrix}}_{\abs{J_T}J_T^{-1}} \underbrace{\frac{1}{e^a a_\varphi} \begin{pmatrix}
		(-a_\phi \sin \theta - \cos \theta) \psi_\beta+(a_\beta \sin \theta + \cos \theta)\psi_\phi \\
		(a_\phi \cos \theta - \sin \theta) \psi_\beta + (-a_\beta \cos \theta + \sin \theta) \psi_\phi \\
	\end{pmatrix}}_{\left( J_T^{-1} \right)^t \begin{pmatrix}
		\rd_\beta \\
		\rd_\phi \\
	\end{pmatrix}\psi } \Bigg) \\
	&=\begin{pmatrix}
		\rd_\beta \\
		\rd_\phi \\
	\end{pmatrix} \cdot \left( -\frac{1}{a_\varphi} \begin{pmatrix}
		\left( (a_\phi)^2+1 \right) \psi_\beta + \left( -a_\beta a_\phi -1 \right) \psi_\phi \\
		(-a_\beta a_\phi -1) \psi_\beta+ \left( (a_\beta)^2+1 \right) \psi_\phi
	\end{pmatrix} \right)\\
	&= \left(  \frac{\left( (a_\phi)^2+1 \right) \psi_\beta + \left( -a_\beta a_\phi -1 \right) \psi_\phi}{-a_\varphi}  \right)_\beta+ \left( \frac{(-a_\beta a_\phi -1) \psi_\beta+ \left( (a_\beta)^2+1 \right) \psi_\phi}{-a_\varphi}  \right)_\phi \\
	\underset{\rd_\beta=\rd_\phi-\rd_\varphi}&{=} \left( \frac{\left( (a_\phi)^2+1 \right) \psi_\beta + \left( -a_\beta a_\phi -1 \right) \psi_\phi+(-a_\beta a_\phi -1) \psi_\beta+ \left( (a_\beta)^2+1 \right) \psi_\phi}{-a_\varphi} \right)_\phi \\
	& \quad \qquad - \left( \frac{\left( (a_\phi)^2+1 \right) \psi_\beta + \left( -a_\beta a_\phi -1 \right) \psi_\phi}{-a_\varphi} \right)_\varphi \\
	&= \left( \frac{\left( a_\phi^2+1 \right)(\psi_\varphi - \psi_\phi)+(a_\beta a_\phi +1)\psi_\phi}{-a_\varphi} \right)_\varphi + \left( \frac{a_\phi a_\varphi \psi_\beta - a_\beta a_\varphi \psi_\phi}{-a_\varphi} \right)_\phi \\
	&= -\left( \frac{\left( a_\phi^2+1 \right) \psi_\varphi}{a_\varphi}-a_\phi \psi_\phi \right)_\varphi - \left(  a_\phi \psi_\beta-a_\beta\psi_\phi \right)_\phi \\
	\underset{\eqref{new coordinate}}&{=} - \left( \frac{2 \psi_\beta \psi_\varphi}{\psi_{\beta \varphi}} \left(1+\left(\frac{\psi_{\beta \phi}}{2 \psi_\beta} \right)^2 \right)-\frac{\psi_{\beta \phi} \psi_\phi}{2 \psi_\beta}\right)_\varphi - \left( a_\phi \psi_\phi- a_\phi \psi_\varphi -a_\beta \psi_\phi \right)_\phi \\
	&= - \left( \frac{2 \psi_\beta \psi_\varphi}{\psi_{\beta \varphi}} \left(1+\left(\frac{\psi_{\beta \phi}}{2 \psi_\beta} \right)^2 \right)-\frac{\psi_{\beta \phi} \psi_\phi}{2 \psi_\beta} \right)_\varphi - \left(  a_\varphi \psi_\phi -a_\phi \psi_\varphi \right)_\phi \\
	\underset{\eqref{new coordinate}}&{=} - \left( \frac{2 \psi_\beta \psi_\varphi}{\psi_{\beta \varphi}} \left(1+\left(\frac{\psi_{\beta \phi}}{2 \psi_\beta} \right)^2 \right)-\frac{\psi_{\beta \phi} \psi_\phi}{2 \psi_\beta} \right)_\varphi - \left(  \frac{\psi_{\beta \varphi}\psi_\phi - \psi_{\beta \phi} \psi_\varphi}{2\psi_\beta} \right)_\phi. \\ 
	\end{align*}
Therefore, equation \eqref{applying adj formula} is transformed to
\begin{equation} \label{L before}
	\left( \frac{2 \psi_\beta \psi_\varphi}{\psi_{\beta \varphi}} \left(1+\left(\frac{\psi_{\beta \phi}}{2 \psi_\beta} \right)^2 \right)-\frac{\psi_{\beta \phi} \psi_\phi}{2 \psi_\beta} \right)_\varphi + \left(  \frac{\psi_{\beta \varphi}\psi_\phi - \psi_{\beta \phi} \psi_\varphi}{2\psi_\beta} \right)_\phi+\frac{\psi_{\beta \varphi}}{2 \mu}w=0.
\end{equation}
Using \eqref{w formula}, \eqref{L before} becomes
\begin{equation} \label{def of L}
	\begin{aligned}
	& \left( \frac{2 \psi_\beta \psi_\varphi}{\psi_{\beta \varphi}} \left(1+\left(\frac{\psi_{\beta \phi}}{2 \psi_\beta} \right)^2 \right)-\frac{\psi_{\beta \phi} \psi_\phi}{2 \psi_\beta} \right)_\varphi + \left(  \frac{\psi_{\beta \varphi}\psi_\phi - \psi_{\beta \phi} \psi_\varphi}{2\psi_\beta} \right)_\phi+\frac{\psi_{\beta \varphi}}{2 \mu} \left(\psi_\varphi \right)^{-\frac{1}{2 \mu}}\Omega \\
	& \eqqcolon L(\psi, \Omega)=0.
\end{aligned}
\end{equation}

As we will see in Section \ref{sec: Main theorem in original coordinate}, $\Omega(\phi)$ is closely related to $g(\theta)$ in %Todo
(main theorem). Since our goal is to find a solution $w(x,t)$ for given $g(\theta)$ in Theorem \ref{main theorem}, our primary goal is to find $\psi = \psi(\Omega)$ for given suitable $\Omega$. After finding $\psi$ associated with $\Omega$, defining $w$ as in \eqref{w formula} make \eqref{EEz} be satisfied, and then we can find self-similar solution of \eqref{EE} through \eqref{tilde func}.

\bigskip

\subsection{Rescaling to eliminate decay}
\label{subsec: Rescaling to eliminate decay}

As we implied, we achieve above primary goal through implicit function theorem. That means we should find at least one solution for the map $L$. Obviously, we expect that pull back of \eqref{trivial solz} to $(\beta, \phi)$ - coordinate should be a trivial solution for map $L$. In this subsection, we explicitly calculate pull back of $\widetilde{\psi}_0$ and related $\Omega_0$ to find a trivial solution of $L$. After that, we normalize stream function $\psi$ to make trivial solution a constant function.

Remember that a trivial solution in $(z_1, z_2)$ - coordinate is given by
\begin{equation*}
	\widetilde{w}_0(z) = C \abs{z}^{-\frac{1}{\mu}}, \qquad \widetilde{\psi}_0(z)=C \left( 2-\frac{1}{\mu} \right)^{-2} \abs{z}^{2-\frac{1}{\mu}}
\end{equation*}
for some constant $C$. Also, we adopted coordinate change map $T$ given by \eqref{map T} and $a(\beta, \phi), \theta(\beta, \phi)$ is given so that \eqref{new coordinate} holds. Then,
\begin{equation} \label{eq1}
	\psi_0(\beta, \phi)=\widetilde{\psi}_0(T(\beta, \phi))=C \left(2-\frac{1}{\mu} \right)^{-2} e^{\left(2-\frac{1}{\mu}\right)a}.
\end{equation}
Differentiating \eqref{eq1} in $\beta$ - coordinate and use \eqref{new coordinate}, we have
\begin{equation*}
	C \left(2-\frac{1}{\mu} \right)^{-1} a_\beta \,e^{\left(2-\frac{1}{\mu}\right)a} \underset{\eqref{eq1}}{=} \psi_\beta \underset{\eqref{new coordinate}}{=} -\mu e^{2a}
\end{equation*}
and this leads to 
\begin{equation*}
	-\frac{a_\beta}{\mu}e^{-\frac{a}{\mu}}=\frac{1}{C} \left( 2-\frac{1}{\mu} \right).
\end{equation*}
Integrating with respect to $\beta$ coordinate, we have
\begin{equation} \label{eq2}
	e^{-\frac{a}{\mu}} = \frac{1}{C} \left( 2-\frac{1}{\mu} \right) \beta.
\end{equation}
Note that integration constant should be chosen zero for $T$ to be bijective. Take the power of $1-2\mu$ both side in \eqref{eq2} and substitute the value of $e^{\left( 2-\frac{1}{\mu}\right)a}$ to \eqref{eq1} gives
\begin{equation} \label{eq3}
	\psi_0(\beta, \phi)= \left( \frac{C \mu^{\frac{1}{2\mu}}}{2-\frac{1}{\mu}} \right)^{2\mu} \frac{1}{2\mu-1} \beta^{1-2 \mu}.
\end{equation}
Also, from the relation
\begin{align*}
	&\left( 2 - \frac{1}{\mu} \right) \beta \underset{\eqref{eq2}}{=}Ce^{-\frac{a}{\mu}} \underset{\eqref{trivial solz}}{=} w(\beta, \phi) \underset{\eqref{w formula}}{=} \left( (\psi_0)_\varphi \right)^{-\frac{1}{2\mu}} \Omega_0(\phi) \\
	& \underset{\eqref{eq3}}{=}\left( -C^{2\mu} \left( 2-\frac{1}{2\mu} \right)^{-1-2\mu} (1-2\mu) \beta^{-2\mu}  \right)^{-\frac{1}{2 \mu}} \Omega_0(\phi) = \frac{1}{C} \left( 2-\frac{1}{\mu} \right)^{1+\frac{1}{2\mu}} (2\mu-1)^{-\frac{1}{2\mu}} \beta \Omega_0(\phi)
\end{align*}
we have
\begin{equation} \label{eq4}
	\Omega_0(\phi)=C \mu^{\frac{1}{2 \mu}}.
\end{equation}

Since $C$ was arbitrary, it is reasonable to choose $C$ so that coefficient of \eqref{eq3} and \eqref{eq4} be as simple as possible. Therefore, we choose
\begin{equation} \label{choice of C}
	C=\mu^{-\frac{1}{2\mu}} \left( 2-\frac{1}{\mu} \right).
\end{equation}
With this choice of $C$, we fix trivial solution of map $L$ from now on:
\begin{equation} \label{trivial sol new}
	\psi_0(\beta, \phi)=\frac{1}{2\mu -1} \beta^{1-2\mu}, \qquad \Omega_0(\phi)=2-\frac{1}{\mu}.
\end{equation}
Note that $\psi_0(\beta, \psi)$ is a positive polynomial function with
\begin{equation*}
	\lim_{\beta \rightarrow 0} \psi_0 (\beta, \phi) = +\infty, \qquad \lim_{\beta \rightarrow \infty} \psi_0(\beta, \phi)=0.
\end{equation*}
In most of the situation, trivial solution is desired to be simple as possible. Since $\psi_0(\beta, \phi)$ is constant times $\beta^{1-2 \mu}$ and other non-trivial solution $\psi(\beta, \phi)$ will be close in some sense to $\psi_0(\beta, \phi)$ (remember, we find non-trivial $\psi(\beta, \phi)$ via implicit function theorem), we deal with function after dividing it with $\beta^{1-2\mu}$, i.e., we deal with $\overline{\psi}(\beta, \phi)$ where
\begin{equation} \label{psi and psi bar}
	\psi(\beta, \phi)=\beta^{1-2\mu} \br{\psi}(\beta, \phi).
\end{equation}
Investigating the map $L$ shows that we need to calculate
\begin{align}
	\psi_\beta \underset{\eqref{psi and psi bar} }&{=} (1-2\mu)\beta^{-2\mu} \br{\psi} + \beta^{1-2\mu} \br{\psi}_\beta = \beta^{-2\mu} \left(\beta \br{\psi}_\beta +(1-2\mu) \bar{\psi} \right) \label{eq5},\\
	\psi_\varphi \underset{\eqref{psi and psi bar} }&{=} \beta^{1-2\mu} \br{\psi}_\phi - (1-2\mu) \beta^{-2\mu} \br{\psi} - \beta^{1-2\mu} \bar{\psi}_\beta=\beta^{-2 \mu} \left( \beta \br{\psi}_\varphi + (2\mu-1) \br{\psi} \right) \label{eq6}, \\
	\psi_\phi \underset{\eqref{psi and psi bar} }&{=} \beta^{1-2\mu} \br{\psi}_\phi \label{eq7}.
\end{align}
For simplicity and convenience, we define \textit{differential bar operator} $\br{\rd}_\beta, \br{\rd}_\varphi$ by
\begin{equation} \label{diff beta bar}
	\br{\rd}_\beta \coloneqq \beta \rd_\beta + (1-2\mu)id,
\end{equation}
\begin{equation} \label{diff varphi bar}
	\br{\rd}_\varphi \coloneqq \beta \rd_\varphi + (2\mu-1)id.
\end{equation}
Then, we can represent partial derivatives of $\psi$ as partial bar derivative of $\bar{\psi}$ as followings:
\begin{align}
	\psi_\beta &= \beta^{-2\mu} \br{\rd}_\beta \br{\psi} \label{beta and bar beta}, \\
	\psi_\varphi &= \beta^{-2\mu} \br{\rd}_\varphi \br{\psi} \label{varphi and bar var phi}, \\
	\psi_\phi &= \beta^{1-2\mu} \rd_\phi \br{\psi} \label{phi and bar phi}, \\
	\psi_{\beta \phi} &= \beta^{-2\mu} \rd_\phi \br{\rd}_\beta \br{\psi} \label {betaphi and bar betaphi}, \\
	\psi_{\beta \varphi} &= \beta^{-2\mu -1} \left( \bar{\rd}_\varphi +1 \right) \br{\rd}_\beta \br{\psi} \label{betavarphi and bar betavarphi}.
\end{align}
Substitute \eqref{beta and bar beta} $\sim$ \eqref{betavarphi and bar betavarphi} to each term in $L$ gives
\begin{equation} \label{eq8}
	\begin{aligned}
		\frac{\psi_{\beta \varphi}}{2 \mu} \left(\psi_\varphi \right)^{-\frac{1}{2 \mu}}\Omega &= \frac{\beta^{-2\mu -1} \left( \bar{\rd}_\varphi +1 \right) \br{\rd}_\beta \br{\psi}}{2\mu} \cdot \beta \left( \br{\rd}_\varphi \br{\psi} \right) ^{-\frac{1}{2\mu}} \Omega\\
		 &= \beta^{-2\mu} \cdot \frac{\left( \bar{\rd}_\varphi +1 \right) \br{\rd}_\beta \br{\psi} \cdot \left( \br{\rd}_\varphi \br{\psi} \right) ^{-\frac{1}{2\mu}} }{2 \mu} \Omega,
	\end{aligned}
\end{equation}

\begin{equation} \label{eq9}
	\begin{aligned}
		\left(  \frac{\psi_{\beta \varphi}\psi_\phi - \psi_{\beta \phi} \psi_\varphi}{2\psi_\beta} \right)_\phi &= \left(  \frac{\beta^{-2\mu -1} \left( \bar{\rd}_\varphi +1 \right) \br{\rd}_\beta \br{\psi} \cdot \beta^{1-2\mu} \rd_\phi \br{\psi} - \beta^{-2\mu} \rd_\phi \br{\rd}_\beta \br{\psi} \cdot \beta^{-2\mu} \br{\rd}_\varphi \br{\psi} }{2 \beta^{-2\mu} \br{\rd}_\beta \br{\psi}} \right)_\phi \\
		&= \beta^{-2\mu} \left(  \frac{ \left( \bar{\rd}_\varphi +1 \right) \br{\rd}_\beta \br{\psi} \cdot  \rd_\phi \br{\psi} - \rd_\phi \br{\rd}_\beta \br{\psi} \cdot \br{\rd}_\varphi \br{\psi}}{2\br{\rd}_\beta \br{\psi}  }\right)_\phi,
	\end{aligned}	
\end{equation}

\begin{equation} \label{eq10}
 \begin{aligned}
 	& \left( \frac{2 \psi_\beta \psi_\varphi}{\psi_{\beta \varphi}} \left(1+\left(\frac{\psi_{\beta \phi}}{2 \psi_\beta} \right)^2 \right)-\frac{\psi_{\beta \phi} \psi_\phi}{2 \psi_\beta} \right)_\varphi \\
 	&= \left( \frac{2\beta^{-2\mu} \br{\rd}_\beta \br{\psi} \cdot \beta^{-2\mu} \br{\rd}_\varphi \br{\psi}}{\beta^{-2\mu -1} \left( \bar{\rd}_\varphi +1 \right) \br{\rd}_\beta \br{\psi}} \left( 1+ \left( \frac{\rd_\phi \br{\rd}_\beta \br{\psi}}{2\br{\rd}_\beta} \right)^2  \right) - \frac{\beta^{-2\mu} \rd_\phi \br{\rd}_\beta \br{\psi} \cdot \beta^{1-2\mu} \rd_\phi \br{\psi}   }{2 \beta^{-2\mu} \br{\rd}_\beta \br{\psi}}  \right)_\varphi \\
 	&= \left( \beta^{1-2\mu} \cdot \left( \frac{2\br{\rd}_\beta \br{\psi} \cdot  \br{\rd}_\varphi \br{\psi}}{\left( \bar{\rd}_\varphi +1 \right) \br{\rd}_\beta \br{\psi}} \left( 1+ \left( \frac{\rd_\phi \br{\rd}_\beta \br{\psi}}{2\br{\rd}_\beta} \right)^2  \right) - \frac{\rd_\phi \br{\rd}_\beta \br{\psi} \cdot  \rd_\phi \br{\psi}   }{2  \br{\rd}_\beta \br{\psi}}  \right) \right)_\varphi \\
 	\underset{\eqref{eq6}}&{=} \beta^{-2\mu} \cdot \br{\rd}_\varphi \left( \frac{2\br{\rd}_\beta \br{\psi} \cdot  \br{\rd}_\varphi \br{\psi}}{\left( \bar{\rd}_\varphi +1 \right) \br{\rd}_\beta \br{\psi}} \left( 1+ \left( \frac{\rd_\phi \br{\rd}_\beta \br{\psi}}{2\br{\rd}_\beta} \right)^2  \right) - \frac{\rd_\phi \br{\rd}_\beta \br{\psi} \cdot  \rd_\phi \br{\psi}   }{2  \br{\rd}_\beta \br{\psi}}  \right).   
 \end{aligned}	
\end{equation}
Hence, $L(\psi, \Omega)$ can be represented as
\begin{align*}
	L(\psi, \Omega) &= \left( \frac{2 \psi_\beta \psi_\varphi}{\psi_{\beta \varphi}} \left(1+\left(\frac{\psi_{\beta \phi}}{2 \psi_\beta} \right)^2 \right)-\frac{\psi_{\beta \phi} \psi_\phi}{2 \psi_\beta} \right)_\varphi
	+\left(  \frac{\psi_{\beta \varphi}\psi_\phi - \psi_{\beta \phi} \psi_\varphi}{2\psi_\beta} \right)_\phi +\frac{\psi_{\beta \varphi}}{2 \mu} \left(\psi_\varphi \right)^{-\frac{1}{2 \mu}}\Omega \\
	\underset{\eqref{eq8},\eqref{eq9}, \eqref{eq10}}&{=} \beta^{-2\mu} \cdot \Bigg(\br{\rd}_\varphi \left( \frac{2\br{\rd}_\beta \br{\psi} \cdot  \br{\rd}_\varphi \br{\psi}}{\left( \bar{\rd}_\varphi +1 \right) \br{\rd}_\beta \br{\psi}} \left( 1+ \left( \frac{\rd_\phi \br{\rd}_\beta \br{\psi}}{2\br{\rd}_\beta} \right)^2  \right) - \frac{\rd_\phi \br{\rd}_\beta \br{\psi} \cdot  \rd_\phi \br{\psi}   }{2  \br{\rd}_\beta \br{\psi}}  \right)  \\
	&\qquad \qquad \qquad \quad+ \rd_\phi \left(  \frac{ \left( \bar{\rd}_\varphi +1 \right) \br{\rd}_\beta \br{\psi} \cdot  \rd_\phi \br{\psi} - \rd_\phi \br{\rd}_\beta \br{\psi} \cdot \br{\rd}_\varphi \br{\psi}}{2\br{\rd}_\beta \br{\psi}  }\right)
	+ \frac{\left( \bar{\rd}_\varphi +1 \right) \br{\rd}_\beta \br{\psi} \cdot \left( \br{\rd}_\varphi \br{\psi} \right) ^{-\frac{1}{2\mu}} }{2 \mu} \Omega \Bigg) \\
	&\eqqcolon \beta^{-2 \mu} \br{L}(\br{\psi}, \Omega).
\end{align*}

\section{Main theorem in new coordinate}
\label{sec: Main theorem in new coordinate}

Through suitable scaling, we've reduced the problem of finding solution $\psi=\psi(\Omega)$ of $L(\psi, \Omega)$ to find $\br{\psi} = \br{\psi}(\Omega)$ of $\br{L}(\br{\psi},\Omega)$ where
\begin{equation} \label{L bar}
	\begin{aligned}
		 	&\br{L}(\br{\psi},\Omega)=\br{\rd}_\varphi \left( \frac{2\br{\rd}_\beta \br{\psi} \cdot  \br{\rd}_\varphi \br{\psi}}{\left( \bar{\rd}_\varphi +1 \right) \br{\rd}_\beta \br{\psi}} \left( 1+ \left( \frac{\rd_\phi \br{\rd}_\beta \br{\psi}}{2\br{\rd}_\beta} \right)^2  \right) - \frac{\rd_\phi \br{\rd}_\beta \br{\psi} \cdot  \rd_\phi \br{\psi}   }{2  \br{\rd}_\beta \br{\psi}}  \right) \\
		 	& \qquad \qquad + \rd_\phi \left(  \frac{ \left( \bar{\rd}_\varphi +1 \right) \br{\rd}_\beta \br{\psi} \cdot  \rd_\phi \br{\psi} - \rd_\phi \br{\rd}_\beta \br{\psi} \cdot \br{\rd}_\varphi \br{\psi}}{2\br{\rd}_\beta \br{\psi}  }\right)
	+ \frac{\left( \bar{\rd}_\varphi +1 \right) \br{\rd}_\beta \br{\psi} \cdot \left( \br{\rd}_\varphi \br{\psi} \right) ^{-\frac{1}{2\mu}} }{2 \mu} \Omega.
	\end{aligned}
\end{equation}
Once $\br{\psi}$ is found, the process of recovering $\psi(\beta, \phi)$ and $w(\beta, \phi)$ is straight-forward through the relation \eqref{psi and psi bar} and \eqref{w formula}. As we want to solve functional equation
\begin{equation*}
	\br{L}(\br{\psi}, \Omega)=0
\end{equation*}
by using implicit function theorem around trivial solution
\begin{equation} \label{trivial solz bar}
	\br{\psi}_0 \underset{\eqref{psi and psi bar}}{=} \frac{\psi_0}{\beta^{1-2\mu}}  \underset{\eqref{trivial sol new}}{=}\frac{1}{2\mu-1}, \qquad \Omega_0 \underset{\eqref{trivial sol new}}{=} 2-\frac{1}{\mu},
\end{equation}
we have to answer following three questions:

\bigskip
\textbf{$\bullet$ Question 1:} \label{Q1} What is the Banach space $X, Y$ and $Z$ that will be a domain and codomain for $\br{L}: X \times Y \rightarrow Z?$ In other words, where does $\br{\psi}, \Omega, \br{L}(\br{\psi},\Omega)$ lives?

\bigskip

\textbf{$\bullet$ Question 2:} \label{Q2} Is the map $\bar{L}$ $C^1$ in some neighborhood of $(\br{\psi}_0, \Omega_0)$?

\bigskip

\textbf{$\bullet$ Question 3:} \label{Q3} Is the partial Fr\'{e}chet derivative $\frac{\rd \br{L}}{\rd \br{\psi}}(\br{\psi}_0, \Omega_0) \in L(X, Z)$\footnote{$L(X,Z)$ denotes continuous map from $X$ to $Z$.} isomorphism?

\bigskip
First of all, we compute $\frac{\rd \br{L}}{\rd \br{\psi}}(\br{\psi}_0, \Omega_0)$.
\begin{proposition}
	The Fr\'{e}chet derivative $\frac{\rd \br{L}}{\rd \br{\psi}}(\br{\psi}_0, \Omega_0)$ is given by
	\begin{equation} \label{diff L bar}
		\frac{\rd \br{L}}{\rd \br{\psi}}(\br{\psi}_0, \Omega_0)=\frac{1}{2\mu^2} \Big( \left( \br{\rd}_\varphi \br{\rd}_\varphi+\mu^2 \rd_\phi \rd_\phi \right) \left( \br{\rd}_\beta+2\mu \right)+\left( 2\mu -1 \right) \left( \br{\rd}_\beta + \br{\rd}_\varphi \right) \Big).
	\end{equation}
\end{proposition}

\begin{proof}
	With explicit expression \eqref{trivial solz bar} and the definition of bar derivative \eqref{diff beta bar}, \eqref{diff varphi bar}, we can easily calculate 
	\begin{equation} \label{diff of trivial sol}
		\dbeta_0=-1, \quad \dvarphi_0=1, \quad \dphi_0=0, \quad \dbetaphi_0=0, \quad \dbetavarphi_0=-2\mu
	\end{equation}
	Then, combining the following elementary calculation
	\begin{align*}
		&\frac{\rd}{\rd \br{\psi}} \left( \frac{\dbetavarphi \cdot \left( \dvarphi \right)^{-\frac{1}{2\mu}}}{2\mu} \Omega_0 \right)(\br{\psi}_0, \Omega_0) \\
		&= \frac{\Omega_0}{2\mu} \left( \left( \dvarphi_0 \right)^{-\frac{1}{2\mu}} \left(\br{\rd}_\varphi +1 \right) \br{\rd}_\beta  \, + \dbetavarphi_0 \cdot \left(-\frac{1}{2\mu} \right) \left( \dvarphi_0 \right)^{-\frac{1}{2\mu}-1} \cdot \br{\rd}_\varphi  \right) \\
		\underset{\eqref{diff of trivial sol}}&{=} \frac{2 \mu -1}{2\mu^2} \left( \left(\br{\rd}_\varphi +1 \right) \br{\rd}_\beta+\br{\rd}_\phi \right),
	\end{align*}
	
	\bigskip
	
	\begin{align*}
		& \frac{\rd}{\rd \br{\psi}}\left( \rd_\phi \left(  \frac{ \left( \bar{\rd}_\varphi +1 \right) \br{\rd}_\beta \br{\psi} \cdot  \rd_\phi \br{\psi} - \rd_\phi \br{\rd}_\beta \br{\psi} \cdot \br{\rd}_\varphi \br{\psi}}{2\br{\rd}_\beta \br{\psi}  }\right) \right)(\br{\psi}_0, \Omega_0) \\
		& = \rd_\phi \Bigg( \dbetavarphi_0 \cdot \dphi_0 - \dbetaphi_0 \cdot \dvarphi_0 \Big) \cdot \left( -\frac{\br{\rd}_\beta}{2(\dbeta_0)^2} \right) \\
		& \quad + \frac{1}{2\dbeta_0} \left(  \dbetavarphi_0 \cdot \rd_\phi +  \dphi_0 \cdot (\br{\rd}_\varphi+1)\br{\rd}_\beta -\dbetaphi_0 \cdot \br{\rd}_\varphi -  \dvarphi_0 \cdot \rd_\phi \br{\rd}_\beta \right) \Bigg) \\
		\underset{\eqref{diff of trivial sol}}&{=} \mu \rd_\phi \rd_\phi + \frac{1}{2} \rd_\phi \rd_\phi \br{\rd}_\beta,
	\end{align*}
	
	\bigskip
	
	\begin{align*}
		&\frac{\rd}{\rd \br{\psi}} \left( \br{\rd}_\varphi \left( \frac{2\br{\rd}_\beta \br{\psi} \cdot  \br{\rd}_\varphi \br{\psi}}{\left( \bar{\rd}_\varphi +1 \right) \br{\rd}_\beta \br{\psi}} \left( 1+ \left( \frac{\rd_\phi \br{\rd}_\beta \br{\psi}}{2\br{\rd}_\beta} \right)^2  \right) - \frac{\rd_\phi \br{\rd}_\beta \br{\psi} \cdot  \rd_\phi \br{\psi}   }{2  \br{\rd}_\beta \br{\psi}}  \right) \right)(\br{\psi}_0, \Omega) \\
		&= \br{\rd}_\varphi \Bigg( 2 \dbeta_0 \cdot \dvarphi_0 \left( 1+ \left( \frac{\dbetaphi_0}{2\dbeta_0} \right)^2 \right) \cdot  \frac{-(\br{\rd}_\varphi+1)\br{\rd}_\beta}{\left( \dbetavarphi_0 \right)^2}  \\
		& \qquad \qquad + \frac{2 \dbeta_0 \cdot \dvarphi_0}{\dbetavarphi_0} \cdot \frac{\dbetaphi_0}{\dbeta_0} \cdot \left( \frac{2 \dbeta_0 \cdot \rd_\phi \br{\rd}_\beta - 2\dbetaphi_0 \cdot \br{\rd}_\beta}{(2 \dbeta_0)^2} \right) \\
		& \qquad \qquad + \frac{2 \dvarphi_0}{\dbetavarphi_0} \left( 1+ \left( \frac{\dbetaphi_0}{2 \dbeta_0} \right)^2 \right) \br{\rd}_\beta \\
		& \qquad \qquad + \frac{2 \dvarphi_0}{\dbetavarphi_0} \left( 1+ \left( \frac{\dbetaphi_0}{2 \dbeta_0} \right)^2 \right) \br{\rd}_\varphi \\
		& \qquad \qquad - \frac{1}{2} \dbetaphi_0 \cdot \dphi_0 \left( \frac{-\br{\rd}_\beta}{(\dbeta_0)^2} \right) \\
		& \qquad \qquad -\frac{1}{2} \frac{\dphi_0}{\dbeta_0} \rd_\phi \br{\rd}_\beta \\
		& \qquad \qquad -\frac{1}{2} \frac{\dbetaphi_0}{\dbeta_0} \rd_\phi \\
		\underset{\eqref{diff of trivial sol}}&{=} \frac{1}{2\mu^2} \Big( \br{\rd}_\varphi(\br{\rd}_\varphi+1)\br{\rd}_\beta+2\mu \br{\rd}_\varphi \br{\rd}_\varphi -2\mu \br{\rd}_\varphi  \br{\rd}_\beta \Big),
	\end{align*}
	we have
	\begin{align*}
		\frac{\rd \br{L}}{\rd \br{\psi}}(\br{\psi}_0, \Omega)&=\frac{1}{2\mu^2} \bigg( \br{\rd}_\varphi(\br{\rd}_\varphi+1)\br{\rd}_\beta+2\mu \br{\rd}_\varphi \br{\rd}_\varphi -2\mu \br{\rd}_\varphi  \br{\rd}_\beta \bigg) \\
		& \qquad +\mu \rd_\phi \rd_\phi + \frac{1}{2} \rd_\phi \rd_\phi \br{\rd}_\beta \\
		& \qquad + \frac{2 \mu -1}{2\mu^2} \left( \left(\br{\rd}_\varphi +1 \right) \br{\rd}_\beta+\br{\rd}_\phi \right) \\
		&= \frac{1}{2\mu^2} \Big( \br{\rd}_\varphi \br{\rd}_\varphi \br{\rd}_\beta+2\mu \br{\rd}_\varphi \br{\rd}_\varphi + 2\mu^3 \rd_\phi \rd_\phi +\mu^2 \rd_\phi \rd_\phi \br{\rd}_\beta +(2\mu-1) \br{\rd}_\beta + (2\mu-1)\br{\rd}_\varphi \Big) \\
		&= \frac{1}{2\mu^2} \Big( (\br{\rd}_\varphi \br{\rd}_\varphi+\mu^2 \rd_\phi \rd_\phi)(\br{\rd}_\beta+2\mu)+(2\mu-1)(\br{\rd}_\beta+\br{\rd}_\varphi)  \Big).
	\end{align*}
\end{proof}

\bigskip

\subsection{Wiener type function space on $\br{\bbR}_+ \times \bbT$ and its continuity results}
\label{subsec: Wiener type function space on RT and its continuity results}

The map $\br{L}$ and the Fr\'{e}chet partial derivative $\frac{\rd \br{L}}{\rd \br{\psi}}$ looks horrible and it seems hard to answer \hyperref[Q2]{Question 2} and \hyperref[Q3]{Question 3}. We overcome this difficulty by decomposing Banach space $X, Y, Z$ and various differential operator into each frequency. For example, suppose $M$ and $N$ are some Banach space with $f(\beta,\phi) \in M, \rd_\phi f(\beta,\phi) \in N$. Instead of constructing $M, N$ and showing the continuity of $\rd_\phi: M \rightarrow N$ directly, we do Fourier transform $f(\beta, \phi)$ with respect to $\phi$ variable, i.e., 
\begin{equation*}
	f(\beta, \phi)=\sum_{n \in \bbZ}f^{(n)}(\beta)e^{in \phi},
\end{equation*}
Note that $n$-th Fourier coefficient $f^{(n)}$ is not a constant but a function of $\beta$. Then, at least formally, we can represent $\rd_\phi f$ as
\begin{equation*}
	\rd_\phi f(\beta, \phi)=\sum in \n{f}(\beta)e^{in\phi},
\end{equation*}
so operator $\rd_\phi$ corresponds to multiplication operator $in$ in each Fourier mode. Here, we represent a key strategy for solving the problem; we first construct $\n{M}, \n{N}$ for each $n \in \bbZ$ where we expect $n$-th Fourier coefficient $\n{f}(\beta)$ lives, and establish the (uniform) continuity of multiplication operator $in: \n{f}(\beta) \mapsto in\n{f}(\beta)$. We then construct Banach space $M, N$ via $\n{M}, \n{N}$ and establish continuity of $\rd_\phi$ by combining the continuity of $in: \n{f}(\beta) \mapsto in\n{f}(\beta)$. In this subsection, we gives some definition and properties of some Banach space to make above argument more precise.

Let $\brk{\beta} \coloneqq (1+\beta^2)^\frac{1}{2}$. For $f \in C^\infty (\br{\bbR}_+)$, define $k$-th Schwarz seminorm $\nrm{\cdot}_{\mathcal{S},k}$ for $k \in \bbN$ by
\begin{equation*}
	\nrm{\cdot}_{\calS,k} \coloneqq \sup_{\substack{\beta \in \bbR_+ \\0\leq m \leq k}} \left( \frac{\brk{\beta}^2}{\beta}\right)^m \abs{\rd^{(m)}f(\beta)}.
\end{equation*}
Then, the Schwarz class on $\br{\bbR}_+$ is defined by
\begin{equation*}
	\calS(\br{\bbR}_+) \eqqcolon \set{f \in C^\infty (\br{\bbR}_+) \, \vert \, \nrm{f}_{\calS, k} < \infty \text{ for all } k \in \bbN}.
\end{equation*}
Also, we define $\calS'(\br{\bbR}_+)$ by the dual of $\calS(\br{\bbR}_+)$. Now, we define Schwarz class and tempered distribution on $\br{\bbR}_+ \times \bbT$.

\begin{definition} \label{Schwarz class}
	The \textit{Schwarz class} on $\br{\bbR}_+ \times \bbT$ is defined by
	\begin{equation*}
		\calS(\br{\bbR}_+ \times \bbT) \coloneqq \left\{\sigmaz \n{f}(\beta)e^{in\phi} \, \Bigg\vert \, \nrm{\set{\n{f}}_{n \in \bbZ}}_{q,k} \coloneqq \sigmaz \brk{n}^q \nrm{\n{f}}_{\calS,k} < \infty \text{ for all } q \in \bbR, k \in \bbN\right\}.
	\end{equation*}
\end{definition}
Next, we define dual space above Schwarz class. We additionally impose a condition rather than just defining as a dual of Schwarz class.

\begin{definition} \label{tempered distribution}
	A continuous linear functional $T$ on $\calS(\br{\bbR}_+ \times \bbT)$ is called \textit{tempered distribution} if there exists a sequence of tempered distribution $\set{\n{T}}_{n \in \bbZ} \in \calS'(\br{\bbR}_+)$ such that
	\begin{equation} \label{eq11}
		T \left( \sigmaz \n{g} e^{in\phi} \right) = \sigmaz \n{T} \left(\n{g} \right)
	\end{equation}
	holds for all $g(\beta, \phi)=\sigmaz \n{g}(\beta) e^{in\phi} \in \calS(\br{\bbR}_+ \times \bbT)$ with $\set{\n{T}}_{n \in \bbZ}$ being \textit{polynomially uniformly bounded}, i.e.,  there exists $q \in \bbR, k \in \bbN$ and a constant $C>0$ such that
	\begin{equation} \label{polynomially uniformly bounded}
		\abs{\n{T} \left( \n{g} \right)} \leq C \brk{n}^q \nrm{\n{g}}_{\calS, k} \text{ for all } n \in \bbZ.
	\end{equation}
	
\end{definition}

\begin{remark}
	Summation in \eqref{eq11} is well-defined due to polynomially uniformly bounded condition \eqref{polynomially uniformly bounded} and definition \ref{Schwarz class}.
\end{remark}

\begin{definition}
	Let $\set{\n{X}}_{n \in \bbZ}$ be a sequence of Banach space that is embedded in Banach space $X \in \calS'(\br{\bbR}_+)$. We say $\n{X}$\footnote{We abbreviate $\set{\n{X}}_{n \in \bbZ}$ to $\n{X}$ when it is clear that it denotes a sequence rather than $n$-th component of sequence.} is $\brk{n}^q$ -\textit{uniformly embedded} in $X$ if there exists $M>0$ such that $\nrm{id}_{[\n{X},X]}<M\brk{n}^q$\footnote{$\nrm{T}_{[A,B]}$ represents the operator norm of $T: A \rightarrow B$.} for all $n \in \bbZ$. In particular if $\n{X}$ is $\brk{n}^0$- uniformly embedded in $X$, we simply say $\n{X}$ is \textit{uniformly embedded} in $X$ and denote it by
	\begin{equation*}
		\n{X} \overset{\sim}{\hookrightarrow}X.
	\end{equation*}
\end{definition}

\begin{remark} \label{tempered distribution from a function}
	A natural and concrete way to construct a tempered distribution is to take a function $\n{f}$ on each Banach space $\n{X}$ which is $\brk{n}^s$-uniformly embedded in $X \in \calS'(\br{\bbR}_+)$ for some $s \in \bbR$\footnote{By regarding a function $f \in X$ as tempered distribution for some Banach space $X$ which consists of function defined on $\br{\bbR}_+$ (e.g., $L^p(\br{\bbR}_+)$), it means there exists $C>0, k \in \bbN$ so that for all $g \in \calS(\br{\bbR}_+)$, $\brk{f,g}=\int_{\br{\bbR}_+} \br{f}g$ with $\abs{\brk{f,g}} \leq C \nrm{f}_X \nrm{g}_{\calS, k}$.}, with \textit{polynomial norm decreasing condition}:
	\begin{equation} \label{polynomial norm decreasing condition}
		\sup_{n \in \bbZ} \,\brk{n}^q \nrm{\n{f}}_{\n{X}} < \infty \text{ for some } q \in \bbR.
	\end{equation}
	Then, for $g(\beta, \phi)=\sigmaz \n{g}(\beta)e^{in\phi} \in \calS(\br{\bbR}_+ \times \bbT)$,
	\begin{equation*}
		\brk{\set{\n{f}}_{n \in \bbZ},g} \coloneqq \sigmaz \brk{\n{f}, \n{g}}.
	\end{equation*}
	Note that $\set{\n{f}}_{n \in \bbZ}$ satisfies polynomially uniformly bounded condition since for all $n \in \bbZ$,
	\begin{align*}
		&\abs{\brk{\n{f}, \n{g}}} \leq C\nrm{\n{f}}_X \nrm{\n{g}}_{\calS,k} \underset{\substack{\brk{n}^s-\text{uniform} \\ \text{embeding}}}{\leq} CM\brk{n}^s \nrm{\n{f}}_{\n{X}} \nrm{\n{g}}_{\calS, k} \\
		&= CM \brk{n}^{s-q} \brk{n}^q \nrm{\n{f}}_{\n{X}} \nrm{\n{g}}_{\calS,k} \leq \underbrace{CM \left( \sup_{n \in \bbZ} \brk{n}^q  \nrm{\n{f}}_{\n{X}} \right)}_{\text{A constant}} \brk{n}^{s-q} \nrm{\n{g}}_{\calS, k}.
	\end{align*}
\end{remark}

Let $\calS(\br{\bbR}_+)^\bbZ, \calS'(\br{\bbR}_+)^\bbZ$ be the set of all sequence whose element is in $\calS(\br{\bbR}_+), \calS'(\br{\bbR}_+)$ respectively. Then, how to define Fourier transform on $\calS(\br{\bbR}_+ \times \bbT)$ and $\calS'(\br{\bbR}_+ \times \bbT)$ is obvious from their definition.

\begin{definition} \label{FT}
	Let $f(\beta, \phi)=\sigmaz \n{f}(\beta) e^{in \phi} \in \calS(\br{\bbR}_+ \times \bbT)$ and $T=\set{\n{T}}_{n \in \bbZ} \in \calS'(\br{\bbR}_+ \times \bbT)$. Then, Fourier transform $\calF: \calS(\br{\bbR}_+ \times \bbT) \rightarrow \calS(\br{\bbR}_+)^\bbZ$ and $\calF: \calS'(\br{\bbR}_+ \times \bbT) \rightarrow \calS'(\br{\bbR}_+)^\bbZ$ is defined by
	\begin{equation*}
		\calF(f)(n)=\n{f}(\beta), \qquad \calF(T)(n)=\n{T}.
	\end{equation*}
\end{definition}

For convenience, when we write $\n{f}(\beta)$ for some function $f(\beta, \phi)$, it always means $n$-th Fourier coefficient of $f(\beta, \phi)$.

It is trivial from Definition \ref{FT} that $\calF$ is injective. When we try to define inverse Fourier transform, some care is needed due to the norm conditions in Definition \ref{Schwarz class} and \ref{tempered distribution}. Therefore, one can only define inverse transform $\calF^{-1}$ not on whole $\calS(\br{\bbR}_+)^\bbZ$ or $\calS'(\br{\bbR}_+)^\bbZ$, but on $\calF \left( \calS \left( \br{\bbR}_+ \right) \right), \calF \left( \calS' \left( \br{\bbR}_+ \right) \right)$ respectively.

Now, we are ready to define an important space which is main space of this paper.

\begin{definition} [$\calA^s$-space] \label{A^s}
	Let $\set{\n{X}}_{n \in \bbZ}$ be a sequence of Banach space that is $\brk{n}^q$-uniformly embedded in $X \in \calS'(\br{\bbR}_+)$ for some $q \in \bbR$. Then, the space $\calA^s \left( \set{\n{X}}_{n \in \bbZ} \right)$ is defined by
	\begin{align*}
		\calA^s \left( \set{\n{X}}_{n \in \bbZ} \right) \coloneqq \Bigg\{T=\set{\n{T}}_{n \in \bbZ} \in \calS'(\br{\bbR}_+ \times \bbT) \, \Bigg\vert \, \n{T} \in \n{X} \text{ for all }n \in \bbZ \text{ with } \\
		\nrm{T}_{\calA^s \left( \set{\n{X}}_{n \in \bbZ} \right)} \coloneqq \sigmaz \brk{n}^s \nrm{\n{T}}_{\n{X}} < \infty \Bigg\}.
	\end{align*}
\end{definition}
Since $\calA^s \left( \set{\n{X}}_{n \in \bbZ} \right)$ is always associated with a sequence of Banach space not with single Banach space, we simply write $\calA^s \left( \n{X} \right)$ instead of $\calA^s \left( \set{\n{X}}_{n \in \bbZ} \right)$ without confusion.

\begin{remark}
	Polynomially uniformly bounded condition \eqref{polynomially uniformly bounded} for tempered distribution is automatically satisfied since
	\begin{equation*}
		\brk{\n{T}, \n{g}} \leq \nrm{\n{T}}_{X} \nrm{\n{g}}_{\calS, k} \leq M \brk{n}^q \nrm{\n{T}}_{\n{x}} \nrm{\n{g}}_{\calS, k} = \underbrace{M \brk{n}^s \nrm{\n{T}}_{\n{X}}}_{\text{uniformly bounded w.r.t }n} \cdot \brk{n}^{q-s}\nrm{\n{g}}_{\calS, k}
	\end{equation*}
\end{remark}

The first thing we should do is to show $\calA^s \left( \n{X} \right)$ is really a Banach space. We need the following lemma.

\begin{lemma} \label{induced Banach space}
	Let $G: X \rightarrow Y$ be an injective linear map from a Banach space $X$ to a vector space $Y$. Then, GX, the image of X by G, is a Banach space with induced norm
	\begin{equation*}
		\nrm{Gx}_{GX} \coloneqq \nrm{x}_{X}.
	\end{equation*}
	With this induced norm, $G: X \rightarrow GX$ becomes isometric isomorphism.
\end{lemma}

\begin{proposition} \label{A^s is Banach}
	Let $\set{\n{X}}_{n \in \bbZ}$ be a sequence Banach space which is $\brk{n}^q$-uniformly imbedded in $X \in \calS'\left( \br{\bbR}_+ \right)$ for some $q \in \bbR$. Then, $\calA^s \left( \n{X} \right) \in \calS' \left( \br{\bbR}_+ \times \bbT \right)$ is a Banach space for all $s \in \bbR$ equipped with $\calA^s \left( \n{X} \right)$-norm given in Definition \ref{A^s}.
\end{proposition}
\begin{proof}
	Note that $\set{\n{X}}_{n \in \bbZ}$ is a vector space with component-wise addition and scalar multiplication. Then,
	\begin{equation*}
		\ell^1 \left( \set{\n{X}}_{n \in \bbZ} \right) \coloneqq \set{\set{\n{f}}_{n \in \bbZ} \in \set{\n{X}}_{n \in \bbZ} \, \Bigg\vert \, \nrm{\set{\n{f}}_{n \in \bbZ}}_{\ell^1 \left( \set{\n{X}}_{n \in \bbZ} \right)} \coloneqq \sigmaz \nrm{\n{f}}_{\n{X}} < \infty}
	\end{equation*}
	is a Banach space with the same proof as Riesz-Fischer theorem.
	
	Consider a component-wise multiplication operator $M_{\brk{n}^{-s}} \rightarrow \calF \left( \calS' \left( \br{\bbR}_+ \times \bbT \right) \right)$ defined by
	\begin{equation*}
		M_{\brk{n}^{-s}} \left( \set{\n{f}}_{n \in \bbZ} \right) \coloneqq \set{\brk{n}^{-s} \n{f}}_{n \in \bbZ}.
	\end{equation*}
	It is trivial that $M_{\brk{n}^{-s}}$ is injective. To show range of $M_{\brk{n}^{-s}}$ is contained in $\calF \left( \calS' \left( \br{\bbR}_+ \times \bbT \right) \right)$, it suffices to check polynomial norm decreasing condition \eqref{polynomial norm decreasing condition} for all $\set{\n{f}}_{n \in \bbZ} \in \ell^1 \left( \set{\n{X}}_{n \in \bbZ} \right)$. Since $\set{\n{f}}_{n \in \bbZ} \in \ell^1 \left( \set{\n{X}}_{n \in \bbZ} \right)$ we have $\sup_{n \in Z}\nrm{\n{f}}_{\n{X}} < \infty$. Hence
	\begin{equation*}
		\sup_{n \in \bbZ} \brk{n}^s \nrm{\brk{n}^{-s}\n{f}}_{\n{X}} = \sup_{n \in \bbZ} \nrm{\n{f}}_{\n{X}} < \infty,
	\end{equation*}
	so $\set{\brk{n}^{-s} \n{f}}_{n \in \bbZ}$ satisfies norm-decreasing condition \eqref{polynomial norm decreasing condition}.
	
	Since
	\begin{equation*}
		\calA^s \left( \n{X}  \right) = \left( \calF^{-1} \, \circ \, M_{\brk{n}^{-s}} \right) \left( \ell^1 \left( \set{\n{X}}_{n \in \bbZ} \right) \right)
	\end{equation*}
	and $M_{\brk{n}^{-s}}: \ell^1 \left( \set{\n{X}}_{n \in \bbZ} \right) \rightarrow \calF \left( \calS' \left( \br{\bbR}_+ \times \bbT \right) \right), \calF^{-1}: \calF \left( \calS' \left( \br{\bbR}_+ \times \bbT \right) \right) \rightarrow  \calS' \left( \br{\bbR}_+ \times \bbT \right)$ are injective, the result follows by Lemma \ref{induced Banach space}.
\end{proof}

Next, we investigate some continuity results regarding on $\calA^s$-space.

\begin{definition} \label{uniformly bounded linear map}
	Let $\n{G}: \n{X} \rightarrow \n{Y}$ be a sequence of linear map where $\set{\n{X}}_{n \in \bbZ}, \set{\n{Y}}_{n \in \bbZ}$ are sequences of normed vector spaces. We say $\set{\n{G}}_{n \in \bbZ}$ is \textit{uniformly bounded} if there exists $M>0$ such that
	\begin{equation*}
		\nrm{\n{G}}_{[\n{X}, \n{Y}]} \leq M
	\end{equation*}
	for all $n \in \bbZ$.
\end{definition}

\begin{proposition} \label{continuity between A^s}
	Let $\n{G}: \n{X} \rightarrow \n{Y}$ be a sequence of uniformly bounded linear map where $\set{\n{X}}_{n \in \bbZ}, \set{\n{Y}}_{n \in \bbZ}$ are sequences of Banach space which satisfies the assumption in \eqref{A^s}. Then, $\set{\n{G}}_{n \in \bbZ}$ induces bounded linear map
	\begin{equation*}
		G: \calA^s \left( \n{X} \right) \rightarrow \calA^s \left( \n{Y} \right)
	\end{equation*}
	defined by $G= \calF^{-1} \, \circ \, \n{G} \, \circ \, \calF$ for all $s \in \bbR$.
\end{proposition}

\begin{proof}
	Almost trivial. For $T=\set{\n{T}}_{n \in \bbZ} \in \calA^s \left( \n{X} \right)$, the map $G$ can be written explicitly by $G\left( \set{\n{T}}_{n \in \bbZ} \right) = \set{\n{G} \left( \n{T} \right)}_{n \in \bbZ}$, and through uniformly bounded condition on $\n{G}$,
	\begin{equation*}
		\nrm{G(T)}_{\calA^s \left( \n{Y} \right)} = \sigmaz \brk{n}^s \nrm{\n{G} \left( \n{T} \right)}_{\n{Y}} \leq \sigmaz M \brk{n}^s \nrm{\n{T}}_{\n{X}}=M \nrm{T}_{\calA^s \left( \n{X} \right)}.
	\end{equation*}
\end{proof}

\begin{remark} \label{concrete element}
	Although we've defined $\calS \left( \br{\bbR}_+ \times \bbT \right), \calS' \left( \br{\bbR}_+ \times \bbT \right), \calA^s \left( \n{X} \right)$ in rather abstract and formal way, what we have in mind is very concrete as in the introduction paragraph of Section \ref{subsec: Wiener type function space on RT and its continuity results}. The specific element we intend to be in $\calA^s \left( \n{X} \right)$ is the function of the form
	\begin{equation*}
		f(\beta, \phi)= \sigmaz \n{f}(\beta)e^{in \phi}
	\end{equation*}
	which satisfies $\sigmaz \brk{n}^s \nrm{\n{f}(\beta)}_{\n{X}} < \infty$, where $\n{X}$ are some (intricate) Banach space which is embedded in a set of continuous bounded function on $[0, \infty)$. A function $f(\beta, \phi)$ can be regarded as tempered distribution by defining
	\begin{equation*}
		\brk{f, \sigmaz \n{g} (\beta) e^{in \phi}}=\brk{\sigmaz \n{f} (\beta) e^{in \phi},\sigmaz \n{g} (\beta) e^{in \phi}}=\sigmaz \int_0^\infty \br{\n{f}(\beta)} \n{g}(\beta)\, d\beta.
	\end{equation*}
	Note that $f$ satisfies polynomially uniformly bounded condition \eqref{polynomially uniformly bounded} as
	\begin{align*}
		\abs{\int_0^\infty \n{f}(\beta) \n{g}(\beta)} & \leq \int_0^\infty \brk{n}^s \abs{\n{f}(\beta)} \brk{n}^{-s} \abs{\n{g}(\beta)} \\
		& \leq \underbrace{\left( \sigmaz \brk{n}^s \nrm{\n{f}(\beta)}_{\n{X}} \right)}_{\text{independent of } n} \brk{n}^{-s} \int_0^\infty \left( \frac{\brk{\beta}^2}{\beta} \right)^{-2} \underbrace{\left( \frac{\brk{\beta}^2}{\beta} \right)^{2} \abs{\n{g}(\beta)}}_{\leq \nrm{\n{g}}_{\calS, 2}} \, d\beta \\
		& \leq \underbrace{\Bigg( \left( \sigmaz \brk{n}^s \nrm{\n{f}(\beta)}_{\n{X}} \right) \underbrace{\int_0^\infty \left( \frac{\brk{\beta}^2}{\beta} \right)^{-2}\, d\beta}_{< \infty} \Bigg)}_{\text{constant independent of }n} \brk{n}^{-s} \nrm{\n{g}}_{\calS, 2}.
	\end{align*}
	Also, considering the definition of derivatives on tempered distribution, we have
	\begin{equation*}
		\rd_\beta \left( \sigmaz \n{f} (\beta) e^{in \phi} \right)=\sigmaz \n{f}_\beta (\beta) e^{in \phi}, \qquad \rd_\phi \left( \sigmaz \n{f} (\beta) e^{in \phi} \right)=\sigmaz in \n{f} (\beta) e^{in \phi}
	\end{equation*}
	in the sense of tempered distribution, and if regularity of $\set{\n{f}(\beta)}_{n \in \bbZ}$ is good enough, this will be true as a function. In the view of proposition \ref{continuity between A^s}, $\rd_\beta$ is induced by $\set{\rd_\beta}_{n \in \bbZ}$ and $\rd_\phi$ is induced by multiplication operator $\set{M_{in}}_{n \in \bbZ}$.
\end{remark}

From now on, we omit `$M$' in multiplication operator for simplicity, so we use $\set{in}_{n \in \bbZ}$ instead of $\set{M_{in}}_{n \in \bbZ}$. Next, we investigate multiplication property of $\calA^s$-space.

\begin{proposition} \label{A^s algebra}
	Let $\set{\n{X}}, \set{\n{Y}}, \set{\n{Z}}$ be sequences of Banach space which is $\brk{n}^q$-uniformly embedded in $X, Y, Z \in \calS'(\br{\bbR}_+)$ respectively for some $q \in \bbR$. Also, suppose $X, Y, Z$ consist of functions on $\br{\bbR}_+$ (cf. Remark \ref{tempered distribution from a function}) and
	\begin{equation*}
		\ui{X} \cdot \uj{Y} \unibed \uij{Z},
	\end{equation*}
	i.e., there exists $M>0$ such that for all $\ui{f} \in \ui{X}$ and $\uj{g} \in \uj{Y}$, we have $\ui{f} \uj{g} \in \uij{Z}$ with $\nrm{\ui{f} \uj{g}}_{\uij{Z}} \leq M \nrm{\ui{f}}_{\ui{X}} \nrm{\uj{g}}_{\uj{Y}}$ for all $i,j \in \bbZ$. Then, for all $s \geq 0$,
	\begin{enumerate}
		\item For $f \in \calA^s \left( \n{X} \right), g \in \calA^s \left( \n{Y} \right)$, $fg \in \calS' \left( \br{\bbR}_+ \times \bbT \right)$ defined by \begin{equation} \label{eq12}
			\left( fg \right)^{(n)} \coloneqq \sum_{k \in \bbZ} f^{(n-k)}g^{(k)}
		\end{equation} belongs to $\calA^s \left( \n{Z} \right)$ with norm bound \begin{equation*}
			\nrm{fg}_{\calA^s \left( \n{Z} \right)} \leq \sqrt{2}M \nrm{f}_{\calA^s \left( \n{X} \right)} \nrm{g}_{\calA^s \left( \n{Y} \right)}.
		\end{equation*}\\
		\item For $f \in \calA^s \left( \n{X} \right), g \in \calA^{-s} \left( \n{Y} \right)$, $fg \in \calS' \left( \br{\bbR}_+ \times \bbT \right)$ defined by \begin{equation} \label{eq13}
			\left( fg \right)^{(n)} \coloneqq \sum_{k \in \bbZ} f^{(n-k)}g^{(k)}
		\end{equation} belongs to $\calA^{-s} \left( \n{Z} \right)$ with norm bound \begin{equation*}
			\nrm{fg}_{\calA^{-s} \left( \n{Z} \right)} \leq \sqrt{2}M \nrm{f}_{\calA^s \left( \n{X} \right)} \nrm{g}_{\calA^{-s} \left( \n{Y} \right)}. 
		\end{equation*}
	\end{enumerate}
\end{proposition}

\begin{proof}
	We first prove (1). For $f \in \calA^s \left( \n{X} \right)$ and $g \in \calA^s \left( \n{Y} \right)$,
	\begin{align*}
		\nrm{fg}_{\calA^s(\n{Z})} &= \sigmaz \brk{n}^s \nrm{\n{(fg)}}_{\n{Z}} \\
		\underset{\eqref{eq12}}&{\leq} \sigmaz \brk{n}^s \sum_{k \in \bbZ} \nrm{f^{(n-k)}g^{(k)}}_{\n{Z}} \\
		\underset{\ui{X} \cdot \uj{Y} \unibed \uij{Z}}&{\leq} M \sigmaz \sum_{k \in \bbZ} \brk{n}^s \nrm{f^{(n-k)}}_{X^{(n-k)}} \nrm{g^{(k)}}_{Y^{(k)}} \\
		\underset{\brk{n} \leq \sqrt{2}\brk{n-k}\brk{k}}&{\leq} \sqrt{2}M \sigmaz \sum_{k \in \bbZ} \brk{n-k}^s \nrm{f^{(n-k)}}_{X^{(n-k)}} \brk{k}^s \nrm{g^{(k)}}_{Y^{(k)}} \\
		&\leq \sqrt{2}M \nrm{f}_{\calA^s \left( \n{X} \right)} \nrm{g}_{\calA^s \left( \n{Y} \right)}.
	\end{align*}
	Since $\nrm{f}_{\calA^s \left( \n{X} \right)} < \infty$ and $\nrm{g}_{\calA^s \left( \n{Y} \right)}<\infty$, we derive $fg$ defined by \eqref{eq12} is well-defined tempered distribution as infinite sum converges absolutely to an element of $\n{Z}$.
	
	Proof of (2) is similar; just use $\brk{n}^{-1} \leq \sqrt{2} \brk{n-k} \brk{k}^{-1}$ instead of  $\brk{n} \leq \sqrt{2}\brk{n-k}\brk{k}$.
\end{proof}

\bigskip

\subsection{Suited function space}
\label{subsec: Suited function space}

In this subsection, we introduce various Banach spaces which will provide the answer for \hyperref[Q1]{Question 1}. At the first glance, these spaces might look artificial, as we will not use familiar Banach spaces such as Lebesgue space or Sobolev space. However, the reasons why we have to choose these unintuitive space become clear.

Suppose $X$ and $Y$ are Banach spaces that is included in some vector space $V$. Then, $X \cap Y$ is also a Banach space with its norm
\begin{equation*}
	\nrm{a}_{X \cap Y} \coloneqq \maxi{\nrm{a}_X, \nrm{a}_Y}.
\end{equation*}  
Also, if $X \cap Y = \0$, $X \oplus Y$ is a Banach space with its norm
\begin{equation*}
	\nrm{x+y}_{X \oplus Y} \coloneqq \nrm{x}_{X}+\nrm{y}_{Y}.
\end{equation*}
Even though $X$ and $Y$ are not disjoint, $X+Y$ is still a Banach space with its norm
\begin{equation*}
	\nrm{a}_{X+Y}=\inf \left\{ \nrm{x}_{X} + \nrm{y}_{Y}\, \vert \, a=x+y, x \in X, y \in Y \right\}.
\end{equation*}
The function $\xi_\infty:[0,\infty) \rightarrow [0,1]$ is a smooth function which is 0 near the origin and 1 near the infinity. More specific conditions imposed on $\xi_\infty$ will be clarified later. Also, we define $\xi_0$ to be $1-\xi_\infty$. Now we are ready to introduce important Banach spaces which will be main characters of this paper;

\begin{align}
	C_b & \coloneqq \set{f: [0, \infty) \rightarrow \bbC \, \Big\vert \, f \text{ is continuous with } \nrm{f}_{C_b} \coloneqq \sup_{\beta \in [0, \infty)} \abs{f(\beta)} < \infty} \label{Cb}, \\
	C_b^\delta &\coloneqq \beta^\delta C_b \cap \beta^{-\delta}C_b = \set{f \in C_b \, \Big\vert \, \nrm{f}_{C_b^\delta} \coloneqq \maxi{\nrm{\beta^\delta f}_{C_b}, \nrm{\beta^{-\delta} f}_{C_b}} < \infty} \label{Cbd}, \\
	\n{W_-} & \coloneqq C_b^\delta \oplus \delta_{0n} \bbC \xi_0 \coloneqq \begin{cases}
		C_b^\delta \oplus \bbC \xi_0 & (n=0) \\
		C_b^\delta & (n \neq 0) 
	\end{cases}, \label{Wnm} \\
	\n{W_0} & \coloneqq C_b^\delta \oplus \bbC \xi_0 \oplus \delta_{0n} \bbC \label{Wnz}, \\
	\n{W_+} & \coloneqq C_b^\delta \oplus \bbC \xi_0 \oplus \bbC \xi_\infty, \label{Wnp} \\
	\Box_{,N} & \coloneqq \begin{cases}
		\Box & (n=N\bbZ) \\
		0 & (n \neq N\bbZ)
	\end{cases}, \qquad \text{where } \Box=C_b, C_b^\delta, \n{W}_-, \n{W}_0, \n{W}_+. \label{N periodic}
\end{align}
Note that
\begin{equation} \label{unibed chain}
	C_b^\delta \unibed \n{W}_- \unibed \n{W}_0 \unibed \n{W}_+ \unibed C_b \in \calS'(\bbR_+),
\end{equation}
where $\n{X} \unibed \n{Y}$ means $\n{X} \subset \n{Y}$ and $id_{[\n{x}, \n{Y}]}$ is uniformly bounded for all $n \in \bbZ$. Above spaces have multiplication property also.

\begin{proposition} \label{W^n algebra}
	For all $N \geq 0$,
	\begin{align}
		\ui{W}_{-,N} \cdot \uj{W}_{+,N} & \unibed \uij{W}_{-,N}, \label{eq14} \\
		\ui{W}_{0,N} \cdot \uj{W}_{0,N} & \unibed \uij{W}_{0,N},\label{eq15} \\
		\ui{W}_{+,N} \cdot \uj{W}_{+,N} & \unibed \uij{W}_{+,N}. \label{eq16} 
	\end{align}
\end{proposition}

\begin{proof}
	Write an element of $W^{(i)}_{\triangle, N}$, where $\triangle$ could be either $-, 0, +$, explicitly and use elementary inequalities
	\begin{equation*}
		\nrm{fg}_{C_b} \leq \nrm{f}_{C_b} \nrm{g}_{C_b},
	\end{equation*}
	\begin{align*}
		\ncbd{f} &= \maxi{\nrm{\beta^\delta f}_{C_b}, \nrm{\beta^{-\delta}f}_{C_b}} \\
		& = \sup_{\beta \in [0, \infty)} \underbrace{\maxi{\beta^\delta, \beta^{-\delta}}}_{\geq 1} \abs{f(\beta)} \\
		& \geq \nrm{f}_{C_b}.
	\end{align*}
\end{proof}

Since $\n{W_{\triangle, N}}$ is uniformly embedded in $C_b$, we can define the following space.

\begin{definition}
	\begin{align*}
		W_{-, N} &\coloneqq  \calA^{0.5} \left( \Wnm \right), \\
		W_{0, N} &\coloneqq  \calA^{0.5} \left( \Wnz \right), \\
		W_{+, N} &\coloneqq  \calA^{0.5} \left( \Wnp \right). 
	\end{align*}
\end{definition}

For three Banach space $X, Y$ and $Z$, we write $X \cdot Y \hookrightarrow Z$ if there exists $M>0$ such that for all $f \in X, g \in Y$, we have $fg \in Z$ with $\nrm{fg}_{Z} \leq M \nrm{f}_{X} \nrm{g}_{Y}$ holds. Using this notation, we can prove the following \textit{algebra} property of $W_{\triangle, N}$.

\begin{proposition} \label{W algebra}
	For all $N \geq 0$,
	\begin{align}
		W_{-,N} \cdot W_{+,N} \hookrightarrow W_{-,N}, \label{pm embedding} \\
		W_{0,N} \cdot W_{0,N} \hookrightarrow W_{0,N},  \\
		W_{+,N} \cdot W_{+,N} \hookrightarrow W_{+,N}. 
	\end{align}
\end{proposition}

\begin{proof}
	Proposition \ref{A^s algebra} and Proposition \ref{W^n algebra}.
\end{proof}

Now, we are ready to construct Banach space to where we apply implicit function theorem to $\overline{L}: X \times Y \rightarrow Z$ defined as in \eqref{L bar}. We have to choose $X, Y, Z$ so that $\overline{L}$ is $C^1$ and $\frac{\rd \overline{L}}{\rd \overline{\psi}} \left( \br{\psi}_0, \Omega_0 \right)$ is an isomorphism. As isomorphism condition is more strict than to just require its Fr\'echet derivative to be bounded, we construct $X, Y, Z$ so that we easily guarantee $\frac{\rd \br{L}}{\rd \psi}(\br{\psi}_0, \Omega_0)$ is an isomorphism.
		To utilize Proposition \ref{continuity between A^s}, we first decompose partial Fr\'echet derivative \eqref{diff L bar} into each Fourier mode. From now on, if a map $T: \calA^s \left( \n{X} \right) \rightarrow \calA^s \left( \n{Y} \right)$ is induced by sequence of uniformly bounded map $\n{G}$ as in proposition \ref{continuity between A^s}, we use notation $\n{T}$ to represent $\n{G}$. Then, differential operators on $\calA^s \left( \n{X} \right)$ were decomposed into
\begin{align}
	\n{\rd_\beta} &= \rd_\beta \label{diff beta n}, \\
	\n{\rd_\phi} &= in \label{diff phi n}, \\
	\n{\rd_\varphi} &= in-\rd_\beta \label{diff varphi n}, \\
	\n{\br{\rd}_\beta} \underset{\eqref{diff beta bar}}&{=} \n{\left( \beta \rd_\beta+(1-2\mu)\right)} \underset{\eqref{diff beta n}}{=} \beta \rd_\beta+(1-2\mu) \label{diff beta bar n}, \\
	\n{\br{\rd}_\varphi} \underset{\eqref{diff varphi bar}}&{=} \n{\left(\beta \rd_\varphi+(2\mu-1)\right)} \underset{\eqref{diff varphi n}}{=} \beta in -\beta \rd_\beta + (2\mu -1) \label{diff varphi bar n},\\
	\n{\left( \left( \br{\rd}_\varphi+1 \right)\br{\rd}_\beta \right)} 
		\underset{\eqref{diff beta bar}, \eqref{diff varphi bar}}&{=} \n{\left( (\beta \rd_\varphi +2\mu)(\beta \rd_\beta +(1-2\mu) \right)} \nonumber \\
		&= \n{\left( \beta \rd_\varphi \beta \rd_\beta + (1-2\mu) \beta \rd_\varphi + 2\mu \beta \rd_\beta +2\mu (1-2\mu) \right)} \nonumber \\
		&= (\beta in-\beta \rd_\beta)\beta \rd_\beta +(1-2\mu)(\beta in-\beta \rd_\beta)+2\mu \beta \rd_\beta +2\mu(1-2\mu) \label{diff long bar n}, \\
		\n{\left( \rd_\phi \br{\rd}_\beta \right)} \underset{\eqref{diff beta bar}}&{=} \n{\left( \rd_\phi (\beta \rd_\beta +1-2\mu) \right)} \underset{\eqref{diff phi n}}{=} in \beta \rd_\beta + in(1-2\mu). \label{diff short bar n}
\end{align}

Then, we can decompose \eqref{diff L bar} to each Fourier mode using \eqref{diff beta bar n} and \eqref{diff varphi bar n};
\begin{align*}
	\n{\left( \frac{\rd \br{L}}{\rd \br{\psi}}\left( \br{\psi}_0, \Omega_0 \right) \right)} \underset{\eqref{diff L bar}}&{=} \frac{1}{2\mu^2} \n{\left( (\br{\rd}_\varphi \br{\rd}_\varphi +\mu^2 \rd_\phi \rd_\phi)(\br{\rd}_\beta +2\mu)+(2\mu-1)(\br{\rd}_\beta +\br{\rd}_\varphi )  \right)} \\
	\underset{\eqref{diff beta bar n}, \eqref{diff varphi bar n}}&{=} \frac{1}{2\mu^2} \left( \left( \left( \beta \left( in-\rd_\beta \right)+ \left( 2\mu -1 \right) \right)^2 -\mu^2 n^2 \right) \left( \beta \rd_\beta +1 \right) + \left( 2\mu-1\right) in \beta \right) \\
	&= \frac{1}{2\mu^2} ( \left( \beta \left( \rd_\beta-in \right)- \left( \left( 2+n \right) \mu -1 \right) \right) \left( \beta \left( \rd_\beta -in \right) - \left( \left( 2-n \right) \mu-1 \right) \right) \left( \beta \rd_\beta +1 \right) \\
	&+\left( 2\mu-1 \right) in \beta )
\end{align*}

Observation of \eqref{diff beta bar n}, \eqref{diff varphi bar n}, \eqref{diff long bar n}, \eqref{diff short bar n} and $\n{\left( \frac{\rd \br{L}}{\rd \br{\psi}}\left( \br{\psi}_0, \Omega_0 \right) \right)}$ reveals similar expression of the following form.

\begin{definition} \label{operator dns}
	For $n \in \bbZ, s \in \bbR$, the operator $\dns{n}{s}$ is defined by
	\begin{equation} \label{dns}
		\dns{n}{s} \coloneqq \beta \left( \rd_\beta - in \right) -s.
	\end{equation}
\end{definition}

To simplify expression more, we use the following notation.
\begin{align}
	\n{P} & \coloneqq \dns{n}{0} = \beta \left( \rd_\beta -in \right) \label{Pn}, \\
	\n{P_+} & \coloneqq \dns{n}{(2+n)\mu -1} = \beta \left( \rd_\beta - in \right) - \left( \left( 2+n \right) \mu -1\right) \label{Pnp}, \\
	\n{P_-} & \coloneqq \dns{n}{(2-n)\mu -1} = \beta \left( \rd_\beta - in \right) - \left( \left( 2-n \right) \mu -1\right) \label{Pnm}, \\
	Q & \coloneqq \dns{0}{0}= \beta \rd_\beta \label{Q}.
\end{align}

Then, above operators are represented simply by
\begin{align}
	\n{\left( \beta \rd_\beta \right)} &= Q \label{eq17}, \\
	\n{\left( \beta \rd_\varphi \right)} &= -\n{P} \label{eq18}, \\
	\n{\left( \br{\rd}_\beta \right)} &= Q+ (1-2\mu)id \label{eq19},\\
	\n{\left( \br{\rd}_\varphi \right)} &= -\n{P}+(2\mu-1)id \label{eq20},\\
	\n{\left( \rd_\phi \right)} &= in  \cdot id,\label{eq21} \\
	\n{\left( \left( \br{\rd}_\varphi +1 \right) \br{\rd}_\beta \right)} &= -\n{P}Q+(2\mu-1) \n{P} +2\mu Q + 2\mu (1-2\mu) id,\label{eq22} \\
	\n{\left( \rd_\phi \br{\rd}_\beta \right)} &= inQ+in(1-2\mu)id, \label{eq23} \\
	\n{\left( \frac{\rd \br{L}}{\rd \br{\psi}}\left( \br{\psi}_0, \Omega_0 \right) \right)} &= \frac{1}{2\mu^2} \left( \n{P_+} \n{P_-} \left( Q+1 \right) + \left(2\mu -1 \right) \left( Q-\n{P} \right) \right). \label{diff L bar n}
\end{align}

To apply implicit function theorem, we have to show that \eqref{diff L bar n} is an isomorphism between some Banach space, but observe that \eqref{diff L bar n} consists of sum of two linear map. To prove sum of linear map is isomorphism, we will show that $\n{P_+} \n{P_-} (Q+1)$ is isomorphism and operator norm of $(2\mu -1) \left( Q-\n{P} \right)$ could be small if we choose $N$ large in \eqref{N periodic}. However, how can we prove $\n{P_+} \n{P_-} (Q+1)$ is an isomorphism? Actually, it is more precise to say that we do not prove it, but construct space so that $\n{P_+} \n{P_-} (Q+1)$ have no choice but to be an isometric isomorphism. The key is Lemma \ref{induced Banach space}, but to apply it, we have to verify that the operator is injective. Therefore, we need to investigate the property of operator $\dns{n}{s}$.

\bigskip

\subsection{Properties of operator $D^{(n, s)}$}
\label{subsec: Properties of operator dns}

In this subsection, we investigate the properties of operator
\begin{equation*}
	\dns{n}{s}=\beta \left( \rd_\beta \ - in \right) -s.
\end{equation*}
First, to utilize Lemma \ref{induced Banach space}, we check the injectivity of $\dns{n}{s}$.

\begin{proposition} \label{injective on C_b}
	For $s \neq 0$, $\dns{n}{s}$ is injective on $C_b$.\footnote{Note that derivative is defined in distribution sense.}
\end{proposition}
\begin{proof}
	An ODE
	\begin{equation*}
		\dns{n}{s}f=\beta\rd_\beta f-\beta inf-sf=0
	\end{equation*}
	has a general solution
	\begin{equation*}
		f(\beta)=C \beta^s e^{in\beta}
	\end{equation*}
	for $C \in \bbR$. Due to $\beta^s$ -term for $s \neq 0$, $f$ belongs to $C_b$ if and only if $c=0$. 
\end{proof}
Since $\dns{n}{s}$ is injective whenever $s \neq 0$ in $C_b$ with $\n{W_\triangle} \unibed C_b$ for $\triangle = +, 0, -$, one can think of the following trivial strategy to make $\n{P_+} \n{P_-} (Q+1)$ isometric isomorphism; Using Lemma \ref{induced Banach space}, set the domain of $\n{P_+} \n{P_-} (Q+1)$ to one of $\n{W}_\triangle$, say $\n{W}_0$, and set the codomain of $\n{P_+} \n{P_-} (Q+1)$ to $\n{P_+} \n{P_-} (Q+1)\n{W}_0$. However, this attempt might fail, because we do not have any guarantee that $(Q+1)\n{W}_0$ belongs to $C_b$. If we take differential operator on $C_b$, regularity of function is getting worse. Therefore, we have to set the domain of $\n{P_+} \n{P_-} (Q+1)$ to sufficiently regular space so that a function can maintain its boundedness after passing through differential operator. To achieve this, we set the domain to inverse of differential operator so that differential operator could be cancelled. With this motivation in mind, we investigate the properties of $\dnsi{n}{s}$.

\begin{proposition} \label{D^{-1} in C_b}
	For $s \neq 0$, the range of $\dns{n}{s}: C_b \rightarrow \dns{n}{s}C_b$ contains $C_b$ so that $\dnsi{n}{s}: C_b \rightarrow C_b$ is well-defined with
	\begin{equation} \label{D^{-1} C_b}
		\nrm{\dnsi{n}{s}}_{C_b} \leq \frac{1}{\abs{s}}.
	\end{equation}
\end{proposition}
\begin{proof}
	Solving the ODE
	\begin{equation*}
		\dns{n}{s}u=\beta \rd_\beta u -\beta i n u-su=f
	\end{equation*}
	for $f \in C_b$ gives
	\begin{align*}
		\dnsi{n}{s}f &= \beta^se^{in\beta}\int_C^\beta = \frac{f(x)e^{-inx}}{x^{s+1}} \, dx \\
		\underset{x=\beta y}&{=}\beta^se^{in\beta}\int_{C\beta}^1 = \frac{f(\beta y)e^{-in\beta y}}{\beta^{s+1}y^{s+1}}\beta \, dy \\
		&= e^{in\beta} \int_{\frac{C}{\beta}}^1 \frac{f(\beta y) e^{-in\beta y}}{y^{s+1}} \, dy,
	\end{align*}
	where $C \geq 0$ is an integral constant. Since the domain of $\dns{n}{s}$ is $C_b$, we are forced to choose
	\begin{equation*}
		C = \begin{cases}
			\infty & (s>0) \\
			0 & (s<0)
		\end{cases}
	\end{equation*}
	to make integral converge. Then, we have
	\begin{equation} \label{eq24}
		\dnsi{n}{s}f= \begin{cases}
			e^{in\beta}\int_\infty^1 \frac{f(\beta y)e^{-in \beta y}}{y^{s+1}} \, dy & (s>0) \\
			e^{in\beta}\int_0^1 \frac{f(\beta y)e^{-in \beta y}}{y^{s+1}} \, dy & (s>0)
		\end{cases}
	\end{equation}
	with
	\begin{align*}
		\nrm{\dnsi{n}{s}f}_{C_b} \leq \nrm{f}_{C_b} \int_1^\infty \frac{1}{y^{s+1}} \, dy= \frac{1}{s} \nrm{f}_{C_b} < \infty \qquad &(s>0), \\
		\nrm{\dnsi{n}{s}f}_{C_b} \leq \nrm{f}_{C_b} \int_0^\infty \frac{1}{y^{s+1}} \, dy= -\frac{1}{s} \nrm{f}_{C_b} < \infty \qquad &(s>0).
	\end{align*}
	Hence, $\dnsi{n}{s}f \in C_b$ for all $f \in C_b$, which implies $C_b \subset \dns{n}{s}C_b$ for $s \neq 0$.
\end{proof}

Using Proposition \ref{D^{-1} in C_b}, we define a following space which will be the space where $\br{\psi}$ and the codomain of $\br{L}$ lives.

\begin{definition} \label{X, Z space}
	\begin{align}
		\n{X_0} &\coloneqq \left(Q+1 \right)^{-1} \left( \n{P_-} \right)^{-1} \n{W_0} \label{Xnz}, \\
		\n{Z_0} &\coloneqq \n{P_+} \n{W}_0 \label{Znz}, \\
		\Box_{,N} &\coloneqq \begin{cases}
			\Box & n=N\bbZ  \\
			0 & n \neq N\bbZ
		\end{cases}\qquad \text{where} \qquad \Box= \n{X_0}, \n{Z_0} \nonumber
	\end{align}
\end{definition}

With this definition, the fact that $\n{P}_-(Q+1): \Xnz \rightarrow \Wnz$ is isometric isomorphism is a direct corollary of Lemma \ref{induced Banach space}. Before we check the continuity of the other differential operators using Proposition \ref{continuity between A^s}, we need to check
\begin{equation} \label{eq25}
	(Q+1)^{-1} \left( \n{P} \right)^{-1} \n{W_0} \unibed K \in \calS'(\br{\bbR}_+)
\end{equation}
for some Banach space $K$. Since
\begin{equation*}
	\Wnz = \cbdn \oplus \cxizn \oplus \delta_{0n} \bbC \hookrightarrow C_b,
\end{equation*}
it suffices to check
\begin{equation*}
	\nrm{\Qinv \Pnminv}_{[C_b, C_b]} \underset{\eqref{Q}, \eqref{Pnm}}{=} \nrm{\dnsi{0}{-1} \dnsi{n}{(2-n)\mu -1}}_{[C_b, C_b]} \underset{\eqref{D^{-1} C_b}}{=} \frac{1}{\abs{(2-n)\mu -1}}
\end{equation*}
is uniformly bounded for all $n \in N\bbZ$. It seems trivial, but what if $(2-n)\mu -1$ is zero for some $n \in N \bbZ$? To prevent this, we impose some condition on $\mu, N$; from now on, we assume
\begin{equation} \label{condition on N, mu}
	N \geq 2, \qquad \mu > \frac{2}{3}.
\end{equation}
With this assumption, $\abs{(2-n) \mu -1}$ never touches zero and $\frac{1}{\abs{(2-n) \mu -1}}$ is uniformly bounded for $n \in N \bbZ$. This shows \eqref{eq25} for $K = C_b$.

Now we treat boundedness property of operator $\dnsi{n}{s}$ on several spaces we have defined.

\begin{proposition} \label{D^{-1} in C_b^delta}	
	For $s \notin [-\delta, \delta]$, the range of $\dns{n}{s}: C_b^\delta \rightarrow \dns{n}{s}C_b^\delta$ contains $C_b^\delta$ so that $\dnsi{n}{s}: C_b^\delta \rightarrow C_b^\delta$ is well-defined with
	\begin{equation} \label{D^{-1} C_b^delta}
		\nrm{\dnsi{n}{s}}_{C_b^\delta} \leq \frac{1}{\dist \left([-\delta, \delta], s\right)}.
	\end{equation}
\end{proposition}

\begin{proof}
	This is also elementary like Proposition \ref{D^{-1} in C_b}. Since $C_b^\delta \hookrightarrow C_b$, $\dns{n}{s} \at{C_b^\delta}{}$ is injective and its inverse is well-defined by Proposition \ref{injective on C_b}. For $s>0$, 
	\begin{align*}
		\nrm{\dnsi{n}{s}f}_{C_b^\delta} \underset{\eqref{eq24}}&{=} \underset{+, -}{\max} \left\{\nrm{\beta^{\pm s} \int_\infty^1 \frac{f(\beta t) e^{-in \beta t}}{t^{s+1}} \, dt}_{C_b} \right\} \\
		&= \underset{+, -}{\max} \left\{ \nrm{\int_\infty^1 \frac{\left( \beta t \right)^{\pm \delta}f(\beta t) e^{in \beta t}}{t^{s+1 \pm \delta}}\, dt}_{C_b}  \right\} \\
		& \leq \nrm{f}_{C_b} \underset{+, -}{\max} \int_1^\infty \frac{1}{t^{s+1 \pm \delta}} \, dt \\
		&\leq \nrm{f}_{C_b^\delta} \int_1^\infty \frac{1}{t^{s+1-\delta}} \, dt \\
		\underset{\text{(converges only when }s>\delta\text{)}}&{\leq} \frac{1}{s-\delta} \nrm{f}_{C_b^\delta}.
	\end{align*}
	Similarly for $s<0$,
	\begin{align*}
		\nrm{\dnsi{n}{s}f}_{C_b^\delta} \underset{\eqref{eq24}}&{=} \underset{+, -}{\max} \left\{\nrm{\beta^{\pm s} \int_0^1 \frac{f(\beta t) e^{-in \beta t}}{t^{s+1}} \, dt}_{C_b} \right\} \\
		\underset{\text{(converges only when }s<-\delta \text{)}}&{\leq} \nrm{f}_{C_b^\delta} \frac{1}{s+\delta}.
	\end{align*}
	Combining these results, we have
	\begin{equation*}
		\nrm{\dnsi{n}{s}f}_{C_b^\delta} \leq \begin{cases}
			\frac{1}{s-\delta} \nrm{f}_{C_b^\delta} & (s>\delta) \\
			\frac{1}{s+\delta} \nrm{f}_{C_b^\delta} & (s<-\delta)
		\end{cases} = \frac{1}{\dist \left( [-\delta, \delta], s\right)} \nrm{f}_{C_b^\delta}
	\end{equation*}
\end{proof}

In this paper, \eqref{D^{-1} C_b^delta} will be the most frequently used inequality since $\n{P_+}, \n{P_-}, Q$ are the special case of $\dns{n}{s}$, where $s=(2+n)\mu -1, (2-n)\mu -1, -1$ respectively. Therefore, we need to calculate $\distPnpm$ and set $\delta$ more concretely.

\begin{proposition} \label{dist bound prop}
	Suppose $\mu > \frac{2}{3}$ and set
	\begin{equation} \label{def of delta}
		\delta=\frac{1}{2} \min \{ 2\mu -1, 1 \}.
	\end{equation} 
	
	Then, for $\abs{n} \geq N \geq 4$,
	\begin{equation} \label{dist bound}
		\frac{(N-2)\mu + \frac{5}{6}}{\brk{N}} \brk{n} \leq \dist \left( [-\delta, \delta], (2 \pm n)\mu -1 \right) \leq \frac{(N+2) \mu -\frac{3}{2}}{\brk{N}}\brk{n}.
	\end{equation}
\end{proposition} 
\begin{proof}
	By symmetry, we assume $n$ is positive. First, we estimate
	\begin{equation*}
		\frac{\dist \left( [-\delta, \delta], (2 + n)\mu -1 \right)}{\brk{n}} = \frac{(2+n) \mu -1-\delta}{\brk{n}}.
	\end{equation*}
	Consider $f: [4, \infty) \rightarrow \bbR$ defined by
	\begin{equation*}
		f(x)=\frac{(2+x)\mu -1-\delta}{\sqrt{x^2+1}}.
	\end{equation*}
	As
	\begin{equation*}
		f'(x)=\frac{\mu \sqrt{x^2+1}-\frac{\left( (2+x) \mu -1 -\delta \right)x}{\sqrt{x^2+1}}}{x^2+1} = \frac{\mu-(2\mu -1 -\delta)x}{\left( x^2+1 \right)^{\frac{3}{2}}},
	\end{equation*}
	we have
	\begin{equation*}
		f'(4)=\frac{4}{17\sqrt{17}}\left( \delta-\left( \frac{7}{4} \mu -1 \right) \right).
	\end{equation*}
	Using elementary inequality
	\begin{equation*}
		\delta=\frac{1}{2} \min \{ 2\mu -1, 1 \} \leq \frac{7}{4}\mu -1
	\end{equation*}
	on $\mu > \frac{2}{3}$, we have $f'(x) \leq 0$ for $x \geq 4$. Since
	\begin{equation*}
		f(N)=\frac{(2+N) \mu -1-\delta}{\brk{N}}\underset{\text{(}\delta \leq  \frac{1}{2} \text{)}}{\geq} \frac{(2+N)\mu -\frac{3}{2}}{\brk{N}} \qquad \text{and} \qquad \lim_{x \rightarrow \infty}f(x)=\mu,
	\end{equation*}
	we have
	\begin{equation*}
		\mu \leq \frac{\dist \left( [-\delta, \delta], (2 + n)\mu -1 \right)}{\brk{n}} \leq \frac{(2+N) \mu -\frac{3}{2}}{\brk{N}}
	\end{equation*}
	for $n \geq N \geq 4$.
	\bigskip

	Now, we estimate
	\begin{equation*}
		\frac{\dist \left( [-\delta, \delta], (2 - n)\mu -1 \right)}{\brk{n}} = \frac{(n-2)\mu +1-\delta}{\brk{n}}.
	\end{equation*}
	Consider $g:[4, \infty) \rightarrow \bbR$ defined by
	\begin{equation*}
		g(x)=\frac{(x-2)\mu +1-\delta}{\sqrt{x^2+1}}.
	\end{equation*}
	As
	\begin{equation*}
		g'(x)=\frac{\mu \sqrt{x^2+1}-\left( (x-2)\mu +1-\delta \right)\frac{x}{\sqrt{x^2+1}}}{x^2+1} = \frac{\mu+(2\mu -1 + \delta)x}{\left( x^2+1 \right) ^{\frac{3}{2}}},
	\end{equation*}
	we know $g'(x) \geq 0$ for $x \geq 4$. Since
	\begin{equation*}
		g(N)=\frac{(N-2)\mu +1 -\delta}{\brk{N}} \underset{\text{(} \delta>\frac{1}{6} \text{)}}{\leq} \frac{(N-2) \mu +\frac{5}{6}}{\brk{N}}\qquad \text{and} \qquad \lim_{x \rightarrow \infty} g(x) = \mu,
	\end{equation*}
	we have
	\begin{equation*}
		\frac{(N-2)\mu +\frac{5}{6}}{\brk{N}} \leq \frac{\dist \left( [-\delta, \delta], (2 - n)\mu -1 \right)}{\brk{n}} \leq \mu.
	\end{equation*}
	Combining estimates for $f$ and $g$, we get the result.
\end{proof}

From now on, we always assume $\mu, N, \delta$ satisfy assumption in Proposition \ref{dist bound prop}.

\begin{corollary} \label{Pninv in cbd}
	\begin{align}
		\nrm{\n{P_\pm}}_{[\cbdn, \cbdn]} & \leq \frac{\brk{N}}{(N-2)\mu +\frac{5}{6}} \cdot \frac{1}{\brk{n}}. \label{Pninv cbd} \\
		\nrm{(Q+1)}_{[\cbdn, \cbdn]} &\leq 2. \label{Qinv cbd}	
	\end{align}
\end{corollary}

\begin{proof}
	Proposition \ref{D^{-1} in C_b^delta} and \ref{dist bound prop}.
\end{proof}

Now, we calculate $\dnsi{n}{s}f$ when $f$ belongs to $\bbC$, a space of the constant function defined on $[0, \infty)$. Observe in \eqref{Wnz} that $\n{W_0}$ contains $\bbC$ only when $n=0$. The reason we included $\bbC$ on zero frequency is that we should include trivial solution \eqref{trivial solz bar} to the domain of $\br{L}$. Therefore, we show boundedness of $\dnsi{n}{s}$ on $\bbC$ only for $n=0$.

\begin{proposition} \label{D^{-1} in C}
	For $s \neq 0$, the range of $\dns{0}{s}: \bbC \rightarrow \dns{0}{s} \bbC$ contains $\bbC$ so that $\dnsi{0}{s}: \bbC \rightarrow \bbC$ is well-defined with
	\begin{equation} \label{D^{-1} C}
		\nrm{\dnsi{0}{s}}_{\bbC} \leq \frac{1}{\abs{s}}.
	\end{equation}
\end{proposition}
\begin{proof}
	Note that $\bbC \hookrightarrow C_b$ so that $\dnsi{0}{s}$ is well-defined by Proposition \ref{injective on C_b}. As
	\begin{equation*}
		\dns{0}{s}c=\left( \beta \rd_\beta -s \right)c=-sc
	\end{equation*}
	for $c \in \bbC$, we have
	\begin{equation*}
		\dnsi{0}{s}c=-\frac{1}{s}c.
	\end{equation*}
\end{proof}

Next, we calculate $\dnsi{n}{s}f$ when $f$ belongs to $\bbC \xi_0$.

\begin{proposition} \label{D^{-1} in xiz}
	For $s \notin [-\delta, \delta]$, $\dnsi{n}{s}: \bbC \xi_0 \rightarrow \cbd \oplus \cxiz$ is continuous with exact norm bound
	\begin{equation}\label{D^{-1} xiz}
		\nrm{\dnsi{n}{s}}_{[\cxiz, \cbd \oplus \cxiz]} \leq \frac{1}{\abs{s}} \left( \frac{\nrm{\beta \rd_\beta \xi_0}_{\cbd}+\abs{n} \ncbd{\beta \xi_0}}{\dist \left( [-\delta, \delta], s \right)} +1 \right).
	\end{equation}
\end{proposition}

\begin{proof}
	As
	\begin{equation*}
		\dns{n}{s}\xi_0 = \underbrace{\beta\rd_\beta \xi_0}_{\cbd} - \underbrace{in\beta \xi_0}_{\cbd} - \underbrace{s \xi_0}_{\cxiz}
	\end{equation*}
	and $\dns{n}{s}$ is injective on $C_b$, taking $\dnsi{n}{s}$ both side gives
	
	\begin{equation} \label{D^{-1} xiz formula}
		\dnsi{n}{s} \xi_0 = \frac{1}{s} \left( \dnsi{n}{s} (\beta \rd_\beta \xi_0) - in \dnsi{n}{s} (\beta \xi_0) - \xi_0  \right).
	\end{equation}
	From this inequality, we have \eqref{D^{-1} xiz}.
\end{proof}

Proposition \ref{D^{-1} in xiz} with Proposition \ref{dist bound prop} gives the following corollary.

\begin{corollary} \label{Pninv in xiz}
	\begin{align}
		\nrm{\left( \n{P_\pm} \right)^{-1}}_{[\cxizn, \cbdn \oplus \cxizn]} &\leq \frac{\brk{N}}{(N-2)\mu + \frac{5}{6}} \left( \frac{\nrm{\beta \rd_\beta \xi_0}_{\cbd}+\brk{N} \ncbd{\beta \xi_0}}{(N-2)\mu +\frac{5}{6}} +1 \right) \cdot \frac{1}{\brk{n}}. \label{Pninv xiz} \\
		\nrm{\left( Q+1 \right)^{-1}}_{[\cxizn, \cbdn \oplus \cxizn]} &\leq 2 \ncbd{\beta \rd_\beta \xi_0} +1. \label{Qinv xiz}
	\end{align}
\end{corollary}

\bigskip

\subsection{Continuity results for differential operator in function space}
\label{subsec: Continuity results for differential operator in function space}

In this subsection, we deal with $C^1$ continuity of $\br{L}: X \times Y \rightarrow Z$, where the spaces $X, Y, Z$ will be clarified. Before we do this, we show the following two lemmas related to commutativity of a differential operator with a multiplication operator. First, we treat the commutativity of $\dns{n}{s}$ and multiplication operator $\beta^\ell$.

\begin{lemma} \label{comm in dns}
	For all $f \in \calS'(\br{\bbR}_+)$ and $n, s, \ell \in \bbR$, 
	\begin{equation*}
		\dns{n}{s} \left( \beta^\ell  f\right) = \beta^\ell \dns{n}{s-\ell}f
	\end{equation*}
\end{lemma}

\begin{proof}
	Considering the definition of tempered distribution, it suffices to check that
	\begin{equation*}
		\brk{\dns{n}{s} \left( \beta^\ell  f\right), g}=\brk{\beta^\ell \dns{n}{s-\ell}f, g}
	\end{equation*}
	for all $g \in \calS(\br{\bbR}_+)$. This is just elementary calculus.
\end{proof}

Next, we deal with commutativity of $\dnsi{n}{s}$ and $\beta^\ell$. Unlike Lemma \ref{comm in dns}, there is some remainder term.

\begin{proposition} \label{comm in dnsi}
	Suppose $f \in \dns{n}{s-\ell}C_b$,\, $\beta^\ell f \in \dns{n}{s}C_b$ for some $n, s, \ell \in \bbR$, and $s, s- \ell \neq 0$. Then, there exists a constant $c \in \bbC$ such that
	\begin{equation} \label{comm dnsi}
		\beta^\ell \dnsi{n}{s-\ell} f = \dnsi{n}{s} \left( \beta^\ell f \right) + c \beta^s e^{in\beta}.
	\end{equation}
	In addition, if $s$ and $s-\ell$ have a same sign, the constant $c$ should be zero.
\end{proposition} 

\begin{proof}
	Since
	\begin{align*}
		& \dns{n}{s} \Bigg( \underbrace{\beta^\ell \underbrace{\dnsi{n}{s-\ell}f}_{\in C_b \text{ by assumption}}}_{\in \calS'(\br{\bbR}_+)} - \underbrace{\dnsi{n}{s} \left( \beta^\ell f \right)}_{\in C_b \hookrightarrow\calS'(\br{\bbR}_+) \text{ by assumption}} \Bigg) \\
		\underset{\text{Lemma } \ref{comm in dns}}&{=} \beta^\ell \dns{n}{s-\ell} \dnsi{n}{s-\ell} f - \dns{n}{s} \dnsi{n}{s} \beta^\ell f \\
		&= \beta^\ell f- \beta^\ell f \\
		&=0,
	\end{align*}
	
	$\beta^\ell \dnsi{n}{s-\ell} f - \dnsi{n}{s} \left(\beta^\ell f\right)$ lies in the kernel of the operator $\dns{n}{s}$. As the kernel of $\dns{n}{s}$ is $\bbC \beta^s e^{in\beta}$ as we saw in Proposition \ref{injective on C_b}, we get \eqref{comm dnsi}.
	
	Dividing both side of \eqref{comm dnsi} by $\beta^s e^{in \beta}$ gives \begin{equation*}
		e^{-in\beta} \beta^{\ell -s} \dnsi{n}{s- \ell}f = e^{-in\beta} \beta^{-s} \dnsi{n}{s} \left( \beta^\ell f \right) +c
	\end{equation*}
	Note here that $\dnsi{n}{s-\ell}f$ and $\dnsi{n}{s} \left( \beta^\ell f \right)$ are bounded by the assumption. If sgn($s$) and sgn($s-\ell$) are both positive, $\beta \rightarrow \infty$ gives $c=0$. On the contrary, if sgn($s$) and sgn($s-\ell$) are both negative, $\beta \rightarrow 0$ gives $c=0$.
\end{proof}

We are now ready to verify uniform boundedness of operator $id, Q, \n{P}, \n{P}Q$, which form an each Fourier mode of differential bar operators \eqref{eq17} $\sim$ \eqref{eq23} in $\br{L}$. These results will be combined later to establish $C^1$ continuity of $\br{L}$.

\begin{proposition} \label{continuity of id}
	$id: \Xnz \to \Wnz $ is uniformly bounded. Furthermore, if $n \neq 0$, we have an exact norm bound
	\begin{equation*}
		\nrm{id}_{\left[\Xnz, \Wnz\right]} \leq \frac{\brk{N}}{(N-2)\mu +\frac{5}{6}} \left( 2 \left( 1+ \nmt\right) \ncbd{\beta \rd_\beta \xi_0}+ 2\nmtn \ncbd{\beta\xi_0}+1 \right) \cdot \frac{1}{\brk{n}}.
	\end{equation*}
\end{proposition}

\begin{proof}
	For $n \notin N \bbZ$, nothing to be considered since norm is trivially zero, so we only treat the case $n \in N\bbZ$. As
	\begin{align*}
		& \nrm{id}_{[\Xnz, \Wnz]}\\
		&= \nrm{id}_{\left[\Qinv \Pnminv \Wnz, \Wnz\right]} \\
		&= \nrm{\Qinv \Pnminv}_{[\Wnz, \Wnz]} \\
		&\leq \nrm{\Qinv \Pnminv}_{[\cbdn, \Wnz]} + \nrm{\Qinv \Pnminv}_{[\cxizn, \Wnz]} + \delta_{0n} \nrm{\Qinv \Pnminv}_{[\bbC, \Wnz]},
	\end{align*}
	norm bounds \eqref{Pninv cbd}, \eqref{Qinv cbd},\eqref{D^{-1} C}, \eqref{Pninv xiz}, \eqref{Qinv xiz} give uniform boundedness.
	
	For $n \neq 0$, we calculate norm bound more explicitly. As $\Pnminv \cbd \subset \cbd, \,\, \Qinv \cbd \subset \cbd \subset \Wnz$, we have
	\begin{equation*}
		\nrm{\Qinv \Pnminv}_{[\cbdn, \Wnz]} \leq \nrm{\Qinv}_{[\cbdn, \cbdn]} \nrm{\Pnminv}_{[\cbdn, \cbdn]} \underset{\eqref{Pninv cbd}, \eqref{Qinv cbd}}{=} 2 \nmtn \cdot \frac{1}{\brk{n}}
	\end{equation*}
	
	Also, since $\Qinv \Pnminv \xi_0$ could be calculated as
	\begin{align*}
		&\Qinv \Pnminv \xi_0 \\
		\underset{\eqref{D^{-1} xiz formula}}&{=} \frac{1}{(2-n)\mu -1} \Qinv \left( \Pnminv (\beta \rd_\beta \xi_0) - in \Pnminv (\beta \xi_0) - \xi_0  \right) \\
		\underset{\eqref{D^{-1} xiz formula}}&{=} \frac{1}{(2-n)\mu -1} \left( \Qinv \Pnminv (\beta \rd_\beta \xi_0) - in \Qinv \Pnminv (\beta \xi_0) + \Qinv (\beta \rd_\beta \xi_0 ) -\xi_0  \right),
	\end{align*}
	its norm can be estimated by
	\begin{align*}
		& \nrm{\Qinv \Pnminv}_{[\cxizn, \cbdn \oplus \cxizn]} \\ 
		&\leq \frac{1}{\abs{(2-n)\mu -1}} \Bigg( \nrm{\Qinv}_{[\cbdn, \cbdn]} \left( \nrm{\Pnminv}_{[\cbdn, \cbdn]}+1 \right) \nrm{\beta \rd_\beta \xi_0}_{\cbdn}  \\
		& \qquad \qquad \qquad \qquad \quad \abs{n} \nrm{\Qinv}_{[\cbdn, \cbdn]} \nrm{\Pnminv}_{[\cbdn, \cbdn]} \nrm{\beta \xi_0}_{\cbd} +1 \Bigg) \\
		\underset{\eqref{Pninv cbd},\eqref{Qinv cbd}}&{\leq} \nmtn \cdot \frac{1}{\brk{n}} \Bigg( 2 \left( \nmtn \cdot \frac{1}{\brk{n}} +1 \right) \nrm{\beta \rd_\beta \xi_0}_{\cbd} \\
		& \qquad \qquad \qquad \qquad \qquad \qquad  + 2\abs{n} \nmtn \cdot \frac{1}{\brk{n}} \nrm{\beta \xi_0}_{\cbd} +1  \Bigg) \\
		\underset{\substack{\abs{n} \geq \abs{N} \\ \brk{n} \geq \abs{n}}}&{\leq} \nmtn \left( 2 \left( 1+ \nmt \right) \nrm{\beta \rd_\beta \xi_0}_{\cbd} +2 \nmtn \nrm{\beta \xi_0}_{\cbd} +1 \right) \cdot \frac{1}{\brk{n}}.
	\end{align*}

\end{proof}

\begin{corollary} \label{continuity of inid}
	$in \cdot id: \Xnz \rightarrow \Wnm$ is uniformly bounded with an exact norm bound
	\begin{equation*}
		\nrm{in \cdot id}_{\left[\Xnz, \Wnz\right]} \leq \frac{\brk{N}}{(N-2)\mu +\frac{5}{6}} \left( 2 \left( 1+ \nmt\right) \ncbd{\beta \rd_\beta \xi_0}+ 2\nmtn \ncbd{\beta\xi_0}+1 \right).
	\end{equation*}
\end{corollary}

\begin{remark}
	Since the multiplier $in$ kills the zero Fourier mode, the range of $in \cdot id$ could be $\Wnm$.
\end{remark}

\begin{proposition} \label{continuity of Q}
	$Q: \Xnz \rightarrow \Wnm$ is uniformly bounded with an exact norm bound
	\begin{align*}
		\nrm{Q}_{[\Xnz, \Wnm]} &\leq \Bigg( 3\nmtn+ \nmtn \left( \frac{3}{(N-2)\mu +\frac{5}{6}} +2 \right) \ncbd{\beta \rd_\beta \xi_0} \\
		& \qquad +3 \left( \nmtn \right)^2 \ncbd{\beta \xi_0} \Bigg) \cdot \frac{1}{\brk{n}}.
	\end{align*}
\end{proposition}

\begin{proof}
	It suffices to focus on $n \in N\bbZ$, since the norm is 0 for $n \notin N\bbZ$. Note here that the range of $Q$ is $\Wnm$, which is smaller than that of $id$. This is because, for a constant $c \in \bbC$, $\Qinv \left( P^{(0)}_- \right)^{-1}c$ is a constant by Proposition \ref{D^{-1} C} and $Q$ is equal to $\beta \rd_\beta$, which kills a constant function. As before, we can estimate the norm of $Q: \Xnz \rightarrow \Wnm$ when $n \in N\bbZ \backslash \{0\}$ as follows;
	\begin{align*}
		& \nrm{Q}_{[\Xnz, \Wnm]} \\
		&= \nrm{Q \Qinv \Pnminv}_{[\Wnz, \Wnm]} \\
		&= \nrm{\left( Q+1-1 \right) \Qinv \Pnminv}_{[\Wnz, \Wnm]} \\
		&= \nrm{\Pnminv-\Qinv \Pnminv}_{[\Wnz, \Wnm]} \\
		&\leq \nrm{\Pnminv-\Qinv \Pnminv}_{[\cbdn, \cbdn]} +\nrm{\Pnminv-\Qinv \Pnminv}_{[\cxizn, \cbdn \oplus \cxizn]}.
	\end{align*}
	
	By \eqref{Pninv cbd} and \eqref{Qinv cbd}, we have
	\begin{align*}
		& \nrm{\Pnminv-\Qinv \Pnminv}_{[\cbdn, \cbdn]} \\
		&\leq \nmtn \cdot \frac{1}{\brk{n}}+2\nmtn \cdot \frac{1}{\brk{n}} \\
		&= 3\nmtn \cdot \frac{1}{\brk{n}}.
	\end{align*}
	
	Also, from the formula
	\begin{align*}
		&\left( \Pnminv - \Qinv \Pnminv \right) \xi_0 \\
		\underset{\eqref{D^{-1} xiz formula}}&{=}\frac{1}{(2-n)\mu -1} \Bigg( \underbrace{\left( \Pnminv - \Qinv - \Qinv \Pnminv \right) (\beta \rd_\beta \xi_0)}_{\underset{\text{Prop }\ref{D^{-1} in C_b^delta}}{\in} \cbd} \\
		&\qquad \qquad \qquad \qquad + \underbrace{in \left( \Qinv \Pnminv - \Pnminv \right) (\beta \xi_0)}_{\underset{\text{Prop }\ref{D^{-1} in C_b^delta}}{\in} \cbd} \Bigg),
	\end{align*}
	we can estimate $\nrm{\Pnminv - \Qinv \Pnminv}_{[\cxizn, \cbdn \oplus \cxizn]}$ by
	\begin{align*}
		& \nrm{\Pnminv - \Qinv \Pnminv}_{[\cxizn, \cbdn \oplus \cxizn]} \\
		&\leq \frac{1}{\abs{(2-n)\mu -1}} \Bigg( \nrm{\Pnminv}_{[\cbdn, \cbdn]}+\nrm{\Qinv}_{[\cbdn, \cbdn]} \\
		&\qquad \qquad \qquad \qquad \quad  +\nrm{\Qinv}_{[\cbdn, \cbdn]} \nrm{\Pnminv}_{[\cbdn, \cbdn]} \Bigg) \ncbd{\beta \rd_\beta \xi_0} \\
		& \qquad +\frac{\abs{n}}{\abs{(2-n)\mu -1}} \left( \nrm{\Qinv}_{[\cbdn,\cbdn]} \nrm{\Pnminv}_{[\cbdn, \cbdn]} + \nrm{\Pnminv}_{[\cbdn, \cbdn]} \right) \ncbd{\beta \xi_0} \\
		\underset{\eqref{Pninv cbd},\eqref{Qinv cbd}}&{\leq} \nmtn \cdot \frac{1}{\brk{n}} \left( 3 \nmtn \cdot \frac{1}{\brk{n}}+2 \right) \ncbd{\beta \rd_\beta \xi_0} \\
		& \qquad \quad + \nmtn \cdot \frac{\abs{n}}{\brk{n}} \left( 3 \nmtn \cdot \frac{1}{\brk{n}} \right) \ncbd{\beta \xi_0} \\
		&\leq \left( \nmtn \left( \frac{3}{(N-2)\mu +\frac{5}{6}} +2 \right) \ncbd{\beta \rd_\beta \xi_0}  +3 \left( \nmtn \right)^2 \ncbd{\beta \xi_0} \right) \cdot \frac{1}{\brk{n}}.
	\end{align*}
	Combining these two estimates yields the result.
\end{proof}

\begin{proposition} \label{continuity of Pn(Q+1)}
	$\n{P}(Q+1): \Xnz \rightarrow \Wnm$ is uniformly bounded with an exact norm bound 
	\begin{equation*}
		\nrm{\n{P}(Q+1)}_{[\Xnz, \Wnm]} \leq 4+ \nmtn \ncbd{\beta \xi_0} + \nmt \ncbd{\beta \rd_\beta \xi_0}.
	\end{equation*}
\end{proposition}

\begin{proof}
	As always, it suffices to focus on $n \in N \bbZ$. As in the case of $Q$, the range of $\n{P}(Q+1)$ is contained in $\Wnm$. This is because, for a constant $c \in \bbC$, $\left( P^{(0)}_- \right)^{-1}c$ is a constant by Proposition \ref{D^{-1} in C} and
	\begin{equation*}
		P^{(0)} (Q+1) \Qinv \left( P^{(0)}_- \right)^{-1} c= \underbrace{P^{(0)}}_{=\beta \rd_\beta} \underbrace{\left( P^{(0)}_- \right)^{-1} c}_{\in \bbC}=0.
	\end{equation*}
	
	Now, we estimate the norm of $\n{P}(Q+1)$ for $n \in N \bbZ \backslash \{0\}$.
	
	\begin{align*}
		&\nrm{\n{P} (Q+1)}_{[\Xnz, \Wnm]} \\
		&= \nrm{\n{P} (Q+1) \Qinv \Pnminv}_{[\Wnz, \Wnm]} \\
		&= \nrm{\n{P} \Pnminv}_{[\Wnz, \Wnm]} \\
		&= \nrm{\left( \n{P}_- + \left( (2-n) \mu -1 \right) \right) \Pnminv}_{[\Wnz, \Wnm]} \\
		\underset{\text{(} n \neq 0 \text{)}}&{=} \nrm{\left( \n{P}_- + \left( (2-n) \mu -1 \right) \right) \Pnminv}_{[\Wnz, \Wnz]} \\
		&\leq \nrm{id}_{[\Wnz, \Wnz]} + \abs{(2-n)\mu -1} \nrm{\Pnminv}_{[\Wnz, \Wnz]} \\
		&\leq  \nrm{id}_{[\Wnz, \Wnz]}+\abs{(2-n)\mu -1} \nrm{\Pnminv}_{[\cbdn, \cbdn]}+\abs{(2-n)\mu -1} \nrm{\Pnminv}_{[\cxizn, \cbdn \oplus \cxizn]}  \\
		\underset{\eqref{D^{-1} C_b^delta},\eqref{D^{-1} xiz}}&{\leq} 1+ \underbrace{\frac{ \abs{(2-n)\mu -1}}{\distPnm}}_{\leq 2} +\nmtn \cdot \frac{1}{\brk{n}} \left( \ncbd{\beta \rd_\beta \xi_0} + \abs{n} \ncbd{\beta \xi_0} \right)+1  \\
		&\leq 4+ \nmtn \ncbd{\beta \xi_0} + \nmt \ncbd{\beta \rd_\beta \xi_0}.
	\end{align*}
\end{proof}

Even though a differential operator might be expected to damage the regularity, we are safe so far. This is because $\Xnz$ was made by taking inverse of differential operator, so even if we apply differential operator, it was cancelled out. However, when we deal with operator $P^{(n)}$, this cancelling does not occur directly which makes the proof harder. We first deal with how $\n{P}$ acts on $\cbd$.

\begin{proposition} \label{continuity of Pn in cbd}
	$\n{P}\Qinv \Pnminv: \cbdn \rightarrow \cbdn \oplus \cxiin$ is uniformly bounded. Furthermore, if $n \neq 0$, we have an exact norm bound
	\begin{align*}
		&\nrm{\n{P} \Qinv \Pnminv}_{[\cbdn, \cbdn \oplus \cxiin]} \\
		&\leq \frac{4}{(N-2)\mu -1} + \frac{66}{5} + \left( \frac{\brk{N}+1}{(N-2)\mu -1}+\frac{66}{5} \right) \maxi{\nrm{\beta^\delta \xi_0}_{C_b}, \nrm{\beta^{-\delta}\xi_\infty}_{C_b}}
	\end{align*}
\end{proposition}

\begin{proof}
	Since $P^{(0)}=Q$ and boundedness of $Q \Qinv \left( P^{(0)}_- \right)^{-1}$ was shown in Proposition \ref{continuity of Q}, we focus on $n \in N\bbZ \backslash \{0\}$. As $\n{P}=Q-in\beta$ and we already dealt with uniform boundedness of
	\begin{equation*}
		Q \Qinv \Pnminv: \cbdn \oplus \cxizn \rightarrow \cbdn
	\end{equation*}
	in Proposition \ref{continuity of Q}, we analyze
	\begin{equation*}
		in \beta \Qinv \Pnminv: \cbdn \oplus \cxizn \rightarrow \cbdn \oplus \cxiin.
	\end{equation*}
	For $f \in \cbd$,
	\begin{equation} \label{eq26}
		\begin{aligned}
			& in \beta \Qinv \Pnminv f \\
			&= in \beta \Qinv \underbrace{\Pnminv \beta^\delta \beta^{-\delta}f}_{(*)} \\
			\underset{\eqref{comm dnsi}}&{=} in \beta \Qinv \underbrace{\beta^\delta \left( \Pnm + \delta \right)^{-1} \beta^{-\delta}f}_{(*)} \\
			&= \beta^\delta \underbrace{in \beta^{1-\delta} \Qinv \beta^\delta \left( \Pnm - \delta \right)^{-1} \beta^{-\delta} f}_{(**)} \\
			\underset{\eqref{comm dnsi}}&{=}\beta^\delta \underbrace{\left(Q+ \delta\right)^{-1} in \beta \left( \Pnm + \delta \right)^{-1} \beta^{-\delta} f}_{(**)}.
		\end{aligned}
	\end{equation}
	
	We explain how the second and the last equality are justified by checking assumptions in Proposition \ref{comm in dnsi}. For second equality, note that $\beta^{-\delta}f \in C_b$ and $\left( \Pnm + \delta \right)^{-1} \beta^{\delta} f \in C_b$ as $f \in \cbd$. Also, by \eqref{def of delta}, $\left((2-n)\mu -1\right)$ and $\left((2-n)\mu -1-\delta\right)$ are not equal to zero. Therefore, the assumption of Proposition \ref{comm in dnsi} are satisfied. Also, the constant term $c$ does not appear as $\sgn((2-n)\mu -1) = \sgn ((2-n)\mu -1-\delta)$.
	
	For the last equality, we check whether $\Qinv \beta^\delta \left( \Pnm + \delta \right)^{-1} \beta^{-\delta} f$ lies in $C_b$. This is clear as
	\begin{equation*}
		\Qinv \beta^\delta \left( \Pnm + \delta \right)^{-1} \beta^{-\delta} f = \Qinv \Pnminv f \underset{\text{Prop }\ref{D^{-1} in C_b^delta}}{\in} \cbd \hookrightarrow C_b
	\end{equation*}
	Second, we check whether $\left( Q+ \delta \right)^{-1} in \beta \left( \Pnm + \delta \right)^{-1} \beta^{-\delta} f$ is contained in $C_b$. The key is to write $in\beta$ as
	\begin{equation} \label{eq27}
		in \beta = (Q+\delta)- \left( \Pnm + \delta \right) - \left( (2-n)\mu -1 \right).
	\end{equation}
	Then, we can enjoy the cancellation so that
	\begin{align*}
		& \left( Q+ \delta \right)^{-1} in \beta \left( \Pnm + \delta \right)^{-1} \beta^{-\delta} f \\
		&= \left( Q+ \delta \right)^{-1} \left( (Q+\delta) - \left(\Pnm + \delta \right) - ((2-n)\mu -1) \right)  \left( \Pnm + \delta \right)^{-1} \beta^{-\delta} f \\
		&= \underbrace{\left( \Pnm + \delta \right)^{-1}  \underbrace{\beta^{-\delta} f}_{\in C_b (\because f \in \cbd)}}_{\underset{\text{Prop }\ref{D^{-1} in C_b}}{\in}C_b} - \underbrace{(Q+\delta)^{-1} \underbrace{\beta^{-\delta}f}_{\in C_b}}_{\underset{\text{Prop }\ref{D^{-1} in C_b}}{\in}C_b}-\underbrace{((2-n)\mu -1) (Q+\delta)^{-1} \left( \Pnm + \delta \right)^{-1} \underbrace{\beta^{-\delta} f}_{\in C_b}}_{\underset{\text{Prop }\ref{D^{-1} in C_b}}{\in}C_b}.
	\end{align*}
	Third, $-1$ and $-\delta$ have same sign, so constant term $c$ does not appear. This justified the last equality.
	
	Since $(**) \in C_b$, $\beta^\delta$ decay near $\beta = 0$ is now guaranteed in
	\begin{equation*}
		in \beta \Qinv \Pnminv f = \beta (**)
	\end{equation*}
	But what happens near $\beta=\infty$? Repeating similar argument with little change, we have
	
	\begin{equation} \label{eq28}
		\begin{aligned}
			& in \beta \Qinv \Pnminv f \\
			&= in \beta \Qinv \underbrace{\Pnminv \beta^{-\delta} \beta^{\delta} f}_{(*)} \\
			\underset{\eqref{comm dnsi}}&{=} in \beta \Qinv \underbrace{\beta^{-\delta} \left( \Pnm - \delta \right)^{-1} \beta^\delta f}_{(*)} \\
			&= \beta^{-\delta} \Big( \underbrace{in \beta^{1+\delta} \Qinv \beta^{-\delta} \left( \Pnm - \delta \right)^{-1} \beta^\delta f}_{(***)} \Big) \\
			\underset{\eqref{comm dnsi}}&{=} \beta^{-\delta} \Big( \underbrace{(Q-\delta)^{-1} in\beta \left( \Pnm - \delta \right)^{-1} \beta^\delta f + c\beta^\delta}_{(***)} \Big) \\
			&= \beta^{-\delta} \underbrace{(Q-\delta)^{-1} in\beta \left( \Pnm - \delta \right)^{-1} \beta^\delta f}_{(****)} +c.
		\end{aligned}
	\end{equation}
	Checking the assumption of Proposition \ref{comm in dnsi} is similar to above, so we omit it. However, a big difference is that in the second application of Proposition \ref{comm in dnsi}, the sign of $-1$ and $\delta$ is different, so constant term $c$ appears. Since the term $(****)$ belongs to $C_b$ for the same reason as $(**)$ belongs to $C_b$,
	\begin{equation*}
		in \beta \Qinv \Pnminv f -c
	\end{equation*}
	shows $\beta^{-\delta}$ decay at $\beta=\infty$. Hence, to estimate $\cbd \oplus \cxii$ norm of $in \beta \Qinv \Pnm f$, we need to estimate a constant $c$ first. From \eqref{eq28}, we have
	\begin{equation*}
		c=in \beta \Qinv \Pnminv f -\beta^{-\delta} (Q-\delta)^{-1} in \beta \left( \Pnm - \delta \right)^{-1} \beta^\delta f.
	\end{equation*}
	Take $\beta =1$ and decomposing $in \beta$ to
	\begin{equation} \label{eq29}
		in \beta = (Q-\delta) - \left( \Pnm - \delta \right) -((2-n)\mu -1 )
	\end{equation}
	gives
	\begin{align*}
		c= & in \Qinv \Pnminv f \\
		&- \left( \Pnm - \delta \right)^{-1} \beta^\delta f \\
		&+ (Q-\delta)^{-1} \beta^\delta f \\
		&+ ((2-n)\mu -1) \left( Q-\delta \right) ^{-1} \left( \Pnm - \delta \right)^{-1} \beta^\delta f \\
		& \eqqcolon I_1+I_2+I_3+I_4.
	\end{align*}
	Let us estimate upper bound of $I_1 \sim I_4$ one by one.
	
	\begin{align*}
		\nrm{I_1}_{C_b} &= \abs{n} \nrm{\Qinv \Pnminv f}_{C_b} \\
		\underset{\eqref{D^{-1} C_b}}&{\leq}\frac{\abs{n}}{\dist ( \{0\}, (2-n)\mu -1)} \nrm{f}_{C_b} \\
		&\leq \frac{\abs{n}}{\dist ( [-\delta, \delta], (2-n)\mu -1)} \nrm{f}_{C_b} \\
		\underset{\eqref{dist bound}}&{\leq} \abs{n} \nmtn \cdot \frac{1}{\brk{n}} \nrm{f}_{C_b} \\
		&\leq \nmtn \ncbd{f},
	\end{align*}
	
	\bigskip
	
	\begin{align*}
		\nrm{I_2}_{C_b} &= \nrm{\left(\Pnm - \delta \right)^{-1} \beta^\delta f}_{C_b} \\
		\underset{\eqref{D^{-1} C_b}}&{\leq} \frac{1}{\abs{(2-n)\mu -1+\delta}} \nrm{\beta^\delta f}_{C_b} \\
		&\leq \frac{1}{\distPnm} \ncbd{f} \\
		\underset{\eqref{dist bound}}&{\leq} \nmtn \cdot \frac{1}{\brk{n}} \ncbd{f} \\ 
		\underset{(n \neq 0)}&{\leq} \nmt \ncbd{f},
	\end{align*}
	
	\bigskip
	
	\begin{align*}
		\nrm{I_3}_{C_b} &\leq \nrm{(Q-\delta)^{-1} \beta^\delta f}_{C_b} \\
		\underset{\eqref{D^{-1} C_b}}&{\leq} \frac{1}{\delta} \nrm{\beta^\delta f}_{C_b} \\
		&\leq \frac{1}{\delta} \ncbd{f} \\
		\underset{\eqref{def of delta}}&{\leq} 6 \ncbd{f},
	\end{align*}
	
	\bigskip
	
	\begin{align*}
		\nrm{I_4}_{C_b} &= \nrm{((2-n)\mu -1) (Q-\delta)^{-1} \left( \Pnm - \delta\right)^{-1} \beta^\delta f}_{C_b} \\
		\underset{\eqref{D^{-1} C_b}}&{\leq} \abs{(2-n)\mu -1} \cdot \frac{1}{\delta} \cdot \frac{1}{\distPnm} \ncbd{f} \\
		&\leq \frac{1}{\delta} \cdot \frac{1}{\dist ([-\delta,\delta], -1)} \ncbd{f} \\
		\underset{\eqref{def of delta}}&{\leq} \frac{36}{5} \ncbd{f}.
	\end{align*}
	
	Combining these result, we have
	\begin{equation} \label{eq30}
		\begin{aligned}
			\abs{c} &\leq \left( \frac{\brk{N}+1}{(N-2)\mu +\frac{5}{6}}+ \frac{66}{5} \right) \ncbd{f}.
		\end{aligned}
	\end{equation}
	
	As we have estimated $\bbC \xi_\infty$ part of $in \beta \Qinv \Pnminv f$, we now estimate $\cbd$ norm of $in \beta \Qinv \Pnminv f - c\xi_\infty$. First,
	\begin{align*}
		& \nrm{\beta^\delta \left( in \beta \Qinv \Pnminv f - c \xi_\infty \right)}_{C_b} \\
		\underset{\eqref{eq28}}&{=} \nrm{\beta^\delta \left( \beta^{-\delta} (Q-\delta)^{-1} in \beta \left( \Pnm - \delta \right)^{-1} \beta^\delta f + c(1-\xi_\infty) \right)}_{C_b} \\
		&\leq \nrm{(Q-\delta)^{-1} in \beta \left( \Pnm - \delta \right)^{-1} \beta^\delta f}_{C_b} + \abs{c} \nrm{\beta^\delta \xi_0}_{C_b} \\
		\underset{\eqref{eq29}}&{\leq} \nrm{I_2}_{C_b}+\nrm{I_3}_{C_b}+\nrm{I_4}_{C_b}+ \abs{c} \nrm{\beta^\delta \xi_0}_{C_b} \\
		\underset{\eqref{eq30}}&{\leq} \left( \nmt + \frac{66}{5} + \left( \frac{\brk{N}+1}{(N-2)\mu +\frac{5}{6}} + \frac{66}{5} \right) \nrm{\beta^\delta \xi_0}_{C_b}  \right) \ncbd{f}. \\
	\end{align*}
	
	Second,
	\begin{align*}
		& \nrm{\beta^{-\delta} \left( in \beta \Qinv \Pnminv f - c\xi_\infty \right) }_{C_b} \\
		\underset{\eqref{eq26}}&{=} \nrm{\beta^{-\delta} \left( \beta^\delta (Q+\delta)^{-1} in \beta \left( \Pnm + \delta \right)^{-1} \beta^{-\delta} f -c \xi_\infty \right)}_{C_b} \\
		&\leq \nrm{(Q+\delta)^{-1} in \beta \left( \Pnm + \delta \right)^{-1} \beta^{-\delta} f}_{C_b} + \abs{c} \nrm{\beta^{-\delta}\xi_\infty}_{C_b} \\
		\underset{\eqref{eq27}}&{\leq} \ncb{\left(\Pnm + \delta\right)^{-1} \beta^{-\delta} f}+ \ncb{(Q+\delta)^{-1} \beta^{-\delta} f} \\
		& \qquad + \abs{(2-n)\mu -1}\ncb{(Q+\delta)^{-1}\left( \Pnm + \delta \right)^{-1} \beta^{-\delta} f} + \abs{c} \ncb{\beta^{-\delta} \xi_\infty}\\
		&\leq \left( \nmt + \frac{66}{5} + \left( \frac{\brk{N}+1}{(N-2)\mu +\frac{5}{6}}+\frac{66}{5} \right) \nrm{\beta^{-\delta} \xi_\infty}_{C_b}  \right) \ncbd{f}.
	\end{align*}
	
	Combining these result, we have
	\begin{equation} \label{continuity of inbeta}
		\begin{aligned}
			& \nrm{in \beta \Qinv \Pnminv}_{[\cbdn, \cbdn \oplus \cxiin]} \\
			& \leq \nmt + \frac{66}{5} + \left( \frac{\brk{N}+1}{(N-2)\mu +\frac{5}{6}}+\frac{66}{5} \right) \maxi{\nrm{\beta^\delta \xi_0}_{C_b}, \nrm{\beta^{-\delta}\xi_\infty}_{C_b}}.
		\end{aligned}
	\end{equation}
	
	Then,
	\begin{align*}
		& \nrm{\n{P} \Qinv \Pnminv}_{[\cbdn, \cbdn \oplus \cxiin]} \\
		\underset{\eqref{Pn}, \eqref{Q}}&{\leq} \nrm{(Q-in\beta) \Qinv \Pnminv}_{[\cbdn, \cbdn \oplus \cxiin]} \\
		&\leq \nrm{Q \Qinv \Pnminv}_{[\cbdn, \cbdn]} + \nrm{in \beta \Qinv \Pnminv}_{[\cbdn, \cbdn \oplus \cxiin]} \\
		\underset{\eqref{D^{-1} C_b^delta},\eqref{continuity of inbeta}}&{\leq} 3\nmtn \cdot \frac{1}{\brk{n}}+\nmt + \frac{66}{5} + \left( \frac{\brk{N}+1}{(N-2)\mu +\frac{5}{6}}+\frac{66}{5} \right) \maxi{\nrm{\beta^\delta \xi_0}_{C_b}, \nrm{\beta^{-\delta}\xi_\infty}_{C_b}} \\
		&\leq \frac{4}{(N-2)\mu +\frac{5}{6}} + \frac{66}{5} + \left( \frac{\brk{N}+1}{(N-2)\mu +\frac{5}{6}}+\frac{66}{5} \right) \maxi{\nrm{\beta^\delta \xi_0}_{C_b}, \nrm{\beta^{-\delta}\xi_\infty}_{C_b}}.
	\end{align*} 
\end{proof}

\begin{proposition} \label{continuity of Pn in xiz}
	$\n{P} \Qinv \Pnminv: \cxizn \rightarrow \cbdn \oplus \cxiin$ is uniformly bounded. Furthermore, if $n \neq 0$, we have an exact norm bound
	\begin{align*}
		&\nrm{\n{P} \Qinv \Pnminv}_{[\cxizn, \cbdn \oplus \cxiin]}  \\
		&\leq \left( \frac{4}{(N-2)\mu -1} + \frac{66}{5} + \left( \frac{\brk{N}+1}{(N-2)\mu -1}+\frac{66}{5} \right) \maxi{\nrm{\beta^\delta \xi_0}_{C_b}, \nrm{\beta^{-\delta}\xi_\infty}_{C_b}} \right) \\
		& \qquad \cdot \Bigg(\left(1+\frac{2}{(N-2)\mu +\frac{5}{6}} \right) \ncbd{\beta \rd_\beta \xi_0}+ \nmtn \ncbd{\beta \xi_0}+\nmtn\ncbd{\beta^2 \rd_\beta \xi_0} \\
		&\qquad \qquad \quad + \nmt \ncbd{\beta^2 \rd^2_\beta \xi_0}  \Bigg)+ \nmt \ncbd{\beta \rd_\beta \xi_0}+ \nmtn \ncbd{\beta \xi_0}.
	\end{align*}
\end{proposition}

\begin{proof}
	Since $P^{(0)}=Q$ for $n=0$ and this case is was already dealt with, we only prove for $n \in N \bbZ \backslash \{0\}$. Direct calculation using \eqref{D^{-1} xiz formula} gives
	\begin{align*}
		& \n{P} \Qinv \Pnminv \xi_0 \\
		&= \frac{1}{(2-n)\mu -1} \underbrace{\n{P} \Qinv \Pnminv \bigg( \underbrace{\Big( 2-((2-n)\mu -1) \Big  ) \beta \rd_\beta \xi_0-in \beta \xi_0 - in \beta^2 \rd_\beta \xi_0 + \beta^2 \rd^2_\beta \xi_0}_{\in \cbd} \bigg)}_{\underset{\text{Prop }\ref{continuity of Pn in cbd}}{\in} \cbd \oplus \cxii} \\
		& \qquad+\frac{1}{(2-n)\mu -1} (\underbrace{-\beta \rd_\beta \xi_0 + in \beta_0}_{\in \cbd}).
	\end{align*}
	Then, $\cbdn \oplus \cxiin$ norm of it could be bounded by
	\begin{align*}
		& \nrm{\n{P} \Qinv \Pnminv \xi_0}_{\cbdn \oplus \cxiin} \\
		&\leq \frac{1}{\abs{(2-n)\mu -1}} \nrm{\n{P} \Qinv \Pnminv}_{[\cbdn, \cbdn \oplus \cxiin]} \\
		& \qquad \cdot \nrm{\Big( 2-((2-n)\mu -1) \Big  ) \beta \rd_\beta \xi_0-in \beta \xi_0 - in \beta^2 \rd_\beta \xi_0 + \beta^2 \rd^2_\beta \xi_0}_{\cbd} \\
		& \qquad+\frac{1}{\abs{(2-n)\mu -1}} \ncbd{-\beta \rd_\beta \xi_0 + in\beta \xi_0} \\
		& \leq \left( \frac{4}{(N-2)\mu +\frac{5}{6}} + \frac{66}{5} + \left( \frac{\brk{N}+1}{(N-2)\mu -1}+\frac{66}{5} \right) \maxi{\nrm{\beta^\delta \xi_0}_{C_b}, \nrm{\beta^{-\delta}\xi_\infty}_{C_b}} \right) \\
		& \qquad \cdot \Bigg(\left(1+\frac{2}{(N-2)\mu +\frac{5}{6}} \right) \ncbd{\beta \rd_\beta \xi_0}+ \nmtn \ncbd{\beta \xi_0}+\nmtn\ncbd{\beta^2 \rd_\beta \xi_0} \\
		&\qquad \qquad \quad + \nmt \ncbd{\beta^2 \rd^2_\beta \xi_0}  \Bigg)+ \nmt \ncbd{\beta \rd_\beta \xi_0}+ \nmtn \ncbd{\beta \xi_0}
	\end{align*}
\end{proof}

Combining Proposition \ref{continuity of Pn in cbd} and Proposition \ref{continuity of Pn in xiz} yields the continuity of $\n{P}$ as follows.

\begin{proposition} \label{continuity of Pn}
	$\n{P}: \Xnz \rightarrow \Wnp$ is uniformly bounded. Furthermore, if $n \neq 0$, we have exact norm bound
	\begin{align*}
		&\nrm{\n{P}}_{[\Xnz, \Wnp]} \\
		&\leq \left(\frac{4}{(N-2)\mu -1} + \frac{66}{5} + \left( \frac{\brk{N}+1}{(N-2)\mu -1}+\frac{66}{5} \right) \maxi{\nrm{\beta^\delta \xi_0}_{C_b}, \nrm{\beta^{-\delta}\xi_\infty}_{C_b}}\right) \\
		& \qquad \cdot \Bigg(\left(1+\frac{2}{(N-2)\mu +\frac{5}{6}} \right) \ncbd{\beta \rd_\beta \xi_0}+ \nmtn \ncbd{\beta \xi_0}+\nmtn\ncbd{\beta^2 \rd_\beta \xi_0} \\
		&\qquad \qquad \quad + \nmt \ncbd{\beta^2 \rd^2_\beta \xi_0} +1 \Bigg)+ \nmt \ncbd{\beta \rd_\beta \xi_0}+ \nmtn \ncbd{\beta \xi_0}.
	\end{align*}
\end{proposition}

\begin{proof}
	\begin{align*}
		& \nrm{\n{P}}_{[\Xnz, \Wnp]} = \nrm{\n{P} \Qinv \Pnminv}_{[\Wnz, \Wnp]} \\
		&\leq \nrm{\n{P} \Qinv \Pnminv}_{[\cbdn, \Wnp]}+\nrm{\n{P} \Qinv \Pnminv}_{[\cxizn, \Wnp]}.
	\end{align*}
\end{proof}

\begin{proposition} \label{continuity of PnQ}
	$\n{P}Q: \Xnz \rightarrow \Wnp$ is uniformly bounded. Furthermore, if $n \neq 0$, we have exact norm bound
	\begin{align*}
		&\nrm{\n{P}Q}_{[\Xnz, \Wnp]} \\
		&\leq \left(\frac{4}{(N-2)\mu -1} + \frac{66}{5} + \left( \frac{\brk{N}+1}{(N-2)\mu -1}+\frac{66}{5} \right) \maxi{\nrm{\beta^\delta \xi_0}_{C_b}, \nrm{\beta^{-\delta}\xi_\infty}_{C_b}}\right) \\
		& \qquad \cdot \Bigg(\left(1+\frac{2}{(N-2)\mu +\frac{5}{6}} \right) \ncbd{\beta \rd_\beta \xi_0}+ \nmtn \ncbd{\beta \xi_0}+\nmtn\ncbd{\beta^2 \rd_\beta \xi_0} \\
		&\qquad \qquad \quad + \nmt \ncbd{\beta^2 \rd^2_\beta \xi_0} +1 \Bigg)+ \frac{2}{(N-2)\mu +\frac{5}{6}} \ncbd{\beta \rd_\beta \xi_0}+ \frac{2\brk{N}}{(N-2)\mu + \frac{5}{6}} \ncbd{\beta \xi_0}+2.
	\end{align*}
\end{proposition}

\begin{proof}
	Apply Proposition \ref{continuity of Pn(Q+1)} and Proposition \ref{continuity of Pn} on
	\begin{align*}
		& \nrm{\n{P}Q}_{[\Xnz, \Wnp]} = \nrm{\n{P}(Q+1)-\n{P}}_{[\Xnz, \Wnp]} \\
		&\leq \nrm{\n{P}(Q+1)}_{[\Xnz, \Wnm]} +\nrm{\n{P}}_{[\Xnz, \Wnp]}.
	\end{align*}
\end{proof}

So far, we have proved

\begin{equation} \label{atomic diff operator}
	\begin{aligned}
		id: \Xnz \rightarrow \Wnz ,\quad in \cdot id &: \Xnz \rightarrow \Wnz, \quad Q: \Xnz \rightarrow \Wnm, \\
		\n{P}(Q+1): \Xnz \rightarrow \Wnm, \quad \n{P} &: \Xnz \rightarrow \Wnp, \quad \n{P}Q: \Xnz \rightarrow \Wnp
	\end{aligned}
\end{equation}
are uniformly bounded with explicit norm bound for $n \neq 0$. As $\rd_\phi$ is induced by $in \cdot id$, $\beta \rd_\beta$ is induced by $Q$, $\beta \rd_\varphi(\beta \rd_\beta+1)$ is induced by $-\n{P}(Q+1)$, $\beta \rd_\varphi$ is induced by $-\n{P}$ and $\beta\rd_\varphi \beta \rd_\beta$ is induced by $-\n{P}Q$, we can get the following proposition about the continuity of differential operators.

\begin{proposition} \label{continuity of atomic diff operator}
	\begin{align*}
		id &: X_{0 , N} \rightarrow W_{0, N}, \\
		\rd_\phi &: X_{0 , N} \rightarrow W_{0, N}, \\
		\beta \rd_\beta &: X_{0 , N} \rightarrow W_{-, N}, \\
		\beta \rd_\varphi &: X_{0 , N} \rightarrow W_{+, N}, \\
		\beta \rd_\varphi (\beta \rd_\beta +1) &: X_{0 , N}, \rightarrow W_{-, N}, \\
		\beta \rd_\varphi \beta \rd_\beta &: X_{0 , N} \rightarrow W_{+, N}
	\end{align*}
	is $C^1$.
\end{proposition}

\begin{proof}
	As all operator in Proposition \ref{continuity of atomic diff operator} is linear, its Fr\'echet derivative is same as itself. Then, boundedness of these operators follow directly by uniform boundedness of \eqref{atomic diff operator} and Proposition \ref{continuity between A^s}.
\end{proof}

As each Fourier mode of differential bard operator in \eqref{eq19} $\sim$ \eqref{eq23} is linear combination of operator in \eqref{atomic diff operator}, the following proposition about differential bar operator follows directly.

\begin{proposition} \label{continuity of diff bar operator}
	\begin{align*}
		\dbbeta &: X_{0 , N} \rightarrow W_{0, N}, \\
		\dbvarphi &: X_{0 , N} \rightarrow  W_{+, N}, \\
		\dbvarphibeta &: X_{0 , N} \rightarrow W_{+, N}, \\
		\dbphibeta &: X_{0 , N} \rightarrow W_{0, N}
	\end{align*}
	is $C^1$
\end{proposition}

Remember, our goal is to show that $\br{L}$ given by

\begin{align*}
		 	&\br{L}(\br{\psi},\Omega)=\br{\rd}_\varphi \left( \frac{2\br{\rd}_\beta \br{\psi} \cdot  \br{\rd}_\varphi \br{\psi}}{\left( \bar{\rd}_\varphi +1 \right) \br{\rd}_\beta \br{\psi}} \left( 1+ \left( \frac{\rd_\phi \br{\rd}_\beta \br{\psi}}{2\br{\rd}_\beta} \right)^2  \right) - \frac{\rd_\phi \br{\rd}_\beta \br{\psi} \cdot  \rd_\phi \br{\psi}   }{2  \br{\rd}_\beta \br{\psi}}  \right) \\
		 	& \qquad \qquad + \rd_\phi \left(  \frac{ \left( \bar{\rd}_\varphi +1 \right) \br{\rd}_\beta \br{\psi} \cdot  \rd_\phi \br{\psi} - \rd_\phi \br{\rd}_\beta \br{\psi} \cdot \br{\rd}_\varphi \br{\psi}}{2\br{\rd}_\beta \br{\psi}  }\right)
	+ \frac{\left( \bar{\rd}_\varphi +1 \right) \br{\rd}_\beta \br{\psi} \cdot \left( \br{\rd}_\varphi \br{\psi} \right) ^{-\frac{1}{2\mu}} }{2 \mu} \Omega
\end{align*}
is $C^1$. For this aim, we need to avoid zero value in the denominators in \eqref{L bar}. First, note that trivial solution $\br{\psi}_0$ is a constant function so that only zeroth Fourier mode is non-trivial. Therefore,
\begin{equation*}
	\br{\psi}_0 \in X_{0,N}.
\end{equation*}
Since $\dbbeta: \Xz \rightarrow \Wz$ is continuous by Proposition \ref{continuity of diff bar operator} and $\dbbeta\br{\psi}_0 \equiv -1 \in \Xz$, there exists $\delta_1 >0$ such that
\begin{equation*}
	\br{\psi} \in B_{\Xz}(\br{\psi}_0, \delta_1) \Rightarrow \dbbeta \br{\psi} \in B_{\Wz} \left(-1, \frac{1}{2} \right). \footnote{We denote $\delta$-ball centered at $x_0$ in space $K$ by $B_K (x_0, \delta)$.}
\end{equation*}
As
\begin{equation*}
	\nrm{\dbbeta \br{\psi} - (-1)}_{\Wz} \underset{\text{def}}{=} \nrm{\left( \dbbeta \br{\psi} \right)^{(0)} - (-1)}_{\cbd \oplus \cxiz \oplus \bbC} + \sum_{n \in N\bbZ \backslash \{0\}} \brk{n}^{0.5} \nrm{\left( \dbbeta \br{\psi} \right)^{(n)}}_{\cbd \oplus \cxiz} < \frac{1}{2}
\end{equation*}
and
\begin{equation*}
	\cbd \hookrightarrow C_b, \qquad \cbd \oplus \cxiz \oplus \bbC \hookrightarrow C_b
\end{equation*}
with
\begin{equation*}
	\nrm{f}_{C_b} \leq \nrm{f}_{\cbd \oplus \cxiz \oplus \bbC}, \qquad \nrm{f}_{C_b} \leq \ncbd{f},
\end{equation*}
we have
\begin{equation*}
	\nrm{\left( \dbbeta \br{\psi} \right)^{(0)} - (-1)}_{C_b} + \sum_{n \in N\bbZ \backslash \{0\}} \nrm{\left( \dbbeta \br{\psi} \right)^{(n)}}_{C_b} < \frac{1}{2}.
\end{equation*}
Then,
\begin{equation} \label{eq31}
	\nrm{\dbbeta \br{\psi} -(-1)}_{C_b \left( \RT \right)} \leq \sum_{n \in N\bbZ} \nrm{\left( \dbbeta \br{\psi} \right)^{(n)} - (-1)^{(n)}}_{C_b} < \frac{1}{2},
\end{equation}
so that $\dbbeta \br{\psi}$ is safely away from zero for $\br{\psi} \in \ball{\Xz}{\br{\psi}_0}{\delta_1}$.

Similarly, as $\dbvarphi: \Xz \rightarrow \Wp$ is continuous by Proposition \ref{continuity of diff bar operator} and $\dbvarphi \br{\psi}_0 \equiv 1$, there exists $\delta_2 >0$ such that
\begin{equation*}
	\br{\psi} \in \ball{\Xz}{\br{\psi}_0}{\delta_2} \Rightarrow \dbvarphi \br{\psi} \in \ball{\Wp}{1}{\frac{1}{2}}
\end{equation*}
and
\begin{equation} \label{eq32}
	\nrm{\dbvarphi \br{\psi} -1}_{C_b \left( \RT \right)} < \frac{1}{2} \qquad \text{for} \qquad \br{\psi} \in \ball{\Xz}{\br{\psi}_0}{\delta_2}.
\end{equation}
Also, since $\dbvarphibeta: \Xz \rightarrow \Wp$ is continuous with $\dbvarphibeta \br{\psi}_0 = -2\mu $, there exists $\delta_3 >0$ such that
\begin{equation*}
	\br{\psi} \in \ball{\Xz}{\br{\psi}_0}{\delta_3} \Rightarrow \dbvarphibeta \br{\psi} \in \ball{\Wp}{-2\mu}{\mu}
\end{equation*}
and
\begin{equation} \label{eq33}
	\nrm{\dbvarphibeta \br{\psi} - (-2\mu)}_{C_b \left( \RT \right)}< \mu \qquad \text{for} \qquad \br{\psi} \in \ball{\Xz}{\br{\psi}_0}{\delta_3}.
\end{equation}

Although we have avoided zero-value in denominator through \eqref{eq31}, \eqref{eq32}, and \eqref{eq33}, we continue to bound $\rd_\phi \dbbeta \br{\psi}$ and $\br{\psi}$ for the future aim. Due to the fact that $\rd_\phi, \rd_\phi \dbbeta: \Xz \rightarrow \Wz$ is continuous with $\rd_\phi \br{\psi}_0 = \rd_\phi \dbbeta \br{\psi}_0 \equiv 0$, there exists $\delta_4>0$ such that 
\begin{equation*}
	\br{\psi} \in \ball{\Xz}{\br{\psi}}{\delta_4} \Rightarrow \rd_\phi \br{\psi},  \rd_\phi \dbbeta \br{\psi} \in \ball{\Wz}{0}{1}
\end{equation*}
and
\begin{equation} \label{phi bound}
	\maxi{\nrm{\rd_\phi \br{\psi}}_{C_b\left( \RT \right)},\nrm{\rd_\phi \dbbeta \br{\psi}}_{C_b\left( \RT \right)}} < 1 \qquad \text{for} \qquad \br{\psi} \in \ball{\Xz}{\br{\psi}_0}{\delta_4}
\end{equation}
Lastly, since $id: \Xz \rightarrow \Wz$ is continuous with $\br{\psi}_0 \equiv \frac{1}{2\mu -1}$, there exists $\delta_5>0$ such that
\begin{equation*}
	\br{\psi} \in \ball{\Xz}{\br{\psi}_0}{\delta_5} \Rightarrow \br{\psi}_0 \in \ball{\Wz}{\frac{1}{2\mu-1}}{\frac{1}{4\mu -2}}
\end{equation*}
and
\begin{equation} \label{id bound}
	\nrm{\br{\psi}- \frac{1}{2\mu -1}}_{C_b \left( \RT \right)} \leq \frac{1}{4\mu -2} \qquad \text{for} \qquad \br{\psi} \in \ball{\Xz}{\br{\psi}_0}{\delta_5}.
\end{equation}

Now that we have made denominator in $\br{L}$ strictly away from zero, we will get a little more closer to $C^1$ continuity.

\begin{proposition} \label{continuity of diff beta bar inv}
	There exists $\delta'_1>0$ such that $\delta'_1 < \delta_1$ and a map
	\begin{equation*}
		\frac{1}{\dbbeta}: \ball{\Xz}{\br{\psi}_0}{\delta'_1} \rightarrow \Wz
	\end{equation*}
	defined by $\frac{1}{\dbbeta} \left( \br{\psi} \right) = \frac{1}{\dbbeta \br{\psi}}$ is $C^1$.
\end{proposition}
 \begin{proof}
 	By Proposition \ref{A^s is Banach} and \ref{W algebra}, $\Wz$ is a Banach algebra. Since $\Wz$ contains constant function, $\Wz$ is an unital Banach algebra. Then, there exists $\varepsilon_1 >0$ such that inverse map
 	\begin{equation*}
 		\text{inv}: f(\beta, \phi) \mapsto \frac{1}{f(\beta, \phi)}
 	\end{equation*}
 	is analytic on $\ball{\Wz}{-1}{\varepsilon_1}$. Since
 	\begin{equation*}
 		\dbbeta: \Xz \rightarrow \Wz
 	\end{equation*}
 	is continuous, there exists $\delta'_1 < \delta_1$, where $\delta_1$ is as in \eqref{eq31}, such that
 	\begin{equation*}
 		\dbbeta \left( \ball{\Xz}{\br{\psi}_0}{\delta'_1} \right) \subset \ball{\Wz}{-1}{\varepsilon_1}.
 	\end{equation*}
 	Then, 
 	\begin{equation*}
 		\text{inv} \circ \dbbeta: \ball{\Xz}{\br{\psi}_0}{\delta'_1} \rightarrow	\Wz
 	\end{equation*}
 	is a desired map.
 \end{proof}

\begin{proposition} \label{continuity of diff varphi bar inv}
	There exists $\delta'_2>0$ such that $\delta'_2 < \delta_2$ and two maps
	\begin{align*}
		\frac{1}{\dbvarphi} &: \ball{\Xz}{\br{\psi}_0}{\delta'_2} \rightarrow \Wp, \\
		\left( \dbvarphi \right)^{-\frac{1}{2\mu}} &: \ball{\Xz}{\br{\psi}_0}{\delta'_2} \rightarrow \Wp
	\end{align*}
	defined by
	\begin{align*}
		\frac{1}{\dbvarphi}\left( \br{\psi} \right) &= \frac{1}{\dbvarphi \br{\psi}}, \\
		\left( \dbvarphi \right)^{-\frac{1}{2\mu}} \left( \br{\psi} \right) &= \left( \dbvarphi \br{\psi} \right)^{-\frac{1}{2\mu}}
	\end{align*}
	is $C^1$
\end{proposition}

\begin{proof}
	Use same method in Proposition \ref{continuity of diff beta bar inv}.
\end{proof}

\begin{proposition} \label{continuity of diff crucial}
	There exists $\delta'_3>0$ such that $\delta'_3 < \delta_3$ and a map
	\begin{equation*}
		\frac{\dbvarphi}{\dbvarphibeta}: \ball{\Xz}{\br{\psi}_0}{\delta'_3} \rightarrow \Wz
	\end{equation*}
	defined by
	\begin{equation*}
		\left( \frac{\dbvarphi}{\dbvarphibeta} \right) \left( \br{\psi} \right) = \frac{\dbvarphi \br{\psi}}{\dbvarphibeta \br{\psi}}
	\end{equation*}
	is $C^1$
\end{proposition}

\begin{proof}
	Focus on the range of the map. This is very interesting. Since, as can be seen in Proposition \ref{continuity of diff bar operator}, the range of $\dbvarphi$ and $\dbvarphibeta$ is $\Wp$, it seems natural that the range of $\frac{\dbvarphi}{\dbvarphibeta}$ be the $\Wp$. Let us see why this interesting phenomenon happens.
	
	First, as in the proof of Proposition \ref{continuity of diff beta bar inv}, take $\varepsilon_2>0$ so that inverse map
	\begin{equation*}
		\text{inv}: f(\beta, \phi) \mapsto \frac{1}{f(\beta, \phi)}
	\end{equation*}
	is analytic on $\ball{\Wz}{-2\mu}{\varepsilon_2}$. Note that
	\begin{align*}
		\frac{\dbvarphibeta \br{\psi}}{\dbvarphi \br{\psi}} \underset{\eqref{diff beta bar}, \eqref{diff varphi bar}}&{=} \frac{(\beta \rd_\varphi +2 \mu)\left(\beta \rd_\beta + (1-2\mu)\right) \br{\psi}}{\left(\beta \rd_\varphi + (2\mu -1 ) ) \br{\psi}\right)} \\
		&= \frac{\beta \rd_\varphi (\beta \rd_\beta +1) \br{\psi} - 2\mu \beta \rd_\varphi \br{\psi} + 2\mu \beta \rd_\beta \br{\psi} + 2\mu (1-2\mu ) \br{\psi}}{\left(\beta \rd_\varphi + (2\mu -1 ) ) \br{\psi}\right)} \\
		&= \frac{\beta \rd_\varphi (\beta \rd_\beta +1) \br{\psi}+2\mu \beta \rd_\beta \br{\psi}-2\mu \left( \beta \rd_\varphi + (2\mu -1) \right) \br{\psi}}{\left(\beta \rd_\varphi + (2\mu -1 ) ) \br{\psi}\right)} \\
		&= \frac{\beta \rd_\varphi (\beta \rd_\beta +1) \br{\psi}}{\dbvarphi \br{\psi}} + 2\mu \frac{\beta \rd_\beta \br{\psi}}{\dbvarphi \br{\psi}} -2\mu
	\end{align*}
	As
	\begin{equation*}
		\beta \rd_\varphi (\beta \rd_\beta +1): \Xz \rightarrow \Wm, \qquad \beta \rd_\beta : \Xz \rightarrow \Wm
	\end{equation*}
	is $C^1$ by Proposition \ref{continuity of atomic diff operator} and
	\begin{equation*}
		\frac{1}{\dbvarphi}: \ball{\Xz}{\trisol}{\delta'_2} \rightarrow \Wp
	\end{equation*}
	defined as in Proposition \ref{continuity of diff varphi bar inv} is $C^1$,
	\begin{equation*}
		\frac{\beta \rd_\varphi (\beta \rd_\beta +1)}{\dbvarphi}: \ball{\Xz}{\trisol}{\delta'_2} \underset{\eqref{pm embedding} }{\rightarrow}  \Wm \hookrightarrow \Wz, \qquad\frac{\beta \rd_\beta}{\dbvarphi}: \ball{\Xz}{\trisol}{\delta'_2} \underset{\eqref{pm embedding} }{\rightarrow} \Wm \hookrightarrow \Wz
	\end{equation*}
	defined by
	\begin{equation*}
		\frac{\beta \rd_\varphi (\beta \rd_\beta +1)}{\dbvarphi} \left( \br{\psi} \right) = \frac{\beta \rd_\varphi (\beta \rd_\beta +1) \br{\psi}}{\dbvarphi \br{\psi}}, \qquad \frac{\beta \rd_\beta}{\dbvarphi} \left( \br{\psi} \right) = \frac{\beta \rd_\beta  \br{\psi}}{\dbvarphi \br{\psi}}
	\end{equation*}
	is $C^1$ by Proposition \ref{W algebra}. Since the constant function $2\mu$ is an element of $\Wz$,
	\begin{equation*}
		\frac{\dbvarphibeta}{\dbvarphi}: \ball{\Xz}{\trisol}{\delta'_2} \rightarrow \Wz
	\end{equation*}
	is $C^1$ as the sum of $C^1$ function. If we take $\delta'_3 < \text{min(}\delta'_2, \delta_3$) so that
	\begin{equation*}
		\frac{\dbvarphibeta}{\dbvarphi} \left( \ball{\Xz}{\trisol}{\delta'_3} \right) \subset \ball{\Wz}{-2\mu}{\varepsilon_2},
	\end{equation*}
	then the map
	\begin{equation*}
		\frac{\dbvarphi}{\dbvarphibeta} \coloneqq \text{inv} \circ \frac{\dbvarphibeta}{\dbvarphi}: \ball{\Xz}{\trisol}{\delta'_3} \rightarrow \Wz
	\end{equation*}
	is $C^1$.
\end{proof}

\begin{proposition} \label{continuity of L bar intermediate}
 	There exists $\delta>0$ such that
 	\begin{equation*}
 		R: \ball{\Xz}{\trisol}{\delta} \rightarrow \Wz, \qquad S: \ball{\Xz}{\trisol}{\delta} \rightarrow \Wp
 	\end{equation*}
 	defined by
 	\begin{align*}
 		R \left( \br{\psi} \right) &= \frac{2\br{\rd}_\beta \br{\psi} \cdot  \br{\rd}_\varphi \br{\psi}}{\left( \bar{\rd}_\varphi +1 \right) \br{\rd}_\beta \br{\psi}} \left( 1+ \left( \frac{\rd_\phi \br{\rd}_\beta \br{\psi}}{2\br{\rd}_\beta} \right)^2  \right) - \frac{\rd_\phi \br{\rd}_\beta \br{\psi} \cdot  \rd_\phi \br{\psi}   }{2  \br{\rd}_\beta \br{\psi}}, \\	
 		S \left( \br{\psi} \right) &= \frac{ \left( \bar{\rd}_\varphi +1 \right) \br{\rd}_\beta \br{\psi} \cdot  \rd_\phi \br{\psi} - \rd_\phi \br{\rd}_\beta \br{\psi} \cdot \br{\rd}_\varphi \br{\psi}}{2\br{\rd}_\beta \br{\psi}  }
 	\end{align*}
 	is $C^1$
\end{proposition}

\begin{proof}
	Use continuity results on each differential operator (Proposition \ref{continuity of atomic diff operator}, \ref{continuity of diff bar operator}, \ref{continuity of diff beta bar inv}, \ref{continuity of diff crucial}) and the fact that product of $C^1$ map is $C^1$ by Proposition \ref{W algebra}. The constant $\delta$ could be chosen $\delta=\text{min}(\delta'_1, \delta'_2, \delta'_3, \delta_4, \delta_5)$ where $\delta'_1, \delta'_2, \delta'_3$ are the one in Proposition \ref{continuity of diff beta bar inv}, \ref{continuity of diff varphi bar inv}, \ref{continuity of diff crucial} and $\delta_4, \delta_5$ as in \eqref{phi bound}, \eqref{id bound}.
\end{proof}

Note that Proposition \ref{continuity of L bar intermediate} gives $C^1$ continuity of the ones inside the differential operator $\dbvarphi, \rd_\phi$ of $\br{L} \left( \br{\psi}, \Omega \right)$. To complete the answer for $C^1$ continuity of $\br{L}$, we need to show that $\dbvarphi, \rd_\phi$ are also $C^1$. But what is the range space for $\dbvarphi, \rd_\phi$? To make $\Pnp \Pnm (Q+1)$ isometric isomorphism, the answer should be
\begin{equation*}
	\Zz \coloneqq \calA^{0.5} \left( \Znz \right),
\end{equation*}
where $\Znz$ is as in Definition \ref{X, Z space}. And then, the way we use to show continuity of
\begin{equation*}
	\dbvarphi: \Wz \rightarrow \Zz , \qquad \rd_\phi: \Wp \rightarrow \Zz
\end{equation*}
is almost same as before; decompose the operator into each Fourier mode and show uniform boundedness of it. However, before going further, we verify $\Znz = \Pnp \Wnz$ is $\brk{n}^q$ - uniformly embedded in Banach space so that we can actually define the space $\calA^{0.5} \left( \Znz \right)$.

\begin{proposition}
	$\Znz=\Pnp \Wnz$ is $\brk{n}^1$ - uniformly embedded in
	\begin{equation*}
		(Q+1) \cbd + \beta \cbd + \cbd + \bbC + \cxiz \in \calS' \left( \br{\bbR}_+ \right)
	\end{equation*}
\end{proposition}

\begin{proof}
	First, suppose $n \neq 0$ so that $\Wnz=\cbd \oplus \cxiz$. For $f \in \cbd$,
	\begin{equation} \label{eq34}
		\begin{aligned}
			& \nrm{\Pnp f}_{(Q+1) \cbd + \beta \cbd + \cbd + \bbC + \cxiz} \\
			\underset{\eqref{Pnp}}&{=} \nrm{\left(\beta(\rd_\beta - in ) - \left( (2+n)\mu -1\right) \right)f}_{(Q+1) \cbd + \beta \cbd + \cbd + \bbC + \cxiz}  \\
			&= \Big\Vert \underbrace{(Q+1)f}_{\in (Q+1) \cbd} - \underbrace{n \cdot i\beta f}_{\beta \cbd} + \underbrace{\left( (2+n)\mu -2 \right) f}_{\in \cbd} \Big\Vert _{(Q+1) \cbd + \beta \cbd + \cbd + \bbC + \cxiz} \\
			&\leq \ncbd{f}+ \brk{n} \ncbd{f} + \abs{(2+n)\mu -2} \ncbd{f} \\
			&\lesssim \brk{n} \ncbd{f} \\
			&= \brk{n} \nrm{\Pnp f}_{\Pnp \Wnz}.
		\end{aligned}
	\end{equation}
	
	Also,
	
	\begin{equation} \label{eq35}
		\begin{aligned}
			&\nrm{\Pnp \xi_0}_{(Q+1) \cbd + \beta \cbd + \cbd + \bbC + \cxiz} \\
			&\leq \Big\Vert \underbrace{\beta \rd_\beta \xi_0}_{\in \cbd} - \underbrace{in \beta \xi_0}_{\cbd} - \underbrace{\left( (2+n)\mu -1 \right) \xi_0}_{\in \bbC \xi_0} \Big\Vert_{(Q+1) \cbd + \beta \cbd + \cbd + \bbC + \cxiz} \\
			&\leq \ncbd{\beta \rd_\beta \xi_0} + \abs{n} \ncbd{\beta \xi_0} + \abs{(2+n)\mu -1} \nrm{\xi_0}_{\cxiz} \\
			&\lesssim \brk{n} \Big( \underbrace{\ncbd{\beta \rd_\beta \xi_0} + \ncbd{\beta \xi_0} +1}_{\text{constant}} \Big) \\
			&\lesssim \brk{n} \\
			&= \brk{n} \nrm{\Pnp \xi_0}_{\Pnp \Wnz}.
		\end{aligned}
	\end{equation}
	
	Combining \eqref{eq34} and \eqref{eq35}, we have $\Znz = \Pnp \Wnz$ is $\brk{n}^1$ - uniformly embedded in Banach space $(Q+1) \cbd + \beta \cbd + \cbd + \bbC + \cxiz$ for $n \neq 0$. The only difference for $W^{(0)}_{0,N}$ compared to $\Wnz \, \, (n \neq 0)$ is $\bbC$. Since
	\begin{equation*}
		P^{(0)}_+ c=\beta \rd_\beta c - (2\mu -1)c = \underbrace{-(2\mu-1)c}_{\in \bbC},
	\end{equation*}
	we get the result without any trouble.
\end{proof}

Now, we verify uniform boundedness of
\begin{equation*}
	id: \Wnp \rightarrow \Znz , \qquad \n{P}: \Wnz \rightarrow \Znz.
\end{equation*}
Since $\Wnp$ contains $\cbd, \cxiz, \cxii$ regardless of $n$, we need the following lemma.

\begin{lemma} \label{Pninv in xii}
	$\Pnpinv: \cxiin \rightarrow \cbdn \oplus \delta_{0n} \cxiin$ is uniformly bounded. furthermore, if $n \neq 0$, we have an exact norm bound
	\begin{equation} \label{Pninv xii}
		\begin{aligned}
			&\nrm{\Pnpinv}_{[\cxiin, \cbdn]}\\
			&\leq \frac{1}{\brk{n}} \left( \frac{1.05}{(N-2)\mu + \frac{5}{6}}   \left( \ncbd{\rd_\beta \xi_\infty}+2\mu \ncbd{\frac{\xi_\infty}{\beta}} \right)+\left( \frac{\brk{N}\mu}{(N-2)\mu +\frac{5}{6}}+1.05 \right) \ncbd{\frac{\xi_\infty}{\beta}}\right).
		\end{aligned}
	\end{equation}
\end{lemma}

\begin{proof}
	We first treat $n=0$ case. From
	\begin{equation*}
		P^{(0)}_+ \xi_\infty = \beta \rd_\beta \xi_\infty - (2\mu -1) \xi_\infty,
	\end{equation*}
	we have
	\begin{equation*}
		\left( P^{(0)}_+ \right)^{-1} \xi_\infty = \frac{1}{2\mu -1} \Bigg( \underbrace{\left( P^{(0)}_+ \right)^{-1} (\beta \rd_\beta \xi_\infty)}_{\underset{\text{Prop } \ref{D^{-1} in C_b^delta}} {\in \cbd}} - \underbrace{\xi_\infty}_{\in \cxii} \Bigg)
	\end{equation*}
	with
	\begin{equation*}
		\nrm{\left( P^{(0)}_+ \right)^{-1}}_{[\cxii, \cbd \oplus \cxii]} \leq \frac{1}{\abs{2\mu -1}} \left( \frac{\ncbd{\beta \rd_\beta \xi_\infty}}{\dist \left( [-\delta, \delta], 2\mu -1 \right)}+1 \right).
	\end{equation*}
	
	Now, we deal with $n \in N \bbZ \backslash \{0\}$ case. Unfortunately, above strategy does not work since $in \beta \xi_\infty$ does not belong to $C_b$. The trick is to consider $\frac{\xi_\infty}{\beta}$ which belongs to $\cbd$. From
	\begin{align*}
		\Pnp \left( \frac{\xi_\infty}{\beta} \right) &= \beta \rd_\beta \left(\frac{\xi_\infty}{\beta} \right)- in \beta \cdot \frac{\xi_\infty}{\beta} -  \left((2+n)\mu -1 \right)  \frac{\xi_\infty}{\beta} \\
		&= \rd_\beta \xi_\infty - in\xi_\infty - (2+n)\mu \frac{\xi_\infty}{\beta},
	\end{align*}
	we can represent $\xi_\infty$ as
	\begin{equation*}
		\xi_\infty = \frac{1}{in} \Bigg( \underbrace{\rd_\beta \xi_\infty}_{\in \cbd}-(2+n)\mu \underbrace{\frac{\xi_\infty}{\beta}}_{\in \cbd} - \Pnp \left( \frac{\xi_\infty}{\beta} \right) \Bigg).
	\end{equation*}
	Then,
	\begin{align*}
		& \ncbd{\Pnpinv \xi_\infty} \\
		&= \bigg\Vert \underbrace{\Pnpinv(\rd_\beta \xi_\infty) -(2+n)\mu  \Pnpinv \left( \frac{\xi_\infty}{\beta} \right)-\frac{\xi_\infty}{\beta}}_{\in \cbd}  \bigg\Vert_{\cbd} \\
		\underset{\eqref{D^{-1} C_b}}&{\leq} \frac{1}{\abs{n}} \left( \frac{\ncbd{\rd_\beta \xi_\infty}+\abs{(2+n)\mu}\ncbd{\frac{\xi_\infty}{\beta}}}{\distPnp} + \ncbd{\frac{\xi_\infty}{\beta}} \right) \\
		\underset{\eqref{dist bound}}&{\leq} \frac{1}{\abs{n}} \left( \nmtn \cdot \frac{1}{\brk{n}} \left( \ncbd{\rd_\beta \xi_\infty}+ (2+\abs{n})\mu \ncbd{\frac{\xi_\infty}{\beta}}\right) + \ncbd{\frac{\xi_\infty}{\beta}} \right) \\
		&= \frac{1}{\brk{n}} \left( \nmtn \cdot \frac{1}{\abs{n}} \left( \ncbd{\rd_\beta \xi_\infty}+ (2+\abs{n})\mu \ncbd{\frac{\xi_\infty}{\beta}}\right)+ \frac{\brk{n}}{\abs{n}}\ncbd{\frac{\xi_\infty}{\beta}} \right) \\
		& \quad \left(\text{Use the fact that for } n \in N\bbZ \backslash \set{0}, \frac{\brk{n}}{\abs{n}} \text{ achieves maximum at } n=N\geq 4 \text{ with } \frac{\brk{4}}{4} \leq 1.05 \right) \\
		& \leq \frac{1}{\brk{n}} \left( \frac{1.05}{(N-2)\mu + \frac{5}{6}}   \left( \ncbd{\rd_\beta \xi_\infty}+2\mu \ncbd{\frac{\xi_\infty}{\beta}} \right)+\left( \frac{\brk{N}\mu}{(N-2)\mu +\frac{5}{6}}+1.05 \right) \ncbd{\frac{\xi_\infty}{\beta}}\right).
		\end{align*}
\end{proof}

\begin{proposition} \label{continuity of id second}
	$id: \Wnp \rightarrow \Znz$ is uniformly bounded. Furthermore, if $n \neq 0$, we have an exact norm bound
	\begin{equation} \label{continuity of id second formula}
		\begin{aligned}
		&\nrm{id}_{[\Wnp, \Znz]} \\
		&\leq \Bigg( \nmtn \left( \frac{\nrm{\beta \rd_\beta \xi_0}_{\cbd}+\brk{N} \ncbd{\beta \xi_0}}{(N-2)\mu +\frac{5}{6}} +2 \right) + \frac{1.05}{(N-2)\mu +\frac{5}{6}} \left( \ncbd{\rd_\beta \xi_\infty}+2\mu \ncbd{\frac{\xi_\infty}{\beta}} \right) \\
		& \qquad + \left( \frac{\brk{N}\mu}{(N-2)\mu +\frac{5}{6}}+1.05 \right) \ncbd{\frac{\xi_\infty}{\beta}} \Bigg) \cdot \frac{1}{\brk{n}}.
	\end{aligned}
	\end{equation}
	
\end{proposition}

\begin{proof}
	By direct calculation, we have
	\begin{align*}
		&\nrm{id}_{[\Wnp, \Znz]}\\
		&= \nrm{\Pnpinv}_{[\Wnp, \Wnz]} \\
		&\leq \nrm{\Pnpinv}_{[\cbdn, \Wnz]}+\nrm{\Pnpinv}_{[\cxizn, \Wnz]}+\nrm{\Pnpinv}_{[\cxiin,\Wnz]} \\
		\underset{\substack{\text{Cor } \ref{Pninv in cbd}, \ref{Pninv in xiz} \\ \text{Lem }\ref{Pninv in xii}}}&{\leq} \nmtn \cdot \frac{1}{\brk{n}} + \frac{\brk{N}}{(N-2)\mu + \frac{5}{6}} \left( \frac{\nrm{\beta \rd_\beta \xi_0}_{\cbd}+\brk{N} \ncbd{\beta \xi_0}}{(N-2)\mu +\frac{5}{6}} +1 \right) \cdot \frac{1}{\brk{n}} \\
		& \qquad +  \left( \frac{1.05}{(N-2)\mu + \frac{5}{6}}   \left( \ncbd{\rd_\beta \xi_\infty}+2\mu \ncbd{\frac{\xi_\infty}{\beta}} \right)+\left( \frac{\brk{N}\mu}{(N-2)\mu +\frac{5}{6}}+1.05 \right) \ncbd{\frac{\xi_\infty}{\beta}}\right) \cdot \frac{1}{\brk{n}} \\
		&= \Bigg( \nmtn \left( \frac{\nrm{\beta \rd_\beta \xi_0}_{\cbd}+\brk{N} \ncbd{\beta \xi_0}}{(N-2)\mu +\frac{5}{6}} +2 \right) + \frac{1.05}{(N-2)\mu +\frac{5}{6}} \left( \ncbd{\rd_\beta \xi_\infty}+2\mu \ncbd{\frac{\xi_\infty}{\beta}} \right) \\
		& \qquad + \left( \frac{\brk{N}\mu}{(N-2)\mu +\frac{5}{6}}+1.05 \right) \ncbd{\frac{\xi_\infty}{\beta}} \Bigg) \cdot \frac{1}{\brk{n}}.
	\end{align*}
\end{proof}

To show continuity of $\dbvarphi$ from $\Wz$ to $\Zz$, we need uniform boundeness of $\n{P}$.

\begin{proposition} \label{continuity of Pn second}
	$\n{P}: \Wnz \rightarrow \Znz$ is uniformly bounded. Furthermore, if $n \neq 0$, we have an exact norm bound
	\begin{align*}
		& \nrm{\n{P}}_{[\Wnz, \Znz]} \\
		&\leq 1+ \Bigg( \nmtn \left( \frac{\nrm{\beta \rd_\beta \xi_0}_{\cbd}+\brk{N} \ncbd{\beta \xi_0}}{(N-2)\mu +\frac{5}{6}} +2 \right) + \frac{1.05}{(N-2)\mu +\frac{5}{6}} \left( \ncbd{\rd_\beta \xi_\infty}+2\mu \ncbd{\frac{\xi_\infty}{\beta}} \right) \\
		& \qquad + \left( \frac{\brk{N}\mu}{(N-2)\mu +\frac{5}{6}}+1.05 \right) \ncbd{\frac{\xi_\infty}{\beta}} \Bigg) \cdot \frac{\abs{(2+n)\mu -1}}{\brk{n}}.
	\end{align*}
\end{proposition}

\begin{proof}
	We can estimate norm of $\n{P}: \Wnz \rightarrow \Znz$ by
	\begin{align*}
		& \nrm{\n{P}}_{[\Wnz, \Znz]} \\
		&= \nrm{\Pnp + \left( (2+n)\mu -1 \right) \cdot id}_{[\Wnz, \Znz]} \\
		&\leq \bigg\Vert \underbrace{\Pnp}_{\text{isometry}} \bigg\Vert_{[\Wnz, \Znz]} + \abs{(2+n)\mu -1} \nrm{id}_{[\Wnz, \Znz]} \\
		&= 1+\abs{(2+n)\mu -1} \nrm{id}_{[\Wnz, \Znz]}.
	\end{align*}
	Hence, it is uniformly bounded by Proposition \ref{continuity of id second}. Also, for $n \neq 0$,
	\begin{align*}
		&\nrm{\n{P}}_{[\Wnz, \Znz]} \\
		&\leq 1+ \abs{(2+n)\mu -1} \nrm{id}_{[\Wnz, \Znz]} \\
		&\leq 1+ \Bigg( \nmtn \left( \frac{\nrm{\beta \rd_\beta \xi_0}_{\cbd}+\brk{N} \ncbd{\beta \xi_0}}{(N-2)\mu +\frac{5}{6}} +2 \right) + \frac{1.05}{(N-2)\mu +\frac{5}{6}} \left( \ncbd{\rd_\beta \xi_\infty}+2\mu \ncbd{\frac{\xi_\infty}{\beta}} \right) \\
		& \qquad + \left( \frac{\brk{N}\mu}{(N-2)\mu +\frac{5}{6}}+1.05 \right) \ncbd{\frac{\xi_\infty}{\beta}} \Bigg) \cdot \frac{\abs{(2+n)\mu -1}}{\brk{n}}.
	\end{align*}
\end{proof}

Observing how $\dbvarphi$ and $\rd_\phi$ are decomposed into each Fourier mode in \eqref{eq20} and \eqref{eq21}, the following proposition comes directly from Proposition \ref{continuity between A^s}.

\begin{proposition} \label{continuity of second diff operator}
	\begin{align*}
		\dbvarphi &: \Wz \rightarrow \Zz, \\
		\rd_\phi &: \Wp \rightarrow \Zz
	\end{align*}
	is $C^1$.
\end{proposition}

\begin{corollary} \label{continuity of L bar almost}
	$\dbvarphi \circ R+\rd_\phi \circ S$ is $C^1$ from $\ball{\Xz}{\trisol}{\delta}$ to $\Zz$ where $R, S, \delta$ are defined as in Proposition \ref{continuity of L bar intermediate}. 
\end{corollary}

We almost reach to $C^1$ continuity of $\br{L}$. The last thing to deal with is the term
\begin{equation*}
	\frac{\left( \bar{\rd}_\varphi +1 \right) \br{\rd}_\beta \br{\psi} \cdot \left( \br{\rd}_\varphi \br{\psi} \right) ^{-\frac{1}{2\mu}} }{2 \mu} \Omega.
\end{equation*}
As a result of Proposition \ref{continuity of diff bar operator}, \ref{continuity of diff varphi bar inv}, we already know
\begin{equation} \label{eq36}
	\frac{\left( \bar{\rd}_\varphi +1 \right) \br{\rd}_\beta \br{\psi} \cdot \left( \br{\rd}_\varphi \br{\psi} \right) ^{-\frac{1}{2\mu}} }{2 \mu}  \in \calA^{0.5} \left( \Wnp \right)
\end{equation}
Then, what is the choice for Banach space for $\Omega$ so that
\begin{equation*}
	\frac{\left( \bar{\rd}_\varphi +1 \right) \br{\rd}_\beta \br{\psi} \cdot \left( \br{\rd}_\varphi \br{\psi} \right) ^{-\frac{1}{2\mu}} }{2 \mu} \Omega
\end{equation*}
could be an element of $\calA^{0.5} \left( \Znz \right)$?

Let $\bbC$ be a space of constant function on $\br{\bbR}_+$, and similar to $\n{W}_{\triangle, N}$ space, define $\n{\bbC}_N$ by
\begin{equation*}
	\n{\bbC}_N = \begin{cases}
		\bbC & (n \in N \bbZ) \\
		0 & (n \notin N\bbZ)
	\end{cases}
\end{equation*}
It is obvious  that $\cn \unibed C_b \in \calS' \left( \br{\bbR}_+ \right)$ and for all $i, j \in \bbZ$, embedding
\begin{equation} \label{eq37}
	W^{(i)}_{+, N} \cdot \bbC^{(j)}_N \unibed W^{(i+j)}_{+, N}
\end{equation}
holds. then, if we define a Banach space $Y_{0,N}$ by
\begin{equation*}
	\Yz \coloneqq \calA^{-0.5} \left( \cn \right),
\end{equation*}
the following proposition holds.

\begin{proposition} \label{embedding to Zn}
	\begin{equation*}
		\Wp \cdot \Yz \hookrightarrow \Zz.
	\end{equation*}
\end{proposition}

\begin{proof}
	By Proposition \ref{A^s algebra}, \eqref{eq37} implies
	\begin{equation} \label{eq38}
		\underbrace{\calA^{0.5} \left( \Wnp \right)}_{\Wp} \cdot \underbrace{\calA^{-0.5} \left( \cn \right)}_{\Yz} \hookrightarrow \calA^{-0.5} \left( \Wnp \right)
	\end{equation}
	Therefore, we have to show
	\begin{equation} \label{eq39}
		\calA^{-0.5} \left( \Wnp \right) \hookrightarrow \calA^{0.5} \left( \Znz \right)
	\end{equation}
	to complete the proof. For $f= \set{\n{f}}_{n \in \bbZ} \in \calA^{-0.5} \left( \Wnp \right)$, we have
	\begin{align*}
		\nrm{f}_{\calA^{0.5} \left( \Znz \right)} &= \sum_{n \in \bbZ} \brk{n}^{0.5} \nrm{\n{f}}_{\Znz} \\
		&\leq \sum_{n \in \bbZ} \brk{n}^{0.5} \nrm{id}_{[\Wnp, \Znz]} \nrm{\n{f}}_{\Wnp} \\
		&= \nrm{id}_{[W^{(0)}_{+, N}, Z^{(0)}_{0, N}]} \nrm{f^{(0)}}_{W^{(0)}_{+,N}} + \sum_{n \in N\bbZ \backslash \{0\}} \brk{n}^{0.5} \nrm{id}_{[\Wnp, \Znz]} \nrm{\n{f}}_{\Wnp} .
	\end{align*}
	Using uniform norm bound \eqref{continuity of id second formula}, i.e.,
	\begin{equation*}
		\nrm{id}_{[\Wnp, \Znz]} \leq \frac{M}{\brk{n}} \qquad \text{for} \qquad n\neq 0
	\end{equation*}
	we have
	\begin{align*}
		\nrm{f}_{\calA^{0.5} \left( \Znz \right)} &\leq \nrm{id}_{[W^{(0)}_{+, N}, Z^{(0)}_{0, N}]} \nrm{f^{(0)}}_{W^{(0)}_{+, N}} + M \sum_{n \in N\bbZ \backslash \{0\}} \brk{n}^{0.5} \brk{n}^{-1} \nrm{\n{f}}_{\Wnp} \\
		& \leq \left( \nrm{id}_{[W^{(0)}_{+, N}, Z^{(0)}_{0, N}]} +M \right) \sum_{n \in N \bbZ} \brk{n}^{-0.5} \nrm{\n{f}}_{\Wnp} \\
		& =  \left( \nrm{id}_{[W^{(0)}_{+, N}, Z^{(0)}_{0, N}]} +M \right) \nrm{f}_{\calA^{-0.5} \left( \Wnp \right)}.
	\end{align*}
	Combining \eqref{eq38} and \eqref{eq39}, then we get the result.
\end{proof}

As a function
\begin{equation*}
	\frac{\dbvarphibeta \cdot \left( \dbvarphi \right)^{-\frac{1}{2\mu}}}{2\mu} : \ball{\Xz}{\trisol}{\delta'_2} \rightarrow \Wp
\end{equation*}
defined by
\begin{equation*}
	\left( \frac{\dbvarphibeta \cdot \left( \dbvarphi \right)^{-\frac{1}{2\mu}}}{2\mu} \right) \left( \br{\psi} \right) =\frac{\dbvarphibeta \br{\psi} \cdot \left( \dbvarphi \br{\psi} \right)^{-\frac{1}{2\mu}}}{2\mu}
\end{equation*}
is $C^1$ by Proposition \ref{continuity of diff bar operator}, \ref{continuity of diff varphi bar inv}, Proposition \ref{embedding to Zn} tells us
\begin{equation} \label{eq40}
	\left( \br{\psi}, \Omega \right) \rightarrow \frac{\dbvarphibeta \br{\psi} \cdot \left( \dbvarphi \br{\psi} \right)^{-\frac{1}{2\mu}}}{2\mu} \Omega
\end{equation}
is $C^1$ map from $\Xz \times \Yz$ to $\Zz$.

Then, finally, we have
\begin{proposition} \label{continuity of L bar}
	There exists $\delta'_3>0$ such that
	\begin{equation*}
		\br{L}: \ball{\Xz}{\trisol}{\delta'_3} \times \Yz \rightarrow \Zz
	\end{equation*}
	is $C^1$
\end{proposition}

\begin{proof}
	Corollary \ref{continuity of L bar almost} and \eqref{eq40}.
\end{proof}

\bigskip

\subsection{Isomorphism in high periodicity and implicit function theorem}
\label{subsec: Isomorphism in high periodicity and implicit function theorem}

Proposition \ref{continuity of L bar} gives positive answer to \hyperref[Q1]{Question 1} and \hyperref[Q2]{Question 2}. To apply implicit function theorem on $\br{L}$, we have to show
\begin{equation*}
	\frac{\rd \br{L}}{\rd \br{\psi}}\left( \br{\psi}_0, \Omega_0 \right) =  \frac{1}{2\mu^2}\left( (\br{\rd}_\varphi \br{\rd}_\varphi +\mu^2 \rd_\phi \rd_\phi)(\br{\rd}_\beta +2\mu)+(2\mu-1)(\br{\rd}_\beta +\br{\rd}_\varphi )  \right)
\end{equation*}
is an isomorphism, and as always our philosophy is that we decompose $\frac{\rd \br{L}}{\rd \br{\psi}}\left( \br{\psi}_0, \Omega_0 \right)$ into each Fourier mode and verify
\begin{equation*}
	\left( \frac{\rd \br{L}}{\rd \br{\psi}}\left( \br{\psi}_0, \Omega_0 \right) \right)^{(n)} = \frac{1}{2\mu^2} \left( \Pnp \Pnm (Q+1) + (2\mu -1) \left( Q-\n{P} \right) \right)
\end{equation*}
is an isomorphism with uniformly lower and upper bounded norm, i.e., there exists $m, M>0$ such that
\begin{equation} \label{eq41}
	m \leq \nrm{\Pnp \Pnm (Q+1) + (2\mu -1) \left( Q-\n{P} \right)}_{[\Xnz, \Znz]} \leq M
\end{equation}
for all $n \in \bbZ$. Here, note that by the way of construction of $\Xnz, \Znz$ space,
\begin{equation*}
	\Pnp \Pnm (Q+1): \Xnz \rightarrow \Znz
\end{equation*}
is an isometry, hence has norm exactly 1. Therefore, to show \eqref{eq41}, it suffices to show that
\begin{equation} \label{eq42}
	\Big\Vert (2\mu -1) \underbrace{\Big( Q -\n{P} \Big)}_{in \beta} \Big\Vert_{[\Xnz, \Znz]} \leq \alpha <1
\end{equation}
for all $n \in \bbZ$. However, observe that for $n=0$, $Q=P^{(0)}$ so that $Q-\n{P}$ is a zero map. Therefore, it suffices to show \eqref{eq42} for $n \in N\bbZ \backslash \{0\}$ and this is why we have estimated norm of differential operators only when $n \in N\bbZ \backslash \{0\}$. For $n \in N \bbZ \backslash \set{0}$,

\begin{equation} \label{eq43}
	\begin{aligned}
		& \nrm{(2\mu -1) in\beta}_{[\Xnz, \Znz]} \\
		&\leq (2\mu -1)\nrm{id}_{[\Wnp, \Znz]} \nrm{in \beta}_{[\Xnz, \Wnp]} \\
		&\leq (2\mu -1)\nrm{id}_{[\Wnp, \Znz]} \\
		& \qquad \cdot \left( \nrm{in \beta \Qinv \Pnminv}_{[\cbdn, \cbdn \oplus \cxiin]}+ \nrm{in \beta \Qinv \Pnminv}_{[\cxizn, \cbdn \oplus \cxiin]} \right) \\
		\underset{\text{Prop }\ref{continuity of id second}, \eqref{continuity of inbeta}}&{\leq} \frac{2\mu -1}{\brk{N}}  \Bigg[  \nmtn \left( \frac{\nrm{\beta \rd_\beta \xi_0}_{\cbd}+\brk{N} \ncbd{\beta \xi_0}}{(N-2)\mu +\frac{5}{6}} +2 \right) \\
		& \qquad + \frac{1.05}{(N-2)\mu +\frac{5}{6}} \left( \ncbd{\rd_\beta \xi_\infty}+2\mu \ncbd{\frac{\xi_\infty}{\beta}} \right) + \left( \frac{\brk{N}\mu}{(N-2)\mu +\frac{5}{6}}+1.05 \right) \ncbd{\frac{\xi_\infty}{\beta}} \Bigg]\\
		& \qquad \cdot \Bigg[ \left( \nmt + \frac{66}{5} + \left( \frac{\brk{N}+1}{(N-2)\mu +\frac{5}{6}}+\frac{66}{5} \right) \maxi{\nrm{\beta^\delta \xi_0}_{C_b}, \nrm{\beta^{-\delta}\xi_\infty}_{C_b}} \right) \\
		& \qquad \qquad \Bigg( \left( 1+ \frac{2}{(N-2)\mu + \frac{5}{6}} \right) \ncbd{\beta \rd_\beta \xi_0} + \nmtn \ncbd{\beta \xi_0} + \nmtn \ncbd{\beta^2 \rd_\beta \xi_0} \\
		& \qquad \qquad \qquad + \nmt \ncbd{\beta^2 \rd^2_\beta \xi_0}+1 \Bigg)+ \nmtn \ncbd{\beta \xi_0} \Bigg].
	\end{aligned}
\end{equation}

Now, it is time to specify $\xi_\infty$ so that we can calculate $\cbd$ norm of various functions in the last inequality. Let
\begin{equation*}
	\eta(\beta) = \begin{cases}
		C e^{\frac{1}{\left( \beta - \frac{3}{2} \right)^2 -\frac{1}{4}}} & (1<\beta<2) \\
		0 & (\text{otherwise}),
	\end{cases}
\end{equation*}
where the constant $C$ is chosen so that $\int_0^\infty \eta(\beta)\,d\beta=1$.\footnote{$C$ is approximately 142.25034 and local maximum of $\eta(\beta)$ is approximately 2.60541}. This is nothing but translated standard mollifier. And then, define $\xi_\infty, \xi_0$ by
\begin{equation*}
	\xi_\infty = \int_0^\beta \eta (s) \, ds , \qquad \xi_0 = 1-\xi_\infty.
\end{equation*}
With this explicit function and the fact that $\frac{1}{6} < \delta < \frac{1}{2}$ by \eqref{def of delta}, we have

\begin{equation} \label{exact bound of cbd norm}
	\begin{aligned}
		\ncbd{\beta \rd_\beta \xi_0} &\leq 5, \\
		\ncbd{\beta \xi_0} & \leq 2.83, \\
		\ncbd{\rd_\beta \xi_\infty} & \leq 3.5, \\
		\ncbd{\frac{\xi_\infty}{\beta}} & \leq 1,\\
		\ncbd{\beta^2 \rd_\beta \xi_0} & \leq 7.5, \\
		\ncbd{\beta^2 \rd^2_\beta \xi_0} & \leq 42, \\
		\maxi{\ncbd{\beta^\delta \xi_0}, \ncbd{\beta^{-\delta} \xi_\infty}} & \leq 1.42.
	\end{aligned}
\end{equation}

Substitute \eqref{exact bound of cbd norm} to \eqref{eq43} gives
\begin{align*}
	& \nrm{(2\mu -1)\left( Q-\n{P} \right)}_{[\Xnz, \Znz]} \\
	&\leq \frac{2\mu-1}{\brk{N}} \Bigg[ \nmtn \left( \frac{5+2.83\brk{N}}{(N-2)\mu+\frac{5}{6}}+2 \right)+\frac{1.05}{(N-2)\mu+\frac{5}{6}}(3.5+2\mu ) +\frac{\brk{N}\mu}{(N-2)\mu+\frac{5}{6}}+1.05 \Bigg] \cdot \\
	& \qquad \Bigg[\left( \frac{1.42\brk{N}}{(N-2)\mu +\frac{5}{6}}+ \frac{2.42}{(N-2)\mu +\frac{5}{6}}+31.95 \right) \left( \frac{10.33 \brk{N}}{(N-2)\mu +\frac{5}{6}}+ \frac{52}{(N-2)\mu +\frac{5}{6}}+6 \right)+\frac{2.83\brk{N}}{(N-2)\mu +\frac{5}{6}} \Bigg] \\
	& \eqqcolon K.
\end{align*}

To estimate K, we give a restriction
\begin{equation} \label{eq44}
	N>2000
\end{equation}
This restriction seems radical at the first glance. However, we set this lower bound, because estimation of $K$ with an assumption \eqref{eq44} gives the range of $N$ larger than 2000 for $K$ to be less than 1. For $N>2000$, we have (remember $\mu >\frac{2}{3}$)
\begin{equation} \label{eq45}
	\nmt<0.00076, \qquad \nmtn < \frac{1.00101}{\mu}.
\end{equation}
Substitute \eqref{eq45} to $K$, then we have
\begin{equation*}
	K \leq \frac{2\mu -1}{\brk{N}} \left(397+\frac{1090}{\mu} + \frac{1264}{\mu^2}+\frac{999}{\mu^3}+\frac{42}{\mu^4} \right).
\end{equation*}
Hence, if we choose $N$ so that
\begin{equation} \label{range of N}
	N > (2\mu -1)\left(397+\frac{1090}{\mu} + \frac{1264}{\mu^2}+\frac{999}{\mu^3}+\frac{42}{\mu^4} \right),
\end{equation}
then
\begin{equation*}
	\nrm{(2\mu -1)\left( Q-\n{P}\right)}_{[\Xnz, \Znz]} \leq K <1
\end{equation*}
for all $n \in N\bbZ \backslash \{0\}$ and accordingly
\begin{equation*}
	\frac{\rd \br{L}}{\rd \br{\psi}}\left( \br{\psi}_0, \Omega_0 \right) \in L\left(\Xz, \Zz \right)
\end{equation*}
becomes isomorphism. Combining all results in this section, we have:

\begin{theorem} \label{main thm in special coordinate}
	Let $\mu \in (\frac{2}{3}, \infty)$ be given. If
	\begin{equation*}
		N > (2\mu -1)\left(397+\frac{1090}{\mu} + \frac{1264}{\mu^2}+\frac{999}{\mu^3}+\frac{42}{\mu^4} \right),
	\end{equation*}
	then there exists $\varepsilon_N^*>0, \delta>0$ and $C^1$ map
	\begin{equation*}
		G: \ball{\Yz}{\Omega_0}{\varepsilon_N^*} \rightarrow \ball{\Xz}{\trisol}{\delta}
	\end{equation*}
	so that
	\begin{equation*}
		\br{L} \left( G(\Omega), \Omega \right)=0
	\end{equation*}
	for all $\Omega \in \ball{\Yz}{\Omega_0}{\varepsilon_N^*}$.
\end{theorem}

\begin{proof}
	Proposition \ref{continuity of L bar} says that
	\begin{equation*}
		\br{L}: \ball{\Xz}{\trisol}{\delta} \times \Yz \rightarrow \Zz
	\end{equation*}
	is $C^1$ and condition \eqref{range of N} implies
	\begin{equation*}
		\frac{\rd \br{L}}{\rd \br{\psi}}\left( \br{\psi}_0, \Omega_0 \right) \in L\left(\Xz, \Zz \right)
	\end{equation*}
	is an isomorphism. Then, implicit function theorem yields the result.
\end{proof}

\begin{remark}
	For the meaningful result, we hope $G(\Omega)$ is an real function when $\Omega$ is real (distribution). If we restrict our Banach space $\Xz, \Yz, \Zz$ to real function (or distribution), which is closed subspace of them as a fixed point of conjugate map, $\br{L}$ and $ \frac{\rd \br{L}}{\rd \br{\psi}}$ are still well-defined since differential operators send real distribution to the real one. Applying same argument to this real restriction of $L$ then gives the desired result.
\end{remark}

\bigskip

\section{Main theorem in original coordinate}
\label{sec: Main theorem in original coordinate}

\subsection{Regularity of stream function and change of coordinate}
\label{subsec: Regularity of stream function and change of coordinate}

So far, we have constructed solution $\br{\psi}^{(\Omega)} = G(\Omega)$ in Theorem \ref{main thm in special coordinate} for given $\Omega \in \ball{\Yz}{\Omega_0}{\varepsilon^*}$ which satisfies $\br{L} \left( \br{\psi}^{(\Omega)}, \Omega \right)=0$. Although $\br{\psi}^{(\Omega)}$ is determined by $\Omega$, we simply write $\br{\psi}$ instead of $\br{\psi}^{(\Omega)}$ for simplicity. The remained thing is to recover physical variable $w(x,t), u(x,t)$ and $\psi(x,t)$ from $\br{\psi}(\beta, \phi), \Omega(\phi)$ and to verify how the conditions and results on Theorem \ref{main theorem} is derived.

First, we eliminate bar notation. Remember, $\psi$ and $\br{\psi}$ are related through
\begin{align*}
	\psi_\beta &= \beta^{-2\mu} \br{\rd}_\beta \br{\psi},  \\
	\psi_\varphi &= \beta^{-2\mu} \br{\rd}_\varphi \br{\psi},  \\
	\psi_\phi &= \beta^{1-2\mu} \rd_\phi \br{\psi},  \\
	\psi_{\beta \phi} &= \beta^{-2\mu} \rd_\phi \br{\rd}_\beta \br{\psi},  \\
	\psi_{\beta \varphi} &= \beta^{-2\mu -1} \left( \bar{\rd}_\varphi +1 \right) \br{\rd}_\beta \br{\psi}.
	\end{align*}
Since  $\br{\psi}$ satisfies $\nrm{\br{\psi}-\trisol}_{\Xz}<\delta<$ min$(\delta_1, \delta_2, \delta_3, \delta_4, \delta_5)$ where $\delta_1, \delta_2, \delta_3, \delta_4, \delta_5$ is as in \eqref{eq31}, \eqref{eq32}, \eqref{eq33}, \eqref{phi bound} and \eqref{id bound}, we have boundedness

\begin{equation} \label{boundedness of bar derivative}
	\begin{aligned}
		-\frac{3}{2} \leq \dbbeta \br{\psi} \leq -\frac{1}{2}, \quad\frac{1}{2} \leq \dbvarphi & \br{\psi} \leq \frac{3}{2}, \quad-3\mu \leq \dbvarphibeta \br{\psi} \leq -\mu, \\
		-1 \leq \rd_\phi \br{\psi} \leq 1, \quad -1 \leq \rd_\phi \dbbeta & \br{\psi} \leq 1 , \quad \frac{1}{4\mu-2} \leq \br{\psi} \leq \frac{3}{4\mu -2}
	\end{aligned}
\end{equation}
for all $(\beta, \phi) \in \RT$. Hence,

\begin{equation} \label{never touch zero}
	\psi_\beta <0, \qquad \psi_\varphi>0, \qquad \psi_{\beta \varphi}<0
\end{equation}
so that \eqref{L before} is well-defined. Also, various assumption we imposed on $\psi(\beta,\phi), a(\beta,\phi)$ in Proposition \ref{new coordinate prop} and Section \ref{subsec: Equations in new coordinate} are now justified. Furthermore, by Proposition \ref{continuity of atomic diff operator} and Proposition \ref{continuity of diff bar operator},
\begin{equation*}
	\br{\psi}, \qquad \dbbeta \br{\psi}, \qquad \dbvarphi \br{\psi}, \qquad \dbvarphibeta \br{\psi}, \qquad \dbphibeta \br{\psi}
\end{equation*}
are all belong to $\Wp=\calA^{0.5} \left( \cbdn \oplus \cxizn \oplus \cxiin   \right) $. Through an embedding
\begin{equation*}
	\calA^{0.5} \left( \cbdn \oplus \cxizn \oplus \cxiin \right)  \hookrightarrow \calA^{0.5} \left( C_b \right) \hookrightarrow \calA^0 \left( C_b \right) \underset{\text{M-test}}{\hookrightarrow} C_b \left( \RT \right),
\end{equation*}
it implies
\begin{equation} \label{psi regularity}
	\psi, \qquad\psi_\beta, \qquad \psi_\varphi, \qquad \psi_\phi, \qquad \psi_{\beta \phi}, \qquad \psi_{\beta \varphi}
\end{equation}
all belong to $C \left( \bbR_+ \times \bbT \right)$, where $C \left( \bbR_+ \times \bbT \right)$ is a set of continuous (not need to be bounded) function on $\bbR_+ \times \bbT$. Now, we can justify special change of coordinate we did in Section \ref{subsec: Special change of coordinate: From spiral to line}.

\begin{proposition} \label{C^1 change of coordinate}
	A map $T \coloneqq B \circ A: \bbR_+ \times \bbT \overset{A}{\longrightarrow} \bbR \times \bbT \overset{B}{\longrightarrow} \bbR^2 \backslash \set{0}$ defined by composition of two maps 
	\begin{equation*}
		(\beta, \phi) \quad \xrightarrow[\begin{cases}
			a=\frac{1}{2} \log \left( -\frac{\psi_\beta}{\mu} \right) \\
			\theta = \beta + \phi
		\end{cases}]{A} \quad  (a, \theta) \quad \xrightarrow[\begin{cases}
			z_1=e^a \cos \theta \\
			z_2=e^a \sin \theta
		\end{cases}]{B} \quad  (z_1, z_2)
	\end{equation*}
\end{proposition}
is a $C^1$-diffeomorphism.

\begin{proof}
	It is well known that polar coordinate change map $B$ is a $C^\infty$-diffeomorphism, so it suffices to prove the result for $A$.
	
	First,
	\begin{equation*}
		a(\beta, \theta) =\frac{1}{2} \log \left( -\frac{\psi_\beta}{\mu} \right)
	\end{equation*}
	is well defined since $\psi_\beta$ is strictly negative by \eqref{never touch zero}.
	
	Second, we prove surjectivity. Since
	\begin{equation*}
		\psi_\beta = \beta^{-2\mu} \dbbeta \br{\psi}
	\end{equation*}
	and
	\begin{equation*}
		 \, \dbbeta \br{\psi} \in C \left( \bbR_+ \times \bbT \right) \qquad \text{with} \qquad -\frac{3}{2} \leq \dbbeta \br{\psi} \leq -\frac{1}{2},
	\end{equation*}
	intermediate value theorem guarantees surjectivity.
	
	Last, we prove injectivity. Suppose
	\begin{equation*}
		\left( a(\beta_1, \phi_1),  \beta_1+\phi_1 \right)=\left( a(\beta_2, \phi_2), \beta_2+\phi_2 \right) = ( a_0, \theta_0).
	\end{equation*}
	The set of $(\beta, \phi) \in \bbR_+ \times \bbT$ which satisfies $\beta + \phi = \theta_0$ is represented as Figure \ref{Fig 3}.
	
	\begin{figure}[h]
	\includegraphics[trim = 40mm 190mm 40mm 33mm, clip, width=12cm]{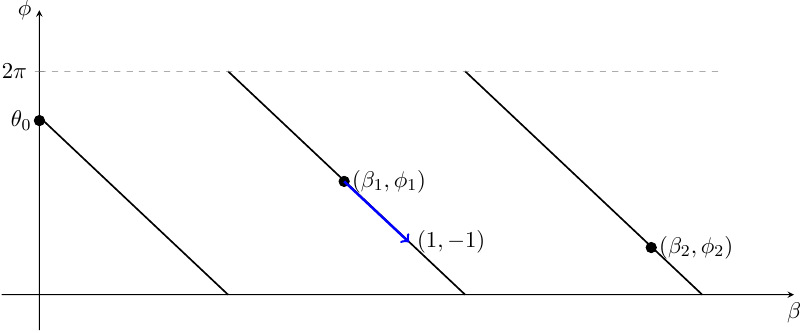}
	\caption{A level set of $\beta+\phi$} \label{Fig 3}
	\end{figure}
	
	%Todo figure 4. A level set of beta+\phi
	
	Then, it can be easily seen that $(\beta_1, \phi_1 ) \neq (\beta_2, \phi_2)$ as in Figure \ref{Fig 3} leads to contradiction since directional derivative of $a(\beta, \phi)$ to the direction of $(1,-1)$ is
	\begin{equation*}
		D_{(1, -1)} a(\beta, \phi) = \rd_\beta a(\beta, \phi) - \rd_\phi a(\beta, \phi) = -a_\varphi \underset{\eqref{new coordinate}}{=} -\frac{\psi_{\beta \varphi}}{\psi_\beta}
	\end{equation*}
	which is strictly negative for all $(\beta, \phi) \in \bbR_+ \times \bbT$. Therefore, $A$ is $C^1$ bijective map. Furthermore, since
	\begin{equation*}
		J_A = \det \begin{pmatrix}
			a_\beta & a_\phi \\
			1 & 1
		\end{pmatrix}= -a_\varphi <0,
	\end{equation*}
	for all $(\beta, \phi) \in \bbR_+ \times \bbT$, its inverse is also $C^1$ by inverse function theorem.
\end{proof}

Before doing change of coordinate, we analyze the regularity of vorticity $w(\beta, \phi)$. As we have relation
\begin{equation*}
	w(\beta, \phi) = \left( \psi_\varphi \right)^{-\frac{1}{2\mu}} \Omega(\phi),
\end{equation*}
we have to investigate regularity of $\Omega(\phi)$. However, the problem is that $\Omega(\phi)$ belongs to $\calA^{-0.5} \left( \cn \right)$, from which we cannot guarantee that $\Omega$ is continuous, or even a function! Therefore, we impose further restriction on $\Omega(\phi)$. We assume 
\begin{equation*}
	\Omega \in \calA^{-0.5} \left( \cn \right) \cap L^p (\bbT)
\end{equation*}
so that 
\begin{equation*}
	w(\beta, \phi) \in L^p_{\textnormal{loc}}(\bbR_+ \times \bbT)
\end{equation*}
for some fixed $p \in [1,2\mu)$\footnote{The reason we require $p \in [1,2\mu)$ is to guarantee the local integrability of $w(x,t)$. This will be seen at Proposition \ref{w is L^p}}. Then, we have

\begin{equation} \label{regularity of dependent variables}
		\begin{aligned}
			&\psi(\beta, \phi) \in C^1(\bbR_+ \times \bbT), \quad   w(\beta, \phi) \in L^p_{\textnormal{loc}}(\bbR_+ \times \bbT) \\
			\underset{\text{Prop } \ref{C^1 change of coordinate}}&{\Longrightarrow} \quad  \tld{\psi}(z) \in C^1 \left( \bbR^2 \backslash \set{0} \right), \,  \tld{u}(z) \coloneqq \nabla^\perp_{z} \tld{\psi}(z) \in C \left( \bbR^2 \backslash \set{0} \right),\tld{w}(z) \in L^p_{\textnormal{loc}} \left( \bbR^2 \backslash \set{0} \right)   \\
			\underset{\eqref{tilde func}}&{\Longrightarrow} \psi(\cdot,t) \in C^1 \left( \bbR^2 \backslash \set{0} \times \bbR_+ \right), \, u(\cdot,t) \in C \left( \bbR^2 \backslash \set{0} \times \bbR_+ \right), \, w(\cdot,t) \in L^p_{\textnormal{loc}} \left( \bbR^2 \backslash \set{0} \times \bbR_+ \right).
		\end{aligned}
	\end{equation}
	
	Also, reading the Section \ref{subsec: Equations in new coordinate} and \ref{subsec: Rescaling to eliminate decay} backwards, we can recover vorticity stream relation
\begin{equation} \label{vorticity stream relation}
\begin{aligned}
	& \quad \br{L}(\br{\psi}(\beta, \phi), \Omega(\phi))=0 \\
	& \Rightarrow L\left( \psi(\beta, \phi), \Omega(\phi) \right) =0 \\
	& \Rightarrow \begin{pmatrix}
		\rd_\beta \\
		\rd_\phi \\
	\end{pmatrix} \cdot \Bigg( \underbrace{\abs{J_T} J_T^{-1}}_{\text{adj}(J_T)} \underbrace{\left(J_T^{-1} \right)^t \begin{pmatrix}
		\rd_\beta \\
		\rd_\phi \\
	\end{pmatrix}}_{\nabla_z} \psi \Bigg) = \abs{J_T}w \\
	\underset{\eqref{adjoint formula}}&{\Rightarrow} \nabla_z \cdot \left(\nabla_z \tld{\psi} \right) = \tld{w}
\end{aligned}
\end{equation}
in the sense of distribution.

\bigskip

\subsection{Weak solution on $\rmz \times [a,\infty)$ for $a>0$}
\label{subsec: Weak solution; weak version}
	
	To show $w(x,t), u(x,t)$ is a weak solution of Euler equation, we mollify $w(\beta, \phi)$, due to the lacks of regularity. Let $\eta$ be a standard mollifier in one dimension and we write $\Omega^\varepsilon(\phi)$ to denote $\left( \Omega(\cdot) * \frac{1}{\varepsilon}\eta \left( \frac{\cdot}{\varepsilon} \right) \right) (\phi)$. Then, for $w(\beta, \phi)= \left( \psi_\varphi (\beta, \phi) \right)^{-\frac{1}{2 \mu}} \Omega(\phi)$, we define $w^{\epdot} (\beta, \phi)$ by
	\begin{equation} \label{one directed mollification}
		w^{\epdot}(\beta, \phi) = \left( \psi_\varphi (\beta, \phi) \right)^{-\frac{1}{2 \mu}} \Omega^\varepsilon (\phi)
	\end{equation}
(note that we only mollify $\Omega$ since $\left( \psi_\varphi \right)^{-\frac{1}{2 \mu}}$ is already $C^1$). Now we show that $w^{\epdot} \xrightarrow[\varepsilon \searrow 0]{} w$ in $L^p_{\textnormal{loc}}(\bbR_+ \times \bbT)$ for $p < \infty$.

\begin{proposition} \label{L^p convergence of w}
	For any $1 \leq p < \infty$ and $0 < a < b  < \infty$,
	\begin{equation*}
		w^{\epdot} \longrightarrow w
	\end{equation*}
	in $L^p \left( [a, b] \times \bbT \right)$. 
\end{proposition}
\begin{proof}
	\begin{align*}
		\nrm{w^{\epdot}-w}_{L^p \left( [a, b] \times \bbT \right)} &= \nrm{\left( \psi_\varphi \right)^{-\frac{1}{2\mu}}\left( \Omega^\varepsilon - \Omega \right)}_{L^p \left( [a, b] \times \bbT \right)} \\
		&\leq \left(\sup_{(\beta, \phi) \in [a, b] \times \bbT} \left( \psi_\varphi \right)^{-\frac{1}{2\mu}} \right)  \big\Vert \underbrace{\Omega^\varepsilon - \Omega}_{\substack{\text{does not} \\ \text{depend on }\beta}} \big\Vert_{L^p \left( [a, b] \times \bbT \right)} \\
		&= \underbrace{(b-a)^{\frac{1}{p}} \left(\sup_{(\beta, \phi) \in [a, b] \times \bbT} \left( \psi_\varphi \right)^{-\frac{1}{2\mu}} \right)}_{< \infty}  \nrm{\Omega^\varepsilon - \Omega}_{L^p(\bbT)}  \\
		& \xrightarrow[\varepsilon \searrow 0]{} 0
	\end{align*}
	by a property of standard mollifier.
\end{proof}

\begin{remark} \label{L^p conv remark}
	If $p=\infty$, by the inclusion $L^q (\bbT) \subset L^r (\bbT)$ for all $1 \leq r \leq q \leq \infty$, the convergence holds for all $1 \leq p < \infty$. In particular, $w^{\epdot} \rightarrow w$ in $L^1 ([a,b] \times \bbT )$ regardless of $p \geq 1$.
\end{remark}

Since $w^{\epdot}(\beta, \phi)$ is $C^1$ and $\left( \psi_\varphi \right)^{\frac{1}{2\mu}} w^{\epdot}$ depends only on $\phi$, reading \eqref{EE to ODE} backwards gives
\begin{equation} \label{eq47}
	-\tld{w}^{\epdot}(z)+ \left( \nabla^\perp_z \tld{\psi}(z) -\mu z \right) \cdot \nabla_z \tld{w}^{\epdot}(z)=0.
\end{equation}
Using \eqref{eq47}, we can verify the first equation of \eqref{EEz}.

\begin{proposition} \label{weak sol for EEz}
	\begin{equation*}
		(-1+2\mu) \tld{w} + \nabla_z \cdot \left( \tld{w} \left( \nabla^\perp_z \tld{\psi} - \mu z \right) \right) =0
	\end{equation*}
	on $\bbR^2 \backslash \set{0}$ in weak sense, i.e., for any test function $\tld{f} \in C^\infty_c \left( \bbR^2 \backslash \set{0} \right)$,
	\begin{equation*}
		\int_{\bbR^2 \backslash \set{0}} (-1+2\mu) \tld{w} \tld{f} - \tld{w}\left( \left( \nabla^\perp_z \tld{\psi} - \mu z \right) \cdot \nabla_z \tld{f} \right)=0.
	\end{equation*}
\end{proposition}
\begin{proof}
	For any test function  $\tld{f} \in C^\infty_c \left( \bbR^2 \backslash \set{0} \right)$, we have
	
	\begin{align*}
		& \int_{\rmz} \tld{w}^{\epdot} \tld{f} + \tld{w}^{\epdot} \left( \left( \nabla^\perp_z \tld{\psi} - \mu z \right) \cdot \nabla_z \tld{f} \right) \\
		\underset{\eqref{eq47}}&{=} \int_{\rmz} \left( \left( \nabla^\perp_z \tld{\psi} - \mu z \right) \cdot \nabla_z \tld{w}^{\epdot} \right) \tld{f} + \tld{w}^{\epdot} \left( \nabla^\perp_z \tld{\psi} - \mu z \right) \cdot \nabla_z \tld{f} \\
		\underset{\text{Leibniz}}&{=} \int_{\rmz} \left( \nabla^\perp_z \tld{\psi} - \mu z \right) \cdot \nabla_z \left( \tld{w}^{\epdot} \tld{f} \right) \\
		&= \int_{\rmz} \underbrace{\nabla^\perp_z \tld{\psi} \cdot \nabla_z \left( \tld{w}^{\epdot} \tld{f} \right)}_{\substack{=0  \,( \because \, \nabla^\perp_z \psi \text{ is weakly}\\ \text{divergence free)}} } - \mu z \cdot \nabla_z \left( \tld{w}^{\epdot} \tld{f} \right) \\
		\underset{IBP}&{=} \int_{\rmz} \left( \mu \nabla_z \cdot z \right) \tld{w}^{\epdot} \tld{f} \\
		&= \int_{\rmz} 2\mu \tld{w}^{\epdot} \tld{f}.
	\end{align*}
	Therefore, we have
	\begin{equation} \label{eq48}
		\int_{\bbR^2 \backslash \set{0}} (-1+2\mu) \tld{w}^{\epdot} \tld{f} - \tld{w}^{\epdot} \left( \left( \nabla^\perp_z \tld{\psi} - \mu z \right) \cdot \nabla_z \tld{f} \right)=0
	\end{equation}
	for all $\varepsilon>0$. Let $\text{supp}\tld{f}=K$ and take $0<a < b < \infty$ such that $T^{-1}(K) \subset [a, b] \times \bbT$. Then,
	\begin{align*}
		& \abs{\int_{\rmz} \left(\tld{w}^{\epdot} \tld{f} - \tld{w}\tld{f}\right) \, dz}  \\
		&= \abs{\int_{T \left( [a,b] \times \bbT \right)} \left(\tld{w}^{\epdot} \tld{f} - \tld{w}\tld{f}\right) \, dz} \\
		&= \abs{\int_{[a,b] \times \bbT} \left( w^{\epdot}f-wf \right) \abs{\det J_T}\, d\beta d\phi} \\
		&\leq \nrm{w^{\epdot}-w}_{L^1 \left( [a,b]\times \bbT \right)} \underbrace{\nrm{f \det J_T}_{L^\infty \left( [a,b]\times \bbT \right)}}_{< \infty \, (\because \, f, J_T \text{ is continuous})} \\
		&\xrightarrow[\varepsilon \searrow 0]{\text{Prop } \ref{L^p convergence of w}} 0.
	\end{align*}
	In similar way, we have
	\begin{equation*}
		\abs{\int_{\rmz} \left( \tld{w}^{\epdot} - \tld{w} \right) \left( \left( \nabla^\perp_z \tld{\psi} -\mu z \right) \cdot
		 \nabla_z \tld{f} \right)} \xrightarrow[\varepsilon \searrow 0]{}0.
	\end{equation*}
	Therefore, sending $\varepsilon$ to 0 in \eqref{eq48} gives desired result.
\end{proof}

Having proved $\tld{w}(z)$ is a weak solution for \eqref{EEz}, we now prove vorticity in space-time coordinate, i.e., $w(x,t)=t^{-1} \tld{w}\left( xt^{-\mu} \right)$, is a weak solution of \eqref{EE}. Let $w^{\epdot}(x,t) \coloneqq t^{-1} \tld{w}^{\epdot} \left( xt^{-\mu} \right)$ and remember $u(x,t)=t^{\mu -1} \nabla^\perp_z \tld{\psi} \left( xt^{-\mu} \right)$.

\begin{lemma} \label{weak solution of EE mollified}
	For all $\varepsilon>0$,
	\begin{equation*}
		w^{\epdot}_t + \nabla_x \cdot \left(w^{\epdot} u \right) =0
	\end{equation*}
	on $\rmz \times [a, \infty)$ in distribution sense for any $a>0$.
\end{lemma}

\begin{proof}
	Remember that
	\begin{equation*}
		-\tld{w}^{\epdot}(z)+ \left( \nabla^\perp_z \tld{\psi}(z) -\mu z \right) \cdot \nabla_z \tld{w}^{\epdot}(z)=0
	\end{equation*}
	in classical sense and $\tld{w}^{\epdot}(z)$ is $C^1$. Then, we can trace back from \eqref{EEz} to \eqref{EE}, which leads to 
	\begin{equation} \label{eq49}
		w^{\epdot}_t + u \cdot \nabla_x w^{\epdot}=0
	\end{equation}
	in classical sense on $\rmz \times \bbR_+$. Then, for a test function $f \in C^\infty_c \left( \rmz \right)$,
	\begin{align*}
		& \int_{\bbR^2 \backslash \set{0}} w(\cdot,t)f(\cdot,t)\, dx+ \int_a^\infty \int_{\bbR^2 \backslash \set{0}}w^{\epdot}f_t+w^{\epdot}(u \cdot \nabla_x f) \, dxdt \\
		\underset{\substack{\text{FTC, Fubini,} \\ \text{Leibniz}}}&{=} -\int_{\rmz \times [a,\infty)} (w^{\epdot}f)_t \, dxdt + \int_{\rmz \times [a,\infty)} w^{\epdot}f_t-f(u \cdot \nabla_x w^{\epdot})+u \cdot \nabla_x(w^{\epdot}f)\, dxdt  \\
		& = \int_{\rmz \times [a, \infty)} -\underbrace{\left( w^{\epdot}_t + u \cdot \nabla_x w^{\epdot} \right)}_{=0 \text{ by} \eqref{eq49}} f + \underbrace{u \cdot \nabla_x \left( w^{\epdot} f \right)}_{\substack{=0 \,(\because \, u \text{ is weakly }\\ \text{divergence free})}} \, dxdt. \\ 
		&=0
	\end{align*}
\end{proof}

\begin{proposition} \label{weak solution of EE in vorticity form}
	The $(w(x,t), u(x,t), \psi(x,t))$ given in \eqref{regularity of dependent variables} satisfies
		\begin{equation*} 
	\left\{\begin{aligned}
		& w_t+ \nabla_x \cdot (wu) =0 & on \quad \rmz \times [a, \infty), \\
		& u=\nabla_x^\perp \psi ,\quad w= \Delta_x \psi &  on\quad \rmz
	\end{aligned}\right.
		\end{equation*}
		in distribution sense for any $a>0$.
\end{proposition}
\begin{proof}
	First, $u=\nabla^\perp_x \psi$ is clear by the definition of $u(x,t)$ in  \eqref{tilde func} and \eqref{regularity of dependent variables}. Also, $\Delta_x \psi = w$ could be obtained directly from \eqref{vorticity stream relation}.
	
	To prove the first equation, take $f \in C^\infty_c \left( \rmz \times [a, \infty) \right)$. Then, we can take annulus $K=\br{B(r, R)} \coloneqq \set{z \in \bbR^2 \, \vert \, r \leq \abs{z} \leq R}$ and $0< a<b<\infty$ such that
	\begin{equation*}
		\text{supp}f \in \br{B(r,R)} \times [a,b].
 	\end{equation*}
 	Let $Q \coloneqq \br{B\left( \frac{r}{b^\mu}, \frac{R}{a^\mu} \right)}$. By Proposition \ref{L^p convergence of w} and Remark \ref{L^p conv remark}, we have
 	\begin{equation*}
 		\tld{w}^{\epdot}(z) \rightarrow \tld{w}(z)
 	\end{equation*}
 	in $L^1(Q)$. Then,
 	\begin{align*}
 		& \int_{K \times [a,b]} \abs{w^{\epdot}(x,t) - w(x,t)} \, dxdt \\
 		\underset{\text{Fubini}}&{=} \int_a^b \int_K t^{-1} \abs{ \tld{w}^{\epdot}(xt^{-\mu}) -\tld{w}(xt^{-\mu}) } \, dxdt \\
 		\underset{xt^{-\mu} \rightarrow z}&{=} \int_a^b \int_{t^{-\mu}K} t^{-1} \abs{\tld{w}^{\epdot}(z)-\tld{w}(z)} t^{2\mu}\, dzdt \\
 		\underset{\substack{t^{-\mu}K \subset Q \\ \text{for } t\in [a,b]}}&{\leq} \int_a^b t^{2\mu -1} \int_Q \abs{\tld{w}^{\epdot}(z)-\tld{w}(z)} \, dzdt \\
 		&\leq \frac{1}{2\mu} \left( b^{2\mu} - a^{2\mu} \right) \cdot \underbrace{\int_Q \abs{\tld{w}^{\epdot}(z)-\tld{w}(z)} \, dz}_{\xrightarrow[\varepsilon \searrow 0]{}0 \text{ by Prop } \ref{L^p convergence of w}},
 	\end{align*}
 	which means
 	\begin{equation} \label{eq105}
 		w^{\epdot}(x,t) \xrightarrow[\varepsilon \searrow 0]{} w(x,t)
 	\end{equation}
 	in $L^1 \left( K \times [a, b] \right)$. Then, we have
 	\begin{align*}
 		0 \underset{\text{Lem } \ref{weak solution of EE mollified} }&{=} -\int_K w^{\epdot}(\cdot,a)f(\cdot,a)- \int_{K \times [a, b]} w^{\epdot} (f_t+u \cdot \nabla f) \\
 		&\xrightarrow[\varepsilon \searrow 0]{} - \int_K w(\cdot,a)f(\cdot,a) - \int_{K \times [a, b]} w (f_t+u \cdot \nabla f),
 	\end{align*}
 	because
 	\begin{align*}
 		& \abs{\int_K \left( w^{\epdot}(x,a)-w(x,a) \right) f(x,a)\, dx}
 		= \abs{\int_K \left( \tld{w}^{\epdot}(xa^{-\mu})-\tld{w}(xa^{-\mu}) \right) f(x,a)\, dx} \\
 		\underset{xa^{-\mu}\rightarrow z}&{=} \abs{\int_{a^{-\mu}K} \left( \tld{w}^{\epdot}(z)-\tld{w}(z) \right) f(a^\mu z,a)a^{2\mu}\, dz} \underset{a^{-\mu}K \subset Q}{\leq } \abs{\int_{Q} \left( \tld{w}^{\epdot}(z)-\tld{w}(z) \right) f(a^\mu z,a)a^{2\mu}\, dz} \\
 		&\leq \underbrace{\nrm{\tld{w}^{\epdot}(z)-\tld{w}(z)}_{L^1(Q)}}_{\xrightarrow[\varepsilon \searrow 0]{}0 \,(\because \, \text{Prop }\ref{L^p convergence of w})} \underbrace{\nrm{a^{2\mu}f(a^\mu z, a)}_{L^\infty(Q)}}_{< \infty \, (\because \, f \text{ is continuous)}}
 	\end{align*}
 	and
 	\begin{equation*}
 		\abs{\int_{K \times [a, b]} (w^{\epdot}-w) (f_t+u \cdot \nabla f)} \leq  \underbrace{\nrm{w^{\epdot}-w}_{L^1(K \times [a,b])}}_{\xrightarrow[\varepsilon \searrow 0]{}0} \underbrace{\nrm{f_t+u \cdot \nabla f}_{L^\infty(K \times [a,b] )}}_{\substack{< \infty \, ( \because \, u \text{ and }f \\ \text{is continuous}}}
 	\end{equation*}
\end{proof}

Having treated Euler equation in vorticity form, we now deal with velocity form. First, we show the following lemma.

\begin{lemma} \label{weak solution for u intermediate 1}
	\begin{equation*}
	\nabla_x \cdot(wu) = \nabla^\perp_x \cdot \left( \nabla_x \cdot ( u \otimes u ) \right)
	\end{equation*}
	on $\rmz$ in the sense of distribution.
\end{lemma}
\begin{proof}
	Take a test function $h \in C^\infty_c(\rmz)$. What we have to show is
	\begin{equation*}
		-\int_{\rmz} w(u \cdot \nabla h) = \int_{\rmz} (u \otimes u) \, : \, \nabla \nabla^\perp h.
	\end{equation*}
	Let $\text{supp}(h) =K$, $U = \bigcup_{x \in K} B \left( x, \frac{\dist (K, \set{0}}{4} \right)$ and $\eta(x)$ be a standard mollifier in $\bbR^2$. We define
	\begin{equation*}
		w^\varepsilon (x,t) \coloneqq \int_{\bbR^2} w(x-y,t) \cdot \frac{1}{\varepsilon^2} \eta \left( \frac{y}{\varepsilon} \right) \, dy \footnote{Note that $w^\varepsilon(x,t)$ and $w^{\epdot}(x,t)$ is different. While $w^{\epdot}(x,t)$ is defined through \eqref{one directed mollification}, $w^\varepsilon(x,t)$ is just a standard mollification on $\bbR^2$. }
	\end{equation*}
	for $x \in U, \varepsilon< \frac{\dist (K, \set{0})}{4}$, and $u^\varepsilon (x,t), \psi^\varepsilon (x,t)$ accordingly. From the compatibility of mollifier with differential operator and the fact that
	\begin{equation*}
		\Delta_x \psi = w, \qquad \nabla_x \cdot u = 0, \qquad \nabla^\perp_x \psi = u, \qquad \nabla^\perp_x \cdot u = w
	\end{equation*}
	in distribution sense, we have
	\begin{equation*}
		\Delta_x \psi^\varepsilon = w^\varepsilon, \qquad \nabla_x \cdot u^\varepsilon = 0, \qquad \nabla^\perp_x \psi^\varepsilon = u^\varepsilon, \qquad \nabla^\perp_x \cdot u^\varepsilon = w^\varepsilon
	\end{equation*}
	in classical sense on $U$. Then, it is clear that
	\begin{equation*}
	\nabla_x \cdot(w^\varepsilon u^\varepsilon) = \nabla^\perp_x \cdot \left( \nabla_x \cdot ( u^\varepsilon \otimes u^\varepsilon ) \right),
	\end{equation*}
	and this implies
	\begin{equation} \label{eq65}
		-\int_{\rmz} w^\varepsilon (u^\varepsilon \cdot \nabla h) = - \int_{\rmz} (u^\varepsilon \otimes u^\varepsilon) \, : \, \nabla \nabla^\perp h.
	\end{equation}
	Since $w^\varepsilon \xrightarrow[\varepsilon \searrow 0]{} w$ in $L^1(K)$  and $u^\varepsilon \xrightarrow[\varepsilon \searrow 0]{}u$ uniformly on K, sending $\varepsilon$ to 0 in \eqref{eq65} yields the result.
\end{proof}

\begin{proposition} \label{weak solution for u intermediate 2}
	\begin{equation*}
		\nabla^\perp_x \cdot \left( u_t + \nabla_x \cdot ( u \otimes u ) \right)=0
	\end{equation*}
	on $\rmz \times [a,\infty)$ in the sense of distribution for any $a>0$.
\end{proposition}

\begin{proof}
	Proposition \ref{weak solution of EE in vorticity form} and Lemma \ref{weak solution for u intermediate 1}.
\end{proof}

\bigskip

\subsection{Convergence to initial data and time independent boundedness}
\label{subsec: Convergence to initial data and time independent boundedness}

Note that Proposition \ref{weak solution of EE in vorticity form} tells us that constructed $w(x,t)$ is a weak solution of Euler equation on time $(0,\infty)$, not on $[0,\infty)$. To show $w$ and $u$ is a weak solution on time $[0,\infty)$, we have to investigate the initial data corresponding to $(w(x,t), u(x,t), \psi(x,t))$.

\begin{proposition} \label{initial data for stream function prop}
	The stream function $\psi(x,t)$ recovered from $\br{\psi}(\beta, \phi)\underset{\text{Thm } \ref{main thm in special coordinate}}{=}G(\Omega)$ through \eqref{regularity of dependent variables} satisfies
	\begin{equation} \label{initial data for stream function}
		\lim_{t \searrow 0} \psi(x,t) = \abs{x}^{2-\frac{1}{\mu}} \left( -\frac{\dbbeta \br{\psi}(0,\theta)}{\mu} \right)^{\frac{1}{2\mu}-1} \br{\psi}( 0, \theta)
	\end{equation}
	locally uniformly on $\rmz$. 
\end{proposition}
\begin{proof}
	We have
	\begin{equation} \label{eq50}
		\psi(x,t) \underset{\eqref{tilde func}}{=} t^{2\mu -1} \tld{\psi}(xt^{-\mu}) \underset{\substack{(\beta, \phi) = T^{-1}(xt^{-\mu}) \\ \psi(\beta, \phi) =\beta^{1-2\mu}\br{\psi}(\beta, \phi)}}{=} \left( \frac{\beta}{t} \right)^{1-2\mu} \br{\psi} (\beta, \phi).
	\end{equation}
	Also,
	\begin{equation} \label{eq51}
		\abs{x} = t^\mu \abs{z} \underset{\eqref{log-polar coordinate}}{=} t^\mu e^{a(\beta, \phi)} \underset{\eqref{new coordinate}}{=} t^\mu \left( -\frac{\psi_\beta}{\mu} \right)^{\frac{1}{2}} \underset{\eqref{beta and bar beta}}{=} \left( \frac{t}{\beta} \right)^\mu \left( -\frac{\dbbeta \br{\psi}}{\mu} \right)^{\frac{1}{2}}.
	\end{equation}
	From \eqref{eq51}, we have
	\begin{equation} \label{eq52}
		\frac{\beta}{t} = \abs{x}^{-\frac{1}{\mu}} \left( -\frac{\dbbeta \br{\psi} (\beta, \phi)}{\mu} \right)^{\frac{1}{2\mu}},
	\end{equation}
	and substituting \eqref{eq52} to \eqref{eq50} gives
	\begin{equation} \label{eq53}
		\psi (x,t) = \abs{x}^{2-\frac{1}{\mu}} \left( -\frac{\dbbeta \br{\psi}(\beta, \phi)}{\mu} \right)^{\frac{1}{2\mu}-1} \br{\psi}( \beta, \phi),
	\end{equation}
	where $(\beta, \phi)=T^{-1}(xt^{-\mu})$. Therefore, we have to investigate $\lim_{t \searrow 0} T^{-1}(xt^{-\mu})$. By using boundedness of $\dbbeta \br{\psi}$ as in \eqref{boundedness of bar derivative}, we get from \eqref{eq52} that
	\begin{equation} \label{eq54}
		\beta \leq \abs{x}^{-\frac{1}{\mu}} \left( \frac{3}{2\mu}\right)^{\frac{1}{2\mu}}t.
	\end{equation}
	Therefore, for all compact region $R$ on $\bbR^2 \backslash \set{0}$, $\beta \rightarrow 0$ as $t \rightarrow 0$ for uniform rate $0 \leq \beta < C_R t$.\footnote{Explicitly, $C_R=\left( \frac{3}{2\mu} \right)^{\frac{1}{2\mu}} \cdot \dist (0,R)^{-\frac{1}{\mu}}$.} From this relation of $\beta$ and $t$, we have for $x \in R$,
	\begin{equation*}
		\abs{T^{-1}(xt^{-\mu})-(0,\theta)} = \abs{(\beta, \phi)-(0,\theta)} \underset{\theta=\beta+\phi}{=}\sqrt{2} \beta < \sqrt{2}C_R t,
	\end{equation*}
	i.e., $(\beta,\phi)$ in \eqref{eq53} converges uniformly to $(0,\theta)$ as times goes to zero for any compact set $R \in \bbR^2 \backslash \set{0}$. Since
	\begin{equation*}
		\left( -\frac{\dbbeta \br{\psi}(\beta, \phi)}{\mu} \right) ^{\frac{1}{2\mu}-1} \br{\psi}(\beta, \phi)
	\end{equation*}
	is uniformly continuous on compact subset of $\RT$, $\psi(x,t)$ in \eqref{eq53} converges uniformly to
	\begin{equation*}
		 \abs{x}^{2-\frac{1}{\mu}} \left( -\frac{\dbbeta \br{\psi}(0,\theta)}{\mu} \right)^{\frac{1}{2\mu}-1} \br{\psi}( 0, \theta).
	\end{equation*}
\end{proof}

Next, we investigate the initial velocity.

\begin{proposition} \label{initial data for velocity prop}
	The velocity field $u(x,t)$ satisfies
	\begin{equation} \label{initial data for velocity}
		\begin{aligned}
		\lim_{t \searrow 0} u(x,t) &= \abs{x}^{1-\frac{1}{\mu}} \left(-\frac{\dbbeta \br{\psi}(0,\theta)}{\mu} \right)^{\frac{1}{2\mu}-1} \frac{2 \dbbeta \br{\psi}(0, \theta)}{\dbvarphibeta \br{\psi}(0,\theta)} \\
		&  \cdot \left( \frac{\dbphibeta \br{\psi}(0,\theta) \cdot \dbvarphi \br{\psi}(0,\theta)-\dbvarphibeta \br{\psi} (0,\theta) \cdot \rd_\phi \br{\psi}(0,\theta)}{2\dbbeta \br{\psi}(0,\theta)} \begin{pmatrix}
			\cos \theta \\
			\sin \theta
		\end{pmatrix} + \dbvarphi \br{\psi} (0,\theta) \begin{pmatrix}
			-\sin \theta \\
			\cos \theta
		\end{pmatrix} \right)
		\end{aligned}
	\end{equation}
	locally uniformly on $\bbR^2 \backslash \set{0}$.
\end{proposition}

\begin{proof}
	\begin{align*}
		u(x,t) \underset{\eqref{tilde func}}&{=} t^{\mu-1} \nabla^\perp_z \tld{\psi}(xt^{-\mu}) \\
		\underset{\eqref{diff change}}&{=} \frac{t^{\mu -1}}{e^a a_\varphi} \begin{pmatrix}
			-a_\phi \cos \theta +\sin \theta & a_\beta \cos \theta - \sin \theta \\
			-a_\phi \sin \theta - \cos \theta & a_\beta \sin \theta + \cos \theta \\
		\end{pmatrix} \begin{pmatrix}
			\rd_\beta \\
			\rd_\phi \\
		\end{pmatrix} \psi \\
		&= \frac{t^{\mu-1}}{e^a a_\varphi} \left(
		(-a_\phi \psi_\beta + a_\beta \psi_\phi) \begin{pmatrix}
			\cos \theta \\
			\sin \theta \\
		\end{pmatrix} + \psi_\varphi \begin{pmatrix}
			-\sin \theta \\
			\cos \theta \\
		\end{pmatrix} \right) \\
		&= \frac{t^{\mu-1}}{e^a a_\varphi} \left(
		(-a_\phi(\psi_\phi - \psi_\varphi ) + (a_\phi - a_\varphi)\psi_\phi) \begin{pmatrix}
			\cos \theta \\
			\sin \theta \\
		\end{pmatrix} + \psi_\varphi \begin{pmatrix}
			-\sin \theta \\
			\cos \theta \\
		\end{pmatrix} \right) \\
		&= \frac{t^{\mu-1}}{e^a a_\varphi} \left(
		(a_\phi \psi_\varphi - a_\varphi  \psi_\phi) \begin{pmatrix}
			\cos \theta \\
			\sin \theta \\
		\end{pmatrix} + \psi_\varphi \begin{pmatrix}
			-\sin \theta \\
			\cos \theta \\
		\end{pmatrix} \right) \\
		\underset{\eqref{new coordinate}}&{=} 2 t^{\mu -1} \sqrt{-\frac{\mu}{\psi_\beta}} \frac{\psi_\beta}{\psi_{\beta \varphi}} \left(
		\frac{1}{2} \left( \frac{\psi_{\beta \phi}\psi_\varphi - \psi_{\beta \varphi} \psi_\phi}{\psi_\beta} \right) \begin{pmatrix}
			\cos \theta \\
			\sin \theta \\
		\end{pmatrix}+ \psi_\varphi \begin{pmatrix}
			-\sin \theta \\
			\cos \theta \\
		\end{pmatrix} \right) \\
		\underset{\eqref{beta and bar beta} \sim \eqref{betavarphi and bar betavarphi}} &{=} 2 \left( \frac{\beta}{t} \right)^{1-\mu} \sqrt{-\frac{\mu}{\dbbeta \br{\psi}}} \frac{\dbbeta \br{\psi}}{\dbvarphibeta \br{\psi}} \left(
		\frac{\rd_\phi \dbbeta \br{\psi} \cdot \dbvarphi \br{\psi} - \dbvarphibeta \br{\psi} \cdot \rd_\phi \br{\psi}}{2 \dbbeta \br{\psi}} \begin{pmatrix}
			\cos \theta \\
			\sin \theta \\
		\end{pmatrix} + \dbvarphi \br{\psi} \begin{pmatrix}
			- \sin \theta \\
			\cos \theta \\
		\end{pmatrix} \right) \\
		\underset{\eqref{eq51}}&{=} \abs{x}^{1-\frac{1}{\mu}} \left( -\frac{\dbbeta \br{\psi}(\beta, \phi)}{\mu} \right)^{\frac{1}{2\mu} -1}  \frac{2 \dbbeta \br{\psi}(\beta, \phi)}{\dbvarphibeta \br{\psi}(\beta, \phi)} \\
		& \qquad \cdot \left(
		\frac{\rd_\phi \dbbeta \br{\psi}(\beta, \phi) \cdot \dbvarphi \br{\psi}(\beta, \phi) - \dbvarphibeta \br{\psi}(\beta, \phi) \cdot \rd_\phi \br{\psi}(\beta, \phi)}{2 \dbbeta \br{\psi}(\beta, \phi)} \begin{pmatrix}
			\cos \theta \\
			\sin \theta \\
		\end{pmatrix} + \dbvarphi \br{\psi}(\beta, \phi) \begin{pmatrix}
			- \sin \theta \\
			\cos \theta \\
		\end{pmatrix} \right)
	\end{align*}
	Then, the same argument for sending $t$ to zero as in Proposition \ref{initial data for stream function prop} gives the result.
\end{proof}

Lastly, we investigate initial data for $w(x,t)$. The situation is a little bit different from $\psi(x,t)$ or $u(x,t)$, because $w(x,t)$ is not continuous generally due to lacks of regularity on $\Omega(\phi)$. Therefore, we cannot guarantee local uniform convergence of $w(x,t)$ as $t$ goes to zero. However, we could guarantee local $L^p$-convergence for $w(x,t)$, actually \textit{including the origin}. Before dealing with convergence, we show:

\begin{proposition} \label{w is L^p}
	Assume $1 \leq p < 2 \mu$. Then, 
	\begin{equation*}
		w(\cdot, t) \in L^p_{\textnormal{loc}}(\bbR^2)
	\end{equation*}
	for all $t \in (0,\infty)$ with its local $L^p$-norm independent of $t$.
\end{proposition}

\begin{proof}
	Take an annulus $\br{B(\varepsilon, R)}$ for $R, \varepsilon>0$. Then,
	\begin{equation*}
		\int_{\br{B(\varepsilon, R)}} \abs{w(x,t)}^p \, dx \underset{\eqref{tilde func} }{=} \int_{\br{B(\varepsilon, R)}} \frac{1}{t^p} \abs{\tld{w}(xt^{-\mu})}^p \,dx \underset{xt^{-\mu} \rightarrow z}{=} \int_{\br{B(\varepsilon t^{-\mu}, Rt^{-\mu})}} \frac{1}{t^p}\abs{\tld{w}(z)}^p t^{2\mu} \, dz.
	\end{equation*}
	By \eqref{eq52} and \eqref{boundedness of bar derivative}, we have
	\begin{equation*}
		\abs{x}^{-\frac{1}{\mu}} \left( \frac{1}{2\mu} \right) ^{\frac{1}{2\mu}} t \leq \beta \leq \abs{x}^{-\frac{1}{\mu}} \left( \frac{3}{2\mu} \right) ^{\frac{1}{2\mu}} t.
	\end{equation*}
	Therefore, $T^{-1} \left( \br{B(\varepsilon t^{-\mu}, Rt^{-\mu})} \right)$, where $T$ is the coordinate change map in Proposition \ref{C^1 change of coordinate}, is contained in
	\begin{equation*}
		\left[ R^{-\frac{1}{\mu}} \left( \frac{1}{2\mu} \right)^{\frac{1}{2\mu}} t , ~~ \varepsilon^{-\frac{1}{\mu}} \left( \frac{3}{2\mu} \right)^{\frac{1}{2\mu}} t \right] \times \bbT.
	\end{equation*}
	Hence,
	\begin{equation} \label{eq90}
		\begin{aligned}
			& \int_{\br{B(\varepsilon t^{-\mu}, Rt^{-\mu})}} \frac{1}{t^p}\abs{\tld{w}(z)}^p t^{2\mu} \, dz \\
			\underset{(z_1, z_2) \xrightarrow[]{T^{-1}}(\beta, \phi),\,  \eqref{determinant}}&{\leq} t^{2 \mu -p} \int_{ R^{-\frac{1}{\mu}} \left( \frac{1}{2\mu} \right)^{\frac{1}{2\mu}} t}^{\varepsilon^{-\frac{1}{\mu}} \left( \frac{3}{2\mu}  \right)^{\frac{1}{2\mu}} t} \int_{\bbT} \abs{w(\beta, \phi)}^p e^{2a} a_\varphi \, d\beta d\phi \\
			\underset{\eqref{new coordinate}, \eqref{w formula} }&{=} \int_{ R^{-\frac{1}{\mu}} \left( \frac{1}{2\mu} \right)^{\frac{1}{2\mu}} t}^{\varepsilon^{-\frac{1}{\mu}} \left( \frac{3}{2\mu}  \right)^{\frac{1}{2\mu}} t} \int_{\bbT} \left( - \frac{\psi_{\beta \varphi}}{\mu} \right) ( \psi_\varphi)^{-\frac{p}{2\mu}} \abs{\Omega}^p \, d\beta d\phi \\
			\underset{\eqref{varphi and bar var phi}, \eqref{betavarphi and bar betavarphi} }&{=}  \int_{ R^{-\frac{1}{\mu}} \left( \frac{1}{2\mu} \right)^{\frac{1}{2\mu}} t}^{\varepsilon^{-\frac{1}{\mu}} \left( \frac{3}{2\mu}  \right)^{\frac{1}{2\mu}} t} \int_{\bbT} \beta^{p-2\mu -1} \frac{\dbvarphibeta \br{\psi}}{-\mu} ( \dbvarphi \br{\psi} )^{-\frac{p}{2\mu}} \abs{\Omega}^p \, d\beta d \phi \\
			\underset{\eqref{boundedness of bar derivative}, \text{Fubini} }&{\leq} 3 \cdot 2^{\frac{p}{2\mu}} t^{2\mu -p} \cdot \int_\bbT \abs{\Omega}^p \, d\phi \cdot \int_{[ R^{-\frac{1}{\mu}} \left( \frac{1}{2\mu} \right)^{\frac{1}{2\mu}} t , ~~ \varepsilon^{-\frac{1}{\mu}} \left( \frac{3}{2\mu}  \right)^{\frac{1}{2\mu}} t ]} \beta^{p-2\mu-1} \, d\beta \\
			&= \frac{3 \cdot 2^{\frac{p}{2\mu}}}{2\mu -p} \cdot \int_\bbT \abs{\Omega}^p \, d\phi \cdot \left( \left( \frac{1}{2\mu} \right)^{\frac{p}{2\mu}-1} R^{2-\frac{p}{\mu}} - \left( \frac{3}{2\mu} \right)^{\frac{1}{2\mu}} \varepsilon^{2-\frac{p}{\mu}} \right).
		\end{aligned}
	\end{equation}
	Taking $\varepsilon \rightarrow 0$ in \eqref{eq90}, we have
	\begin{equation*}
		w(\cdot, t) \in L^p ( \br{B(0, R)})
	\end{equation*}
	with its time-independent norm bound
	\begin{equation} \label{eq91}
		\nrm{w(\cdot, t )}_{L^p(\br{B(0,R)})} \leq \left( \frac{6\mu}{2\mu-p} \right)^{\frac{1}{p}} \mu^{-\frac{1}{2\mu}} \nrm{\Omega})_{L^p(\bbT)} R^{\frac{2}{p}-\frac{1}{\mu}}.
	\end{equation}
\end{proof}

Now, we state convergence of $w$ as $t \rightarrow 0$ in $L^p_{\textnormal{loc}}(\bbR^2)$.

\begin{proposition} \label{initial data for vorticity prop}
	For $1 \leq p<2\mu$, the vorticity $w(x,t)$ satisfies
	\begin{equation} \label{initial data for vorticity}
		\lim_{t \searrow 0} w(x,t) = \abs{x}^{-\frac{1}{\mu}} \left( \frac{\dbbeta \br{\psi}(0,\theta)}{-\mu \dbvarphi \br{\psi} (0, \theta)} \right)^{\frac{1}{2\mu}} \Omega(\theta),
	\end{equation}
	where the convergence is local $L^p$-convergence on $\bbR^2$.
	\end{proposition}

\begin{proof}
	We first anticipate $\lim_{t\rightarrow 0}w(x,t)$. From the equality
	\begin{equation} \label{eq92}
	\begin{aligned}
		w(x,t) \underset{\eqref{tilde func}}&{=} t^{-1} \tld{w} (xt^{-\mu}) \underset{(\beta, \phi) = T^{-1} (xt^{-\mu})}{\overset{\eqref{w formula}}{=}} t^{-1} \left( \psi_\varphi (\beta, \phi) \right)^{-\frac{1}{2\mu}} \Omega(\phi)\\
		 \underset{\eqref{varphi and bar var phi}}&{=} \frac{\beta}{t}\left( \dbvarphi \br{\psi} (\beta, \phi) \right)^{-\frac{1}{2\mu}} \Omega(\phi) \\
		 \underset{\beta+\phi=\theta,\, \eqref{eq52}}&{=} \abs{x}^{-\frac{1}{\mu}} \left( -\frac{\dbbeta \br{\psi} (\beta, \theta - \beta)}{\mu \dbvarphi \br{\psi} (\beta, \theta - \beta)} \right)^{\frac{1}{2\mu}} \Omega(\theta - \beta)
	\end{aligned}
	\end{equation}
	and the fact that $t \rightarrow 0$ corresponds to $(\beta, \phi) \rightarrow (0, \theta)$ on each compact set of $\rmz$ as seen in \eqref{eq54}, we can guess\footnote{This is not so rigorous, because the Proposition \ref{initial data for vorticity prop} deal with local $L^p$-convergence including the origin, while \eqref{eq54} does not.}
	\begin{equation*}
		\lim_{t \searrow 0} w(x,t) = \abs{x}^{-\frac{1}{\mu}} \left( \frac{\dbbeta \br{\psi}(0,\theta)}{-\mu \dbvarphi \br{\psi} (0, \theta)} \right)^{\frac{1}{2\mu}} \Omega(\theta)
	\end{equation*}
	in $L^p_{\textnormal{loc}}(\bbR^2)$. We prove this rigorously. For $R>0, t>0$,
	\begin{equation} \label{eq93}
		\begin{aligned}
			& \int_{\br{B(0,R)}} \abs{w(x,t)-\abs{x}^{-\frac{1}{\mu}} \left( \frac{\dbbeta \br{\psi}(0,\theta)}{-\mu \dbvarphi \br{\psi} (0, \theta)} \right)^{\frac{1}{2\mu}} \Omega(\theta)}^p \, dx \\
			\underset{\eqref{tilde func} }&{=} \int_{\br{B(0,R)}} \abs{\frac{1}{t}\tld{w}(xt^{-\mu}) - \abs{x}^{-\frac{1}{\mu}} \left( \frac{\dbbeta \br{\psi}(0,\theta)}{-\mu \dbvarphi \br{\psi} (0, \theta)} \right)^{\frac{1}{2\mu}} \Omega(\theta) }^p \, dx \\
			\underset{xt^{-\mu} \rightarrow z}&{=} \int_{\br{B(0,Rt^{-\mu})}} \abs{\frac{1}{t} \tld{w}(z)-\frac{1}{t} \abs{z}^{-\frac{1}{\mu}} \left( \frac{\dbbeta \br{\psi}(0,\theta)}{-\mu \dbvarphi \br{\psi} (0, \theta)} \right)^{\frac{1}{2\mu}} \Omega(\theta) }^p t^{2\mu} \, dz \\
			\underset{(z_1, z_2) \rightarrow (\beta, \phi) }&{=} \int \int_{T^{-1}(\br{B(0, Rt^{-\mu})})} t^{2\mu -p} \abs{w(\beta, \phi)-e^{-\frac{a(\beta, \phi)}{\mu}} \left( \frac{\dbbeta \br{\psi}(0,\beta+\phi)}{-\mu \dbvarphi \br{\psi} (0, \beta+\phi)} \right)^{\frac{1}{2\mu}} \Omega(\beta+\phi)  }^p e^{2a} a_\varphi \, d\beta d\phi \\
			\underset{\eqref{w formula}, \eqref{new coordinate} }&{=} \int \int_{T^{-1}(\br{B(0, Rt^{-\mu})})} t^{2\mu -p} \abs{ \left( \psi_\varphi (\beta, \phi) \right)^{-\frac{1}{2\mu}} \Omega(\phi)- \left( -\frac{\psi_\beta(\beta, \phi)}{\mu} \right)^{-\frac{1}{2\mu}} \left( \frac{\dbbeta \br{\psi}(0,\beta+\phi)}{-\mu \dbvarphi \br{\psi} (0, \beta+\phi)} \right)^{\frac{1}{2\mu}} \Omega(\beta+\phi)  }^p \\
			& \qquad \qquad \cdot \left( -\frac{\psi_{\beta \varphi}}{\mu} \right) \, d\beta d\phi \\
			\underset{\eqref{beta and bar beta}, \eqref{varphi and bar var phi}, \eqref{betavarphi and bar betavarphi} }&{=} \int \int_{T^{-1}(\br{B(0, Rt^{-\mu})})} t^{2\mu -p} \beta^{p-2\mu-1} \cdot \left( - \frac{\dbvarphibeta \br{\psi}(\beta, \phi)}{\mu} \right) \\
			& \qquad \qquad \abs{\frac{\Omega(\phi)}{(\dbvarphi \br{\psi} (\beta, \phi) )^{\frac{1}{2\mu}}} - \frac{1}{(\dbvarphi \br{\psi} (0, \beta+ \phi) )^{\frac{1}{2\mu}}} \left( \frac{-\dbbeta \br{\psi} (0, \beta+\phi)}{-\dbbeta \br{\psi} (\beta, \phi)} \right)^{\frac{1}{2\mu}} \Omega(\beta+\phi) }^p \, d\beta d\phi \\
			\underset{\eqref{boundedness of three functions} }&{\leq} 3 \int \int_{T^{-1}(\br{B(0, Rt^{-\mu})})} t^{2\mu -p} \beta^{p-2\mu-1}  \\
			& \qquad \qquad \abs{\frac{\Omega(\phi)}{(\dbvarphi \br{\psi} (\beta, \phi) )^{\frac{1}{2\mu}}} - \frac{1}{(\dbvarphi \br{\psi} (0, \beta+ \phi) )^{\frac{1}{2\mu}}} \left( \frac{-\dbbeta \br{\psi} (0, \beta+\phi)}{-\dbbeta \br{\psi} (\beta, \phi)} \right)^{\frac{1}{2\mu}} \Omega(\beta+\phi) }^p \, d\beta d\phi \\
			\underset{\beta \rightarrow t \gamma, \, \eqref{eq52}, \eqref{boundedness of bar derivative} }&{\leq} 3 \int_\bbT \int_{(2\mu)^{-\frac{1}{2\mu}} R^{-\frac{1}{\mu}} }^\infty \gamma^{p-2\mu-1} \\
			 & \qquad \qquad \qquad \abs{\frac{\Omega(\phi)}{(\dbvarphi \br{\psi} (t \gamma, \phi) )^{\frac{1}{2\mu}}} - \frac{1}{(\dbvarphi \br{\psi} (0, t \gamma + \phi) )^{\frac{1}{2\mu}}} \left( \frac{-\dbbeta \br{\psi} (0, t \gamma+\phi)}{-\dbbeta \br{\psi} (t \gamma, \phi)} \right)^{\frac{1}{2\mu}} \Omega(t \gamma+\phi) }^p \, d\gamma d\phi \\
		\end{aligned}
	\end{equation}
	\begin{equation*}
		\begin{aligned}
		&\leq 3\cdot 2^p \int_\bbT \int_{(2\mu)^{-\frac{1}{2\mu}}R^{-\frac{1}{\mu}} }^\infty \gamma^{p-2\mu-1} \frac{1}{(\dbvarphi \br{\psi} (t\gamma, \phi))^{\frac{p}{2\mu}}} \abs{\Omega(\phi)-\Omega(t \gamma +\phi) }^p \, d\gamma d\phi \\
			 & \qquad + 3\cdot 2^p \int_\bbT \int_{(2\mu)^{-\frac{1}{2\mu}}R^{-\frac{1}{\mu}} }^\infty \gamma^{p-2\mu-1} \Bigg\lvert \left( \frac{1}{\dbvarphi \br{\psi} (t\gamma, \phi)} \right)^{\frac{1}{2\mu}} \\
			 & \qquad \qquad \qquad - \frac{1}{(\dbvarphi \br{\psi} (0, t \gamma + \phi) )^{\frac{1}{2\mu}}} \left( \frac{-\dbbeta \br{\psi} (0, t \gamma+\phi)}{-\dbbeta \br{\psi} (t \gamma, \phi)} \right)^{\frac{1}{2\mu}}\Bigg\rvert ^p \abs{\Omega(t \gamma+\phi)}^p \, d\gamma d\phi \\
			\underset{\eqref{boundedness of bar derivative},\, \text{Fubini} }&{\leq} 3 \cdot 2^{p+\frac{p}{2\mu}} \int_{(2\mu)^{-\frac{1}{2\mu}} R^{-\frac{1}{\mu}}}^\infty \gamma^{p-2\mu-1} \int_\bbT \abs{\Omega(\phi) - \Omega(t\gamma+\phi) }^p \, d\phi d\gamma \\
			 & \qquad + 3\cdot 2^p \int_\bbT \int_{(2\mu)^{-\frac{1}{2\mu}}R^{-\frac{1}{\mu}} }^\infty \gamma^{p-2\mu-1} \Bigg\lvert \left( \frac{1}{\dbvarphi \br{\psi} (t\gamma, \phi)} \right)^{\frac{1}{2\mu}} \\
			 & \qquad \qquad \qquad - \frac{1}{(\dbvarphi \br{\psi} (0, t \gamma + \phi) )^{\frac{1}{2\mu}}} \left( \frac{-\dbbeta \br{\psi} (0, t \gamma+\phi)}{-\dbbeta \br{\psi} (t \gamma, \phi)} \right)^{\frac{1}{2\mu}}\Bigg\rvert ^p \abs{\Omega(t \gamma+\phi)}^p \, d\gamma d\phi. \\
		\end{aligned}
	\end{equation*}
	Since
	\begin{equation*} 
	\left\{\begin{aligned}
		& \int_\bbT \abs{\Omega(\phi) - \Omega( t\gamma + \phi)}^p \xrightarrow[t \searrow 0]{} 0 \qquad \text{ for all }\gamma>0, \\
		& \Bigg\lvert \left( \frac{1}{\dbvarphi \br{\psi} (t\gamma, \phi)} \right)^{\frac{1}{2\mu}} - \frac{1}{(\dbvarphi \br{\psi} (0, t \gamma + \phi) )^{\frac{1}{2\mu}}} \left( \frac{-\dbbeta \br{\psi} (0, t \gamma+\phi)}{-\dbbeta \br{\psi} (t \gamma, \phi)} \right)^{\frac{1}{2\mu}}\Bigg\rvert ^p \xrightarrow[t \searrow 0]{} 0 \qquad \text{ for all } \gamma>0, \phi \in \bbT, \\
		& \int_\bbT \int_{(2\mu)^{-\frac{1}{2\mu}}}^\infty \gamma^{p-2\mu -1} \abs{\Omega(\phi)}^p \, d\gamma d\phi \qquad \text{ is integrable},
	\end{aligned}\right.
\end{equation*}
Lebesgue's dominate convergence theorem yields
\begin{equation*}
	\lim_{t \searrow 0 } \nrm{w(\cdot, t) - w(\cdot, 0)}_{L^p(\br{B(0,R)})} =0.
\end{equation*}
\end{proof}

\begin{definition} \label{initial data}
	From now on, we define initial data $(w(\cdot,0), u(\cdot, 0), \psi(\cdot,0))$ by the right hand side of \eqref{initial data for vorticity}, \eqref{initial data for velocity}, \eqref{initial data for stream function}.
\end{definition}

Combining results and their proof of Proposition \ref{initial data for stream function prop}, \ref{initial data for velocity prop}, \ref{w is L^p}, \ref{initial data for vorticity prop}, we have the following boundedness property.

\begin{corollary} \label{boundedness of three functions}
	There exists a positive constant $A=A(\Omega)$ and $C=C(R)$ such that
	\begin{equation*}
		\abs{\psi(x,t)} \leq A \abs{x}^{2-\frac{1}{\mu}}, \qquad \abs{u(x,t)} \leq A \abs{x}^{1-\frac{1}{\mu}}, \qquad \nrm{w(\cdot,t)}_{L^p(\br{B(0,R)})} \leq C
	\end{equation*}
	for all $1\leq p <2\mu$ and $t \in [0, \infty)$. 
\end{corollary}

\bigskip

\subsection{The spiral roll-up phenomenon}
\label{subsec: The spiral roll-up phenomenon}

As we clarified initial vorticity in Proposition \ref{initial data for velocity prop}, we can show spiral roll-up phenomenon of $w(x,t)$. Let $\phi_0 \in \bbT$ be an angle such that $\Omega(\phi_0)=0$. Then, the line in $(\beta, \phi)$ - coordinate
\begin{equation*}
	\ell \coloneqq \set{(\beta, \phi_0) \in \bbR_+ \times \bbT \, \vert \, \beta \in \bbR_+}
\end{equation*}
is transformed to
\begin{equation} \label{eq91}
	T(\ell) = \left( -\frac{\dbbeta \br{\psi}(\beta, \phi_0))}{\mu} \right)^{\frac{1}{2}} \beta^{-\mu} \Big( \cos (\beta+\phi_0 ), \sin (\beta+\phi_0 ) \Big), \qquad \beta \in \bbR_+
\end{equation}
in $(z_1, z_2)$ - coordinate by \eqref{new coordinate}, \eqref{map T} and \eqref{beta and bar beta}. Since
\begin{equation*}
	\sqrt{\frac{1}{2\mu}} \leq \left( -\frac{\dbbeta \br{\psi} (\beta, \phi_0)}{\mu} \right)^{\frac{1}{2}} \leq \sqrt{\frac{3}{2\mu}}
\end{equation*}
for all $(\beta, \phi) \in \RT$ by \eqref{boundedness of bar derivative}, $T(\ell)$ is nearly an algebraic spiral.

\begin{figure}[!h]	
	\includegraphics[trim = 31mm 206mm 7mm 33mm, clip, width=\textwidth]{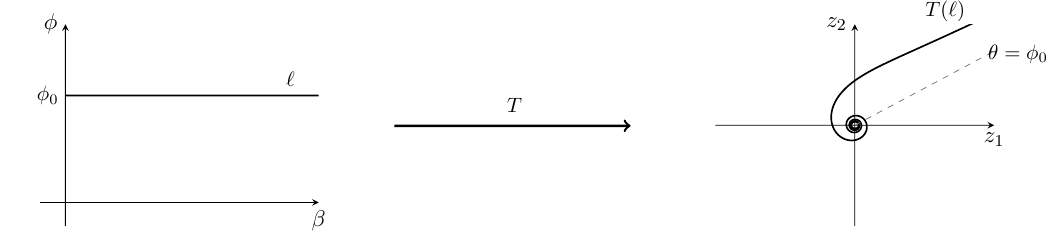}
	\caption{Line $\ell=\{ \phi=\phi_0 \}$ is push forwarded to nearly algebraic spiral, i.e., $T(\ell) \sim \abs{z}^{-\mu}$.} \label{Fig 4}
\end{figure}

As
\begin{equation} \label{eq92}
	\tld{w} \left( T ( \beta, \phi_0 ) \right) \underset{\eqref{pull back notation}}{=} w(\beta, \phi_0 ) \underset{\eqref{w formula}}{=} \beta \left( \dbvarphi \br{\psi}(\beta, \phi_0 ) \right)^{-\frac{1}{2\mu}} \Omega(\phi_0)=0
\end{equation}
for all $\beta \in \bbR_+$, we have $\tld{w}(z)=0$ for all $z \in T(\ell) \subset \bbR^2 \backslash \set{0}$. Since
\begin{equation*}
	w(x,t)=\frac{1}{t} \tld{w}(xt^{-\mu})
\end{equation*}
we get that
\begin{equation*}
	w(x,t)=0
\end{equation*}
for all 
 \begin{equation*}
 	x \in t^\mu T(\ell) \underset{\eqref{eq91} }{=}  \left( -\frac{\dbbeta \br{\psi}(\beta, \phi_0))}{\mu} \right)^{\frac{1}{2}}  \underbrace{\left( \frac{t}{\beta} \right)^{\mu} \Big( \cos (\beta+\phi_0 ), \sin (\beta+\phi_0 ) \Big)}_{\text{algebraic spiral for all }\phi_0 \in \bbT, \, t>0 } \subset \bbR^2 \backslash \set{0}.
 \end{equation*} Also, \eqref{eq92} implies that  $\tld{w}(T(\beta, \phi_1)) \neq 0$  for $\phi_1$ such that $\Omega(\phi_1) \neq 0$. This shows spiral roll-up phenomenon in Definition \ref{spiral roll-up}. Note that as $t \rightarrow 0$, $t^\mu T(\ell)$ converges to half-line with angle $\theta=\phi_0$ as shown in Proposition \ref{initial data for vorticity prop}.

\begin{figure}[!h]
	\includegraphics[trim = 31mm 208mm 7mm 33mm, clip, width=\textwidth]{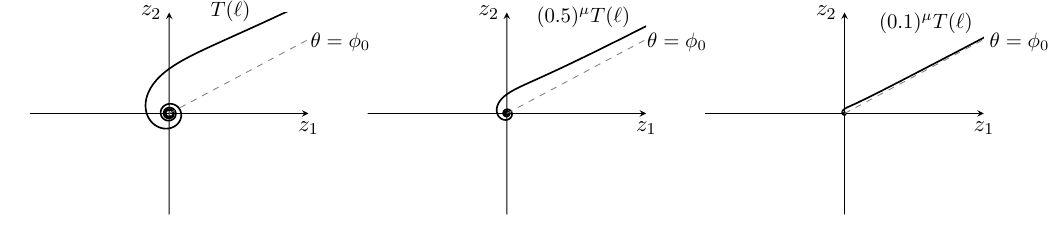}
	\caption{Spiral roll-up phenomenon.} \label{Fig 5}
\end{figure}

\bigskip

\subsection{Weak solution on $\bbR^2 \times [0,\infty)$}
\label{subsec: Weak solution; strong version }

Having clarified initial data as in Definition \ref{initial data} and investigated time independent boundedness in Corollary \ref{boundedness of three functions}, we can prove that $\psi, u, w$ is a weak solution of Euler equation including time zero and origin. First, we treat velocity form. The proof is almost same as in [Elling \cite{Elling-2013}, Section 8.3.].

\begin{proposition} \label{weak solution of EE in velocity form upgraded}
	\begin{equation*}
		u_t+\nabla_x \cdot (u \otimes u) + \nabla_x p =0
	\end{equation*}
	on $\bbR^2 \times [0,\infty)$ in the sense of distribution, i.e., for any divergence free vector field $g \in C^\infty_c \left( \bbR^2 \times [0, \infty) ; \bbR^2 \right)$,
	\begin{equation*}
		\int_{\bbR^2} u(x,0) \cdot g(x,0) \, dx + \int_{\bbR^2 \times [0, \infty)} u \cdot g_t + (u \otimes u): \nabla_x g \, dxdt=0.
	\end{equation*}
\end{proposition}

\begin{proof}
	As we already proved Proposition \ref{weak solution for u intermediate 2}, what we have to consider is integration near the origin and time zero. Consider a smooth, radial cut-off function $\zeta:\bbR^2 \rightarrow [0,1]$ such that
	\begin{equation*}
		\zeta(r) \equiv 1 \quad \text{on} \quad B(0,1), \qquad \zeta(r) \equiv 0 \quad \text{on} \quad \bbR^2 \backslash B(0,2)
	\end{equation*}
	and define $\zeta^\delta$ by $\zeta \left( \frac{\cdot}{\delta} \right)$. Then,
	\begin{equation} \label{eq66}
		\supp (\zeta^\delta ) \subset \br{B(0, 2\delta)}, \qquad \supp (\nabla_x \zeta^\delta ) \subset \br{B(0, 2\delta)} \backslash B(0,\delta), \qquad \supp ( 1- \zeta^\delta) \subset \bbR^2 \backslash B(0,\delta)
	\end{equation}
	and there exists $M_1>0$ such that
	\begin{equation} \label{eq67}
		\nrm{\nabla_x \zeta^\delta}_{L^\infty (\bbR^2)} < \frac{M_1}{\delta}, \qquad \nrm{\nabla_x^2 \zeta^\delta}_{L^\infty(\bbR^2)} < \frac{M_1}{\delta^2}.
	\end{equation}
	Now, we analyze a test function. Since $g$ is a divergence free test function, we can take a scalar function $f \in C^\infty_c \left( \bbR^2 \times [0,\infty) \right)$ such that $\nabla_x^\perp f = g$. Replacing $f(x,t)$ with $f(x,t)-f(0,t)$, we can assume that $f(0,t)=0$ for all $t \in [0,\infty)$.\footnote{After replacing, $f$ has not compact support anymore, but it does not matter since we will only treat $f$ after cut-off or differentiation.} As $f_t(0,t)=0$ for all $t \in [0, \infty)$ also, there exists $M_2>0$ and $R>0$ such that
	\begin{equation} \label{eq68}
		\abs{f(x,t)}, \abs{f_t(x,t)} \leq M_2 \abs{x}
	\end{equation}
	for all $0 \leq \abs{x} \leq R$ and $t\geq 0$, and
	\begin{equation} \label{eq69}
		\nrm{\nabla f (\cdot,t)}_{L^\infty (\bbR^2)}, \nrm{\nabla^2 f(\cdot, t)}_{L^\infty (\bbR^2)}, \nrm{\nabla f_t (\cdot, t)}_{L^\infty (\bbR^2)}, \nrm{\nabla^\perp f_t(\cdot,t)}_{L^\infty (\bbR^2)} \leq M_2
	\end{equation}
	for all $t\geq 0$. Let $M= \maxi{M_1, M_2}$. Then, for all $0< \delta \leq \mini{\frac{R}{2}, \frac{1}{M}}$ and $t \geq 0$,
	\begin{equation} \label{eq70}
		\begin{aligned}
			&\nrm{\nabla(\zeta^\delta f(\cdot, t))}_{L^\infty(\bbR^2 )} \\\underset{\eqref{eq66}}&{=} \nrm{\nabla(\zeta^\delta f(\cdot, t))}_{L^\infty(B(0,2\delta))} \\
			& \leq \nrm{\nabla \zeta^\delta}_{L^\infty(B(0,2\delta) )} \nrm{f(\cdot, t)}_{L^\infty(B(0,2\delta) )} + \nrm{\nabla f(\cdot, t)}_{L^\infty(B(0,2\delta) )} \nrm{\zeta^\delta}_{L^\infty(B(0,2\delta) )} \\
			\underset{\eqref{eq67}, \eqref{eq68}, \eqref{eq69}} &{\leq} \frac{M}{\delta} \cdot 2M\delta + M \cdot 1 \\
			&\leq 2M^2+1,
		\end{aligned}
	\end{equation}
	and $\nrm{\nabla(\zeta^\delta f_t(\cdot, t))}_{L^\infty(\bbR^2)}$ has same bound. Similarly, for all $t\geq 0$,
	\begin{equation} \label{eq71}
		\begin{aligned}
			& \nrm{\nabla^2 (\zeta^\delta f(\cdot, t))}_{L^\infty(\bbR^2)} \\
			\underset{\eqref{eq66}}&{=} \nrm{\nabla^2 (\zeta^\delta f(\cdot, t))}_{L^\infty(B(0,2\delta))} \\
			&\leq \nrm{\nabla^2 \zeta^\delta}_{L^\infty(B(0,2\delta) )} \nrm{f(\cdot, t)}_{L^\infty(B(0,2\delta) )}+2 \nrm{\nabla \zeta^\delta}_{L^\infty(B(0,2\delta) )} \nrm{\nabla f(\cdot, t)}_{L^\infty(B(0,2\delta) )} \\
			& \qquad + \nrm{\zeta^\delta}_{L^\infty(B(0,2\delta) )} \nrm{\nabla^2 f(\cdot, t)}_{L^\infty(B(0,2\delta) )} \\
			\underset{\eqref{eq67}, \eqref{eq68}, \eqref{eq69}  }&{\leq} \frac{M}{\delta^2} \cdot 2M\delta + \frac{2M}{\delta} \cdot M + 1 \cdot M \\
			&\leq \frac{4M^2+1}{\delta}.
		\end{aligned}
	\end{equation}
	
	Now, it is time to use $\mu> \frac{2}{3}$ condition crucially. Let
	\begin{equation*}
		p \coloneqq \begin{cases}
			1+ \frac{1}{2\left( \frac{1}{\mu}-1 \right)}=\frac{2-\mu}{2(1-\mu)} & (\frac{2}{3}<\mu <1) \\
			3 & (\mu \geq 1)
		\end{cases}.
	\end{equation*}
	Then, we have $p>2$ and $u(x,t) \in L^{2p}_{\textnormal{loc}}(\bbR^2)$ with its local norm uniformly controlled independent of t, because for $\frac{2}{3} <\mu <1$, 
	\begin{equation*}
		\abs{u(x,t)}^{2p} \underset{\text{Cor } \ref{boundedness of three functions}}{\leq} A^{2+\frac{1}{\frac{1}{\mu}-1}} \abs{x}^{2\left(1-\frac{1}{\mu} \right)-1}
	\end{equation*}
	with $2 \left( 1-\frac{1}{\mu} \right) -1 >-2$. For $\mu \geq 1$, $u(x,t)$ has no singularity at origin, so local integrability is trivial. When we denote conjugate exponent of $p$ by $p'$, then $p'<2$ and
	\begin{equation} \label{eq72}
	\begin{aligned}
		\nrm{\nabla( \zeta^\delta f(\cdot, t))}_{L^{p'}(\bbR^2)} \underset{\eqref{eq66}}&{=} \nrm{\nabla( \zeta^\delta f(\cdot, t))}_{L^{p'}(B(0,2\delta))} \\
		& \leq  \nrm{\nabla (\zeta^\delta f(\cdot,t)}_{L^\infty(B(0,2\delta))} \cdot \left( (2\delta)^2 \pi \right)^{\frac{1}{p'}} \\
		\underset{\eqref{eq70}}&{\leq} (2M^2+1)(4 \pi )^{\frac{1}{p'}} \cdot \delta^{\frac{2}{p'}}
	\end{aligned}
	\end{equation}
	for all $0< \delta < \mini{\frac{R}{2}, \frac{1}{M}}, t \geq 0$. We could get same bound for $\nrm{\nabla ( \zeta^\delta f_t (\cdot, t ))}_{L^{p'} (\bbR^2)}$. Also,
	\begin{equation} \label{eq73}
		\begin{aligned}
			\nrm{\nabla^2 ( \zeta^\delta f( \cdot, t))}_{L^{p'}(\bbR^2)} \underset{\eqref{eq66}}&{=} \nrm{\nabla^2 ( \zeta^\delta f (\cdot, t))}_{L^{p'}(B(0, 2\delta ))} \\
			&\leq \nrm{\nabla^2 (\zeta^\delta f (\cdot,t ))}_{L^\infty (B(0,2\delta))} \cdot \left( (2\delta)^2 \pi \right)^{\frac{1}{p'}}  \\
			\underset{ \eqref{eq71} }&{\leq} (4M^2+1)(4 \pi)^{\frac{1}{p'}} \cdot \delta^{\frac{2}{p'}-1}
		\end{aligned}
	\end{equation}
	for all $0 < \delta < \mini{\frac{R}{2}, \frac{1}{M}}, t \geq 0$.

	Now, all the preparations for proving Proposition \ref{weak solution of EE in velocity form upgraded} have been done. By Proposition \ref{weak solution for u intermediate 2}, using $(1-\zeta^\delta) f \chi_{\{ t\geq s \}}$ as a test function, we have
	\begin{equation} \label{eq106}
		\int_{\bbR^2} u(x,s) \cdot \nabla^\perp \left( (1-\zeta^\delta) f (x,s) \right) \, dx + \int_{\bbR^2 \times [s, \infty)} u \cdot \nabla^\perp \left( (1-\zeta^\delta) f_t \right) + (u \otimes u): \nabla \nabla^\perp \left( (1-\zeta^\delta) f \right) \, dxdt = 0
	\end{equation}
	for all $0< \delta < \mini{\frac{R}{2}, \frac{1}{M}}, s>0$. Since $(1-\zeta^\delta) f(x,t) \in C^\infty_c \left( \bbR^2 \backslash B(0,2\delta) \times [0, \infty) \right)$, all its derivative converges uniformly as $t \rightarrow 0$. Also, since we puncture the disk near origin through $1-\zeta^\delta$, Proposition \ref{initial data for vorticity prop} implies
	\begin{equation} \label{eq74}
		\begin{aligned}
			& \int_{\bbR^2} u(x,0) \cdot \nabla^\perp \left( (1-\zeta^\delta) f (x,0) \right) \, dx \\
			&+ \int_{\bbR^2 \times [0, \infty)} u \cdot \nabla^\perp \left( (1-\zeta^\delta) f_t \right) + (u \otimes u): \nabla \nabla^\perp \left( (1-\zeta^\delta) f \right) \, dxdt \\
			\underset{\text{Prop } \ref{initial data for vorticity prop} }&{=} \lim_{s \searrow 0} \int_{\bbR^2} u(x,s) \cdot \nabla^\perp \left( (1-\zeta^\delta) f (x,s) \right) \, dx \\
			&+ \lim_{s \searrow 0} \int_{\bbR^2 \times [s, \infty)} u \cdot \nabla^\perp \left( (1-\zeta^\delta) f_t \right) + (u \otimes u): \nabla \nabla^\perp \left( (1-\zeta^\delta) f \right) \, dxdt \\
			=0
		\end{aligned}
	\end{equation}
	for all $0 < \delta < \mini{\frac{R}{2}, \frac{1}{M}}$. Now, we send $\delta$ to zero. By \eqref{eq72} and \eqref{eq73}, we have
	\begin{equation} \label{eq75}
		\nabla^\perp \left( (1-\zeta^\delta)f(\cdot, t) \right) \xrightarrow[\delta \searrow 0]{}\nabla^\perp f(\cdot,t ), \qquad \nabla \nabla^\perp \left( (1-\zeta^\delta ) f(\cdot,t) \right) \xrightarrow[\delta \searrow 0]{} \nabla \nabla^\perp f(\cdot, t)
	\end{equation}
	in $L^{p'}(\bbR^2)$ for uniform rate with respect to $t$. As $u \in L^{2p}_{\textnormal{loc}} (\bbR^2)$ (hence $u \in L^p_{\textnormal{loc}}(\bbR^2)$ also) with its norm on compact set uniformly controlled independent of $t$ as in the proof of Proposition \ref{initial data for velocity prop}, H\"older inequality yields
	\begin{align*}
		& \int_{\bbR^2} u(x,0) \cdot \nabla^\perp  f (x,0)  \, dx \\
			&+ \int_{\bbR^2 \times [0, \infty)} u \cdot \nabla^\perp  f_t + (u \otimes u): \nabla \nabla^\perp f  \, dxdt \\
		\underset{\substack{\eqref{eq75},\, \text{Prop } \ref{initial data for vorticity prop}, \\ \text{H\"older}}}&{=} \lim_{\delta \searrow 0} \int_{\bbR^2} u(x,0) \cdot \nabla^\perp \left( (1-\zeta^\delta) f (x,0) \right) \, dx \\
			&+ \lim_{\delta \searrow 0} \int_{\bbR^2 \times [0, \infty)} u \cdot \nabla^\perp \left( (1-\zeta^\delta) f_t \right) + (u \otimes u): \nabla \nabla^\perp \left( (1-\zeta^\delta) f \right) \, dxdt \\
			\underset{\eqref{eq106} }&{=}0.
	\end{align*}
	
	Since $\nabla^\perp f=g$, this completes the proof.
\end{proof}

In a similar way, it can be shown that $w(x,t)$ is a weak solution of Euler equation of vorticity form \eqref{Euler eq-vor} on $\bbR^2 \times [0, \infty)$. A difference is that we have to require stronger condition on $\mu$ to utilize Proposition \ref{w is L^p} and \ref{initial data for vorticity prop} for eliminating singularity at origin.

\begin{proposition} \label{weak solution of EE in vorticity form upgraded}
	If $\mu>1$,
	\begin{equation*}
		w_t + \nabla_x \cdot (wu) =0
	\end{equation*}
	on $\bbR^2 \times [0, \infty)$ in the sense of distribution, i.e., for any test function $f \in C^\infty_c \left( \bbR^2 \times [0, \infty) \right)$,
	\begin{equation*}
		\int_{\bbR^2} w(x,0) f(x,0) \, dx + \int_{\bbR^2 \times [0, \infty)}  wf_t + w(u \cdot \nabla f) \, dxdt=0.
	\end{equation*}
\end{proposition}

\begin{proof}
	Take the cut-off function $\zeta^\delta$ same as in the proof of Proposition \ref{weak solution of EE in velocity form upgraded} and let $f \in C^\infty_c \left( \bbR^2 \times [0, \infty) \right)$ be given. As we cannot assume $f(0,t) \equiv 0$, we cannot expect bound \eqref{eq68}. Instead, we could just guarantee that there exists $M_2>0$ such that
	\begin{equation} \label{eq76}
		\nrm{f(\cdot,t)}_{L^\infty(\bbR^2)}, \nrm{\nabla f(\cdot, t)}_{L^\infty(\bbR^2)}, \nrm{f_t (\cdot,t)}_{L^\infty(\bbR^2)}, \nrm{\nabla f_t (\cdot, t)}_{L^\infty(\bbR^2)} \leq M_2
	\end{equation}
	for all $t \geq 0$. Let $M = \maxi{M_1, M_2}$. Then, for all $0 < \delta < \frac{1}{M}$, $t \geq 0$
	\begin{equation} \label{eq77}
		\begin{aligned}
			& \nrm{\nabla ( \zeta^\delta f )}_{L^\infty (\bbR^2)} \\
			\underset{\eqref{eq66} }&{=} \nrm{\nabla \zeta^\delta}_{L^\infty(B(0,2\delta))} \nrm{f}_{L^\infty(B(0,2\delta))} + \nrm{\zeta^\delta}_{L^\infty(B(0,2\delta))} \nrm{\nabla f}_{L^\infty(B(0,2\delta))} \\
			\underset{\eqref{eq67}, \eqref{eq76} }&{\leq} \frac{M^2}{\delta} + M \\
			&\leq \frac{M^2+1}{\delta}.			
		\end{aligned}
	\end{equation}
	
	Then, for all $r \geq 1$, we have
	\begin{equation} \label{eq78}
		\begin{aligned}
			\nrm{\zeta^\delta f (\cdot, t)}_{L^{r'}(\bbR^2)} \underset{\eqref{eq66} }&{=} \nrm{\zeta^\delta f (\cdot, t)}_{L^{r'} (B(0,2\delta))} \\
			&\leq \nrm{\zeta^\delta f(\cdot, t)}_{L^\infty(B(0,2 \delta))} \left( \pi (2 \delta)^2 \right)^{\frac{1}{r'}} \\
			&\leq M (4 \pi)^{\frac{2}{r'}} \delta^{\frac{2}{r'}}
		\end{aligned}
	\end{equation}
	with same bound for $f_t$. Also,
	\begin{equation} \label{eq79}
		\begin{aligned}
			\nrm{\nabla(\zeta^\delta f (\cdot, t))}_{L^{r'}(\bbR^2)} \underset{\eqref{eq66} }&{=} \nrm{\nabla(\zeta^\delta f (\cdot, t)}_{L^{r'} (B(0,2\delta))} \\
			&\leq \nrm{\nabla(\zeta^\delta f(\cdot, t))}_{L^\infty(B(0,2 \delta))} \left( \pi (2 \delta)^2 \right)^{\frac{1}{r'}} \\
			\underset{\eqref{eq77} }&{\leq} M (4 \pi)^{\frac{2}{r'}} \delta^{\frac{2}{r'}-1}.
		\end{aligned}
	\end{equation}
Hence, if $r>2$, we have
\begin{equation} \label{eq94}
	\lim_{\delta \searrow 0} \nrm{\zeta^\delta f(\cdot,t)}_{L^{r'}(\bbR^2)}=0, \qquad \lim_{\delta \searrow 0} \nrm{\zeta^\delta f_t(\cdot,t)}_{L^{r'}(\bbR^2)}=0, \qquad \lim_{\delta \searrow 0} \nrm{\nabla(\zeta^\delta f(\cdot,t))}_{L^{r'}(\bbR^2)}=0.
\end{equation}
Now, we investigate the integrability of $w$ and $wu$. As we saw in the Proposition \ref{boundedness of three functions}, we have
\begin{equation} \label{eq95}
		w(x,t) \in L^p_{\textnormal{loc}} (\bbR^2), \qquad
		u(x,t) \in L^q_{\textnormal{loc}}(\bbR^2 ; \bbR^2)
\end{equation}
for all

\begin{equation*}
	1 \leq p<2\mu \qquad \text{and} \qquad \begin{cases}
	q< \frac{2\mu}{1-\mu} & (\frac{2}{3}<\mu<1) \\
	q=\infty & (\mu \geq 1) \\
\end{cases},
\end{equation*}
 where its norm on compact set is controlled independent of $t$. Hence,
 \begin{equation} \label{eq96}
 	wu \in L^r_{\textnormal{loc}}(\bbR^2 ; \bbR^2)
 \end{equation}
for
\begin{equation*}
	r < \begin{cases}
		 \left( \frac{1}{2\mu}+\frac{1-\mu}{2\mu} \right)^{-1} =\frac{2\mu}{2-\mu} & (\frac{2}{3} < \mu <1) \\
		2\mu & (\mu \geq 1) 
	\end{cases}
\end{equation*}
with its local $L^r$-norm uniformly controlled over time. If $\mu \leq 1$, we have $r<\frac{2\mu}{2-\mu}<2$, by which we cannot guarantee the convergence in \eqref{eq94}. This is why we require $\mu >1$ in the assumption of Proposition \ref{weak solution of EE in vorticity form upgraded}. Then, if we set $r=1+\mu$, \eqref{eq94} and \eqref{eq96} are satisfied for all $\mu>1$.

We are ready to show that $w(x,t)$ is a weak solution of \eqref{EEz}. By Proposition \ref{weak solution of EE in vorticity form}, taking a test function $(1-\zeta^\delta) f \chi_{\{t \geq s \}}$ gives
\begin{equation} \label{eq82}
	\int_{\bbR^2} w(\cdot, s) \cdot \left( (1-\zeta^\delta) f(\cdot, s) \right)  + \int_{\bbR^2 \times [s, \infty)} w \left( (1-\zeta^\delta) f_t+ u \cdot \nabla \left((1-\zeta^\delta)f) \right) \right) =0
\end{equation}
for all $s>0$ and $0<\delta<\frac{1}{M}$. By Proposition \ref{initial data for vorticity prop}, $w(\cdot, s) \rightarrow w(\cdot, 0)$ in $L^p_{\textnormal{loc}} (\bbR^2)$. Since $(1-\zeta^\delta)f(\cdot, s)$ converges uniformly as $s \searrow 0$ as a test function, we have
\begin{equation*}
	\lim_{s \searrow 0} \int_{\bbR^2} w(\cdot, s) \left((1-\zeta^\delta)f(\cdot, s)\right) = \int_{\bbR^2} w(\cdot, 0) \left((1-\zeta^\delta) f(\cdot, 0)\right).
\end{equation*}
Note that by the assumption of $\mu>1$, Corollary \ref{boundedness of three functions} implies $u \in L^\infty_{\textnormal{loc}}(\bbR^2)$ and $w \in L^p_{\textnormal{loc}}(\bbR^2)$ with their local norm controlled independent of $t$. Therefore, we have

\begin{equation} \label{eq83}
	\lim_{s \searrow 0} \int_{\bbR^2 \times [s,\infty)} w\left( (1-\zeta^\delta ) f_t + u \cdot \nabla ((1-\zeta^\delta)f) \right)=\int_{\bbR^2 \times [0,\infty)} w\left( (1-\zeta^\delta ) f_t + u \cdot \nabla ((1-\zeta^\delta)f) \right).
\end{equation}

Sending $\delta$ to zero to include origin, we have

\begin{equation} \label{eq84}
	\int_{\bbR^2} \underbrace{w(\cdot, 0)}_{\underset{\text{Prop } \eqref{w is L^p} }{\in} L^r_{\textnormal{loc}}(\bbR^2)} \quad \underbrace{(1-\zeta^\delta)f(\cdot,0)}_{\xrightarrow[\delta \searrow 0 ]{\eqref{eq94} }f(\cdot, 0) \text{ in } L^{r'}_{\textnormal{loc}}(\bbR^2)} \xrightarrow[\delta \searrow 0]{\text{H\"older}} \int_{\bbR^2} w(\cdot, 0) f(\cdot, 0)
\end{equation}

\begin{equation} \label{eq85}
	\begin{aligned}
	& \int_{\bbR^2 \times [0, \infty)} \underbrace{w}_{\in L^r_{\textnormal{loc}}(\bbR^2)} \underbrace{(1-\zeta^\delta) f_t}_{\substack{ \xrightarrow[\delta \searrow 0]{\eqref{eq94} } f_t \text{ in } L^{r'}_{\textnormal{loc}}(\bbR^2) \\ \text{for uniform rate over time}}} + \underbrace{wu}_{\underset{\eqref{eq96} }{\in} L^r_{\textnormal{loc}}(\bbR^2)} \cdot \underbrace{\nabla((1-\zeta^\delta)f)}_{\substack{ \xrightarrow[\delta \searrow 0]{\eqref{eq94} } \nabla f \text{ in } L^{r'}_{\textnormal{loc}}(\bbR^2) \\ \text{for uniform rate over time}}} \\ 
	& \xrightarrow[\delta \searrow 0]{\text{H\"older}} \int_{\bbR^2 \times [0, \infty)} w f_t + wu \cdot \nabla f.
	\end{aligned}
\end{equation}
Therefore, sending $s$ to 0 and $\delta$ to 0 successively in \eqref{eq82} yields

\begin{equation*}
		\int_{\bbR^2} w(x,0) f(x,0) \, dx + \int_{\bbR^2 \times [0, \infty)}  wf_t + w(u \cdot \nabla f) \, dxdt=0.
\end{equation*}
\end{proof}

To prove $(w, u, \psi)$ is an weak solution for \eqref{Euler eq} and \eqref{Euler eq-vor}, it only remains to show that:

\begin{proposition} \label{second equality of EE-vor}
	\begin{equation} \label{eq97}
		\nabla_x \cdot u =0, \qquad \nabla_x^\perp \psi = u, \qquad \Delta_x \psi = w
	\end{equation}
	in distribution sense.
\end{proposition}

\begin{proof}
	We already show in Proposition \ref{weak solution of EE in vorticity form} that $\nabla_x^\perp \psi = u, \, \Delta_x \psi = w$ in the sense of distribution on $\rmz$. Also, since $u$ is defined by $\nabla_x^\perp \psi$, it is trivial that $u$ is weakly divergence free on $\rmz$. To include origin, by the same way in the proof of Proposition \ref{weak solution of EE in velocity form upgraded} and \ref{weak solution of EE in vorticity form upgraded}, it is suffices to show
	\begin{equation*}
		\psi, \nabla_x \psi \in L^a_{\textnormal{loc}}(\bbR^2), \qquad u, w \in L^b_{\textnormal{loc}}(\bbR^2)
	\end{equation*}
	for $a>2, b>1$, which is clear by Corollary \ref{boundedness of three functions} with $\mu>\frac{2}{3}$.
\end{proof}

\bigskip

\subsection{Initial data in original coordinate and main theorem}
\label{subsec: Initial data in original coordinate and main theorem}

In summary, we construct self-similar weak solution of Euler equation through following process

\begin{equation} \label{transform process}
	\Omega(\phi) \underset{\eqref{main thm in special coordinate}}{\Rightarrow} \br{\psi}(\beta, \phi)=G(\Omega(\phi)) \underset{\eqref{w formula}}{\Rightarrow} w(\beta, \phi)=\left( \psi_\varphi \right)^{-\frac{1}{2\mu}} \Omega(\phi) \underset{\eqref{pull back notation}}{\Rightarrow} \begin{cases}
		\tld{w}(z) \\
		\tld{u}(z) \\
		\tld{\psi}(z)
	\end{cases} \underset{\substack{\text{Def } \ref{self-similarity} \text{ for} \\ \lambda = t^\mu, a=\mu^{-1}}}{\Rightarrow} \begin{cases}
		w(x,t) \\
		u(x,t) \\
		\psi(x,t)
	\end{cases},
\end{equation}
and we shows that $w,u,\psi$ indeed a weak solution for \eqref{Euler eq}, \eqref{Euler eq-vor} in Proposition \ref{weak solution of EE in velocity form upgraded}, \ref{weak solution of EE in vorticity form upgraded}, \ref{second equality of EE-vor}.
As we used implicit function theorem, we cannot choose $\Omega(\phi)$ arbitrary in $\Yz$, but sufficiently close to trivial solution, i.e. in $\ball{\Yz}{\Omega_0}{\varepsilon_N^*}$ where $N$ should be chosen depending on $\mu$ as in \eqref{range of N}. However, we want to impose restriction on initial data on physical variable $w_0(x)\coloneqq w(x,0)$ directly rather than on $\Omega(\phi)$, because limitation on $\Omega(\phi)$ become ambiguous after transformation process \eqref{transform process}. For this aim, we prove local surjectivity of initial data map which corresponds $\Omega(\phi)$ to \eqref{initial data for vorticity prop}.

\begin{proposition} \label{local surjectivity}
	The map which corresponds to angular variety of initial data \eqref{initial data for vorticity} in physical coordinate from $\Omega \in \ball{\Yz}{\Omega_0}{\varepsilon^*_N}$
	\begin{equation} \label{initial data map}
		\Omega(\theta) \mapsto \left( \frac{\dbbeta \br{\psi}^{(\Omega)}(0, \theta)}{-\mu \dbvarphi \br{\psi}^{(\Omega)} (0, \theta)}\right)^{\frac{1}{2\mu}} \Omega(\theta) \eqqcolon h^{(\Omega)} (\theta)
	\end{equation}
	is a local $C^1$ diffeomorphism to $\Yz$ at $\Omega=\Omega_0$. In particular, this implies that, when we denote $h^{(\Omega_0)}(\theta)\equiv \mu^{-\frac{1}{2\mu}} \left( 2-\frac{1}{\mu}\right)$ by $\overset{\circ}{w_0}$, there exists $\varepsilon^{**}_N>0$ and neighborhood $U$ of $\Omega_0$ such that if $\nrm{h(\theta)-\overset{\circ}{w_0}}_{\Yz}< \varepsilon^{**}_N$, there exists unique $\Omega(\theta) \in U \subset \ball{\Yz}{\Omega_0}{\varepsilon^*_N}$ so that
	\begin{equation*}
		\left( \frac{\dbbeta \br{\psi}^{(\Omega)}(0,\theta)}{-\mu \dbvarphi \br{\psi}^{(\Omega)} (0,\theta)}\right)^{\frac{1}{2\mu}} \Omega(\theta) = h (\theta).
	\end{equation*}
	Also, if $h(\theta) \in L^p(\bbT)$, corresponding $\Omega(\theta) $ also belongs to $L^p (\bbT)$.
\end{proposition}

\begin{proof}
	First, we check that the map is well-defined, i.e.,
	\begin{equation*}
		\left( \frac{\dbbeta \br{\psi}^{(\Omega)}(\beta, \phi)}{-\mu \dbvarphi \br{\psi}^{(\Omega)} (\beta, \phi)}\right)^{\frac{1}{2\mu}} \Omega(\phi) \Bigg\vert_{(\beta, \phi)=(0, \theta)}  \in \Yz = \calA^{-0.5} \left( \cn \right).
	\end{equation*}
	Using same method as in Proposition \ref{continuity of diff bar operator}, \ref{continuity of diff beta bar inv}, and the fact that $\calA^{0.5}\left(\Wnp\right)$ is an algebra, we have
	\begin{equation} \label{eq86}
		\left( \frac{\dbbeta \br{\psi}^{(\Omega)}(\beta, \phi)}{-\mu \dbvarphi \br{\psi}^{(\Omega)} (\beta, \phi)}\right)^{\frac{1}{2\mu}}  \in \calA^{0.5}\left( \Wnp \right).
	\end{equation}
	As a restriction map
	\begin{equation*}
		\textrm{Res}^{(n)}: \Wnp \rightarrow \cn
	\end{equation*}
	defined by $\text{Res}^{(n)}\left(f^{(n)} (\beta) \right) = f^{(n)}(0)$ is uniformly bounded ($\because \abs{f(0)} \leq \nrm{f(\beta)}_{C_b} \leq \nrm{f(\beta)}_{\Wnp})$, the induced restriction
	\begin{equation*}
		\text{Res}: \calA^{0.5}\left( \Wnp \right) \rightarrow \calA^{0.5} \left( \cn \right)
	\end{equation*}
	defined by
	\begin{equation} \label{restriction map}
		\text{Res} \left( f(\beta, \phi) \right) =f(0, \phi)
	\end{equation}
	is continuous. Then, by Proposition \ref{A^s algebra}, we have
	\begin{equation*}
		\left( \frac{\dbbeta \br{\psi}^{(\Omega)}(\beta, \phi)}{-\mu \dbvarphi \br{\psi}^{(\Omega)} (\beta, \phi)}\right)^{\frac{1}{2\mu}} \Omega(\phi) \Bigg\vert_{(\beta, \phi)=(0, \theta)}  
		=
		\underbrace{\left( \frac{\dbbeta \br{\psi}^{(\Omega)}(\beta, \phi)}{-\mu \dbvarphi \br{\psi}^{(\Omega)} (\beta, \phi)}\right)^{\frac{1}{2\mu}} \res}_{\in \calA^{0.5} \left( \cn \right) ~~(\because \text{ restriction)}} \underbrace{\Omega(\theta)}_{\in \calA^{-0.5} \left( \cn \right)} \in \calA^{-0.5} \left( \cn \right)	.
	\end{equation*}
	This shows the map \eqref{initial data map} is well defined map from $\Yz$ to $\Yz$.
	
	To show \eqref{initial data map} is a local diffeomorphism at $\Omega = \Omega_0$, we utilize inverse function theorem. First, we show \eqref{initial data map} is $C^1$ in $\ball{\Yz}{\Omega_0}{\varepsilon^*_N}$. By chain rule, we have
	\begin{equation} \label{eq87}
		\begin{aligned}
			&\frac{d}{d\Omega} \left( \left( \frac{\dbbeta \br{\psi}^{(\Omega)}(\beta, \phi)}{-\mu \dbvarphi \br{\psi}^{(\Omega)} (\beta, \phi)}\right)^{\frac{1}{2\mu}} \Omega(\phi) \Bigg\vert_{(\beta, \phi)=(0, \theta)} \right) \\
			&= \left( \frac{\dbbeta \br{\psi}^{(\Omega)}(\beta, \phi)}{-\mu \dbvarphi \br{\psi}^{(\Omega)} (\beta, \phi)}\right)^{\frac{1}{2\mu}}\Bigg\vert_{(\beta, \phi)=(0, \theta)} \cdot id \\
			& \qquad + \Omega(\theta) \cdot \left( \frac{d}{d \br{\psi}} \left( \frac{\dbbeta \br{\psi}^{(\Omega)}(\beta, \phi)}{-\mu \dbvarphi \br{\psi}^{(\Omega)} (\beta, \phi)}\right)^{\frac{1}{2\mu}} \res \circ \frac{\quad d\br{\psi}^{(\Omega)}}{d \Omega} \right) \\
			&= \underbrace{\left(\frac{\dbbeta \br{\psi}^{(\Omega)}(\beta, \phi)}{-\mu \dbvarphi \br{\psi}^{(\Omega)} (\beta, \phi)}\right)^{\frac{1}{2\mu}}\Bigg\vert_{(\beta, \phi)=(0, \theta)}}_{\eqqcolon I_1(\Omega)} \cdot id \\
			& \qquad -\frac{1}{2\mu^2}\underbrace{\Omega(\theta) \cdot \left( \left( \frac{\dbbeta \br{\psi}^{(\Omega)}(\beta, \phi)}{-\mu \dbvarphi \br{\psi}^{(\Omega)} (\beta, \phi)}\right)^{\frac{1}{2\mu}-1} \frac{1}{\dbvarphi \br{\psi}^{(\Omega)}} \right) \res}_{\eqqcolon I_2(\Omega)} \cdot \underbrace{\left( \dbbeta \circ \frac{\quad \rd \br{\psi}^{(\Omega)}}{\rd \Omega} \right) \res}_{\eqqcolon J_1} \\
			&\qquad \underbrace{+\frac{1}{2\mu^2} \Omega(\theta) \cdot \left( \left( \frac{\dbbeta \br{\psi}^{(\Omega)}(\beta, \phi)}{-\mu \dbvarphi \br{\psi}^{(\Omega)} (\beta, \phi)}\right)^{\frac{1}{2\mu}-1} \frac{\dbbeta \br{\psi}^{(\Omega)}}{\left(\dbvarphi \br{\psi}^{(\Omega)}\right)^2} \right) \res}_{\eqqcolon I_3(\Omega)} \cdot\underbrace{\left( \dbvarphi \circ \frac{\quad \rd \br{\psi}^{(\Omega)}}{\rd \Omega} \right) \res}_{\eqqcolon J_2}
		\end{aligned}
	\end{equation}
	Note that
	\begin{equation*}
		I_1: \ball{\Yz}{\Omega_0}{\varepsilon^*_N} \rightarrow \Yz
	\end{equation*}
	is the composition of three map
	\begin{equation*}
		\text{Res} \circ \left( \frac{\dbbeta}{-\mu \dbvarphi} \right)^{\frac{1}{2\mu}} \circ G,
	\end{equation*}
	where
	\begin{equation*}
		\text{Res}: \Wp \rightarrow \Yz
	\end{equation*}
	is the restriction map to $\beta=0$ as defined in \eqref{restriction map},
	\begin{equation*}
		\left( \frac{\dbbeta}{-\mu \dbvarphi} \right)^{\frac{1}{2\mu}}: \ball{\Xz}{\br{\psi}}{\delta} \rightarrow \Wp
	\end{equation*}
	is a map defined by
	\begin{equation*}
		\left( \frac{\dbbeta}{-\mu \dbvarphi} \right)^{\frac{1}{2\mu}}\left(\br{\psi}\right)= \left( \frac{\dbbeta \br{\psi}}{-\mu \dbvarphi \br{\psi}} \right)^{\frac{1}{2\mu}},
	\end{equation*}
	and
	\begin{equation*}
		G: \ball{\Yz}{\Omega_0}{\varepsilon^*_N} \rightarrow \ball{\Xz}{\br{\psi}_0}{\delta}
	\end{equation*}
	is a map in Theorem \ref{main thm in special coordinate}. Since these three maps are continuous, $I_1$ is continuous.
	
	Now we investigate $I_2$ term. Note that $I_2$ is a product of two map; $\Omega \rightarrow \Omega$, and
	\begin{equation} \label{eq88}
		\Omega \rightarrow \left( \left( \frac{\dbbeta \br{\psi}^{(\Omega)}(\beta, \phi)}{-\mu \dbvarphi \br{\psi}^{(\Omega)} (\beta, \phi)}\right)^{\frac{1}{2\mu}-1} \frac{\dbbeta \br{\psi}^{(\Omega)}}{\left(\dbvarphi \br{\psi}^{(\Omega)}\right)^2} \right) \res.
	\end{equation}
	Since \eqref{eq88} is continuous for the same reason as $I_1$, a map $I_2$ is continuous as a product of continuous map. $I_3$ is also continuous by same argument as $I_2$.

	Since the map $\Omega \rightarrow \br{\psi}^{(\Omega)}$ is $C^1$ on $\ball{\Yz}{\Omega_0}{\varepsilon_N^{*}}$ by Theorem \ref{main thm in special coordinate}, and
	\begin{align*}
		\dbbeta &: \Xz \rightarrow \Wp ,\\
		\dbvarphi &: \Xz \rightarrow \Wp ,\\
		\text{Res} &: \Wp \rightarrow \Yz
	\end{align*} 
	is continuous, $J_1, J_2$ is also continuous.
	Combining these result on $I_1, I_2, I_3, J_1, J_2$, we conclude the map \eqref{initial data map} is $C^1$.
	
	Next, we show
	\begin{equation} \label{eq89}
		\frac{d}{d\Omega} \left( \left( \frac{\dbbeta \br{\psi}^{(\Omega)}(\beta, \phi)}{-\mu \dbvarphi \br{\psi}^{(\Omega)} (\beta, \phi)}\right)^{\frac{1}{2\mu}} \Omega(\phi) \Bigg\vert_{(\beta, \phi)=(0, \theta)} \right) (\Omega_0)
	\end{equation}
	is an isomorphism. Let $\Omega = \Omega_0$ in \eqref{eq87} and using the fact that
	\begin{equation*}
		\dbbeta \trisol \equiv 1, \qquad \dbvarphi \trisol \equiv 1,
	\end{equation*}
	we have
	\begin{align*}
		& \frac{d}{d\Omega} \left( \left( \frac{\dbbeta \br{\psi}^{(\Omega)}(\beta, \phi)}{-\mu \dbvarphi \br{\psi}^{(\Omega)} (\beta, \phi)}\right)^{\frac{1}{2\mu}} \Omega(\phi) \Bigg\vert_{(\beta, \phi)=(0, \theta)} \right) (\Omega_0) \\
		& = \mu^{-\frac{1}{2\mu}} \cdot id - \left( 1-\frac{1}{2\mu} \right) \mu^{-\frac{1}{2\mu}-1} \left( \left( \dbbeta+\dbvarphi \right) \circ \frac{\quad \rd \br{\psi}^{(\Omega)}}{\rd \Omega} \right) \res \\
		\underset{\eqref{diff beta bar}, \eqref{diff varphi bar}}&{=} \mu^{-\frac{1}{2\mu}} \cdot id - \left( 1-\frac{1}{2\mu} \right) \mu^{-\frac{1}{2\mu}-1} \underbrace{\left( \left( \beta \rd_\phi \right) \circ \frac{\quad \rd \br{\psi}^{(\Omega)}}{\rd \Omega} \right) \res}_{=0 \text{ due to multiplication } \beta \text{ at } \beta=0} \\
		&= \mu^{-\frac{1}{2\mu}} \cdot id.
	\end{align*}
	Hence, Fr\'echet derivative \eqref{eq89} is an isomorphism. As \eqref{initial data map} is $C^1$ in $\ball{\Yz}{\Omega_0}{\varepsilon^*_N}$ and its Fr\'echet derivative at $\Omega_0$ is an isomorphism, it is a local diffemorphism at $\Omega=\Omega_0$ by inverse function theorem.
	
	The last statement is clear since
	\begin{equation*}
		\left( \frac{\dbbeta \br{\psi}^{(\Omega)}(0,\theta)}{-\mu \dbvarphi \br{\psi}^{(\Omega)} (0,\theta)}\right)^{\frac{1}{2\mu}}
	\end{equation*}
	is bounded by two positive number as can be seen in \eqref{boundedness of bar derivative}.
	\end{proof}

We gather all ingredients to prove main theorem.

\begin{theorem} \label{main thm}
	Let $\mu > \frac{2}{3}, \,1\leq p < 2\mu$ be given. For all
	\begin{equation} \label{range of N again}
		N > (2\mu -1)\left(397+\frac{1090}{\mu} + \frac{1264}{\mu^2}+\frac{999}{\mu^3}+\frac{42}{\mu^4} \right),
	\end{equation}	
	 there exists a constant $C=C(N)$ such that for all $g \in \calA^{-0.5}(\bbT) \cap L^p (\bbT) $ satisfying
	\begin{enumerate}
		\item Periodicity: $g(\theta)=g\left(\theta+ \frac{2 \pi}{N} \right)$,
		\item Small $\calA^{-0.5}$-seminorm relative to $\wht{g}(0)$: $\sum_{n \in \bbZ \backslash \set{0}} \brk{n}^{-0.5} \abs{\wht{g}(n)} < C \abs{\wht{g}(0)}$,
	\end{enumerate}
	
	there exist 
	\begin{equation} \label{eq107}
		w \in L^\infty([0,\infty); L^p_{\textnormal{loc}}(\bbR^2)), \quad u \in \begin{cases}
			L^\infty([0,\infty); L^{\frac{2-\mu}{1-\mu}}_{\textnormal{loc}}(\bbR^2; \bbR^2)) & \left( \frac{2}{3} < \mu <1 \right)\\
			L^\infty([0,\infty); L^\infty_{\textnormal{loc}}(\bbR^2; \bbR^2)) & (\mu \geq 1)\\
		\end{cases},\quad  \psi \in L^\infty([0,\infty); C^1(\bbR^2 \backslash \set{0}))
	\end{equation}
	which satisfy
	\begin{enumerate}
		\item [\circled{1}] $u$ is a weak solution of \eqref{Euler eq} with with initial velocity given by \eqref{initial data for velocity prop}.
		\item [\circled{2}] If $\mu>1$, $w$ is a weak solution of \eqref{Euler eq-vor} with initial vorticity $w_0(x)=\abs{x}^{-\frac{1}{\mu}}g(\theta)$.
		\item [\circled{3}] This solution is self-similar and shows spiral roll-up in the sense of \eqref{spiral roll-up}.
	\end{enumerate}
\end{theorem}

\begin{proof}
	Suppose $g(\theta) = g \left( \theta+\frac{2 \pi}{N} \right)$ and define $g_0(\theta)$ by
	\begin{equation*}
		g_0(\theta) = \frac{\overset{\circ}{w_0}}{\wht{g}(0)} g(\theta),
	\end{equation*}
	where $\overset{\circ}{w_0}$ is as in Proposition \eqref{local surjectivity} so that $\abs{\wht{g_0}(0)} = \overset{\circ}{w_0}$. Then,
	
	\begin{align*}
		\nrm{g_0(\theta) - \overset{\circ}{w_0} }_{\calA^{-0.5} \left( \cn \right)} &= \underbrace{\abs{\wht{g_0}(0)-\overset{\circ}{w_0}}}_{=0}+\sum_{n \in N\bbZ \backslash \set{0}} \brk{n}^{-0.5} \abs{\wht{g_0}(n)} \\
		& = \frac{\overset{\circ}{w_0}}{\abs{\wht{g}(0)}} \sum_{n \in N\bbZ \backslash \set{0}} \brk{n}^{-0.5} \abs{\wht{g}(n)}.
	\end{align*}
	By Proposition \ref{local surjectivity}, if $\frac{\overset{\circ}{w_0}}{\abs{\wht{g}(0)}} \sum_{n \in N\bbZ \backslash \set{0}} \brk{n}^{-0.5} \abs{\wht{g}(n)}<\varepsilon^{**}_N$,\footnote{$\varepsilon^{**}_N$ is the constant in Proposition \ref{local surjectivity}.} there exists $\Omega \in \ball{\Yz}{\Omega_0}{\varepsilon^{*}_N}$\footnote{$\varepsilon^{*}_N$ is the constant in Theorem \ref{main thm in special coordinate}.} such that
	\begin{equation*}
		\left( \frac{\dbbeta \br{\psi}^{(\Omega)}(0,\theta)}{-\mu \dbvarphi \br{\psi}^{(\Omega)} (0,\theta)}\right)^{\frac{1}{2\mu}} \Omega(\theta) = g_0 (\theta)
	\end{equation*}
	for $\br{\psi}^{(\Omega)} = G(\Omega)$ (So the constant $C=C(N)$ in the Theorem \ref{main thm} is actually $\varepsilon^{**}_N / \overset{\circ}{w_0}$). Then, we can construct self-similar weak solution of Euler equation with initial vorticity $w(x,t)=g_0(\theta)\abs{x}^{-\frac{1}{\mu}}$ through $\Omega$ by the process \eqref{transform process}. We already saw that \eqref{eq107} follows from Proposition \ref{boundedness of three functions}, \circled{1} follows from Proposition \ref{weak solution of EE in velocity form upgraded}, \circled{2} follows from Proposition \ref{weak solution of EE in vorticity form upgraded}, and \circled{3} follows from Section \ref{subsec: The spiral roll-up phenomenon}.
	
	Although the angular variety of initial data was changed from $g(\theta)$ to $g_0(\theta)$, we easily solve this issue by time scaling symmetry of Euler equation. We use the fact that if $w(x,t)$ is a solution of Euler equation, so is $\lambda w (x, \lambda t)$ for all $\lambda>0$. If $\wht{g}(0)>0$, $\lambda=\wht{g}(0)  / \overset{\circ}{w_0}$ gives the desired solution with initial data $g(0)$. If $\wht{g}(0) < 0$, it suffices to choose $C$ in \eqref{choice of C} as negative constant and repeat same process all above.
\end{proof}

\end{document}